\documentclass[11pt]{amsart}
\usepackage{amsmath,amsfonts,amsthm,mathrsfs}
\usepackage{amssymb}

\usepackage[unicode,bookmarks]{hyperref}
\usepackage[dvipsnames]{xcolor}
\hypersetup{colorlinks=true,citecolor=NavyBlue,linkcolor=Brown,urlcolor=Orange}

\usepackage{tikz-cd,adjustbox}
\usepackage[arrow,curve,matrix]{xy}

\usepackage[margin=1in]{geometry}

\newif\ifdraft
\draftfalse

\counterwithin{equation}{subsection}
\counterwithout{subsection}{section}
\counterwithin{figure}{subsection}

\newtheorem{theorem}[equation]{Theorem}
\newtheorem*{theorem*}{Theorem}
\newtheorem{lemma}[equation]{Lemma}
\newtheorem*{lemma*}{Lemma}
\newtheorem{corollary}[equation]{Corollary}
\newtheorem{proposition}[equation]{Proposition}
\newtheorem*{proposition*}{Proposition}
\newtheorem{conjecture}[equation]{Conjecture}

\theoremstyle{definition}
\newtheorem{definition}[equation]{Definition}
\newtheorem*{definition*}{Definition}
\newtheorem{remark}[equation]{Remark}
\newtheorem*{remark*}{Remark}

\newtheorem{example}[equation]{Example}
\newtheorem*{example*}{Example}
\newtheorem*{problem*}{Problem}

\theoremstyle{plain}

\newcounter{intro}

\newtheorem{intro-conjecture}[intro]{Conjecture}
\newtheorem{intro-corollary}[intro]{Corollary}
\newtheorem{intro-theorem}[intro]{Theorem}
\newtheorem{intro-proposition}[intro]{Proposition}

\newcommand{\theoremref}[1]{\hyperref[#1]{Theorem~\ref*{#1}}}
\newcommand{\lemmaref}[1]{\hyperref[#1]{Lemma~\ref*{#1}}}
\newcommand{\definitionref}[1]{\hyperref[#1]{Definition~\ref*{#1}}}
\newcommand{\propositionref}[1]{\hyperref[#1]{Proposition~\ref*{#1}}}
\newcommand{\conjectureref}[1]{\hyperref[#1]{Conjecture~\ref*{#1}}}
\newcommand{\corollaryref}[1]{\hyperref[#1]{Corollary~\ref*{#1}}}
\newcommand{\exampleref}[1]{\hyperref[#1]{Example~\ref*{#1}}}

\newcommand{\A}{\mathbb{A}}
\newcommand{\ZZ}{\mathbb{Z}}
\newcommand{\Q}{\mathbb{Q}}
\newcommand{\C}{\mathbb{C}}
\renewcommand{\O}{\mathcal{O}}
\renewcommand{\H}{\mathcal{H}}
\DeclareMathOperator{\Pic}{Pic}

\DeclareMathOperator{\coker}{coker}
\DeclareMathOperator{\im}{im}
\DeclareMathOperator{\gr}{gr}

\DeclareMathOperator{\id}{id}

\DeclareMathOperator{\sing}{sing}

\DeclareMathOperator{\mot}{mot}

\DeclareMathOperator{\CH}{CH}
\DeclareMathOperator{\Cl}{Cl}
\DeclareMathOperator{\cdh}{cdh}

\DeclareMathOperator{\DB}{\underline{\Omega}}

\DeclareMathOperator{\HH}{\mathbb{H}}

\DeclareMathOperator{\cl}{cl}
\DeclareMathOperator{\ord}{ord}
\DeclareMathOperator{\divisor}{div}
\DeclareMathOperator{\Alb}{Alb}

\makeatletter
\let\old@caption\caption
\renewcommand*{\caption}[1]{%
	\setcounter{figure}{\value{equation}}%
	\stepcounter{equation}%
	\old@caption{#1}\relax%
}
\makeatother

\begin{document}

\title{Negative $K$-Theory and Hodge Theory}

\author{Andrew Burke}
\address{Department of Mathematics, Harvard University, 1 Oxford Street, Cambridge, MA 02138, USA}
\email{aburke@math.harvard.edu}

\author{Mihnea Popa}
\address{Department of Mathematics, Harvard University, 1 Oxford Street, Cambridge, MA 02138, USA}
\email{mpopa@math.harvard.edu}

\author{Wanchun Shen}
\address{Department of Mathematics, Harvard University, 1 Oxford Street, Cambridge, MA 02138, USA}
\email{wshen@math.harvard.edu}

\thanks{The authors were partially supported by the NSF grant DMS-2401498.}


\date{\today}

\begin{abstract}
We study the negative $K$-groups of complex varieties from a mixed Hodge-theoretic perspective, making use of the theory of higher singularities, Chow groups, and the Minimal Model Program.
\end{abstract}

\maketitle

\makeatletter
\newcommand\@dotsep{4.5}
\def\@tocline#1#2#3#4#5#6#7{\relax
  \ifnum #1>\c@tocdepth 
  \else
    \par \addpenalty\@secpenalty\addvspace{#2}%
    \begingroup \hyphenpenalty\@M
    \@ifempty{#4}{%
      \@tempdima\csname r@tocindent\number#1\endcsname\relax
    }{%
      \@tempdima#4\relax
    }%
    \parindent\z@ \leftskip#3\relax
    \advance\leftskip\@tempdima\relax
    \rightskip\@pnumwidth plus1em \parfillskip-\@pnumwidth
    #5\leavevmode\hskip-\@tempdima #6\relax
    \leaders\hbox{$\m@th
      \mkern \@dotsep mu\hbox{.}\mkern \@dotsep mu$}\hfill
    \hbox to\@pnumwidth{\@tocpagenum{#7}}\par
    \nobreak
    \endgroup
  \fi}
\def\l@section{\@tocline{1}{0pt}{1pc}{}{\bfseries}}
\def\l@subsection{\@tocline{2}{0pt}{25pt}{5pc}{}}
\makeatother

\tableofcontents

\section{Introduction}
An algebraic variety $X$ is endowed with a series of abelian groups $K_j (X)$, for $j \in \ZZ$, called the $K$-groups of $X$. This theory was initially developed in various stages, but now exists in a unified form for all $j$; for a comprehensive introduction, see the book \cite{k-book}. In this paper we are concerned with the case $j < 0$ (and at times $j =0$), where $K$-groups also have a more elementary definition due to Bass \cite{Bass}, obstructing a lifting problem for classes in certain $K_0$-groups of vector bundles, and recalled in \S\ref{scn:K-intro}. A homotopy invariance result attributed to Grothendieck, see \cite[Ch.XII, \S3]{Bass} or \cite[Theorem II.7.8]{k-book}, implies that if $X$ is smooth, then $K_j (X) = 0$ for all $j < 0$. Therefore negative $K$-theory is a, rather mysterious, invariant of singularities.
The door towards both a better general understanding and the proof of various predictions, such as Weibel's $K$-dimension conjecture, was opened by the important works \cite{k weibel} and \cite{bass nk} (and earlier work of the authors of these papers referenced therein), where a strong connection with cdh-theoretic methods is established.

The purpose of our paper is to formulate, and in some cases prove, conjectures that provide systematic Hodge-theoretic or singularity-theoretic reasons for the vanishing (or torsion) behavior of negative $K$-theory in the complex setting. More generally, for certain $K$-groups we propose concrete formulas that go beyond vanishing. The hypotheses are either in terms of the mixed Hodge structure on singular cohomology, or in terms of the recently developed theory of higher rational singularities, both of which seem to fit the problem ideally.

In what follows $X$  will always be a complex variety of dimension $n$. Recall that Weibel's conjecture, proved in \cite{k weibel} and 
\cite{descent} under more general hypotheses, says that $K_j (X) = 0$ for $j < - n$, while $K_{-n}$ is homotopy invariant and isomorphic to the cdh cohomology $H^n_{\rm cdh} (X, \ZZ)$.
Our results or predictions give similar statements under appropriate conditions, as we go up the tower of $K$-groups up to $K_{-1}$.

Another point of view is to fully understand the negative $K$-groups of varieties of small dimension, as done by Weibel for curves and normal surfaces in the foundational \cite{weibel surfaces}. We start the discussion of our results from this perspective, by providing a comprehensive description of the $K$-groups of projective threefolds with rational singularities, as an illustration of applications of the main results of the paper.

\begin{intro-theorem}\label{thm:threefolds}
Let $X$ be a complex projective threefold. Then

\noindent 
(i)  If $X$ is of klt type, then $K_{-3}(X) = 0$ and 
$$K_{-2} (X) \simeq \coker \big( \bigoplus_j \Pic S_j \rightarrow \bigoplus_k \Pic C_k \big),$$
where $S_j$ and $C_k$ are the surface and curve strata of the exceptional divisor of a strong resolution of singularities $f \colon Y \to X$. 
Moreover, if $X$ has isolated singularities,
$$K_{-1}(X)_\mathbb{Q} \simeq  \gr_2^W H^3(X, \mathbb{Q}) \simeq  \frac{\ker \big( \bigoplus_j \Pic S_j \rightarrow \bigoplus_k \Pic C_k \big)}{\im \big( \Pic Y \rightarrow \bigoplus_j \Pic S_j \big)} \otimes \Q $$
and  there is an exact sequence 
$$0 \to \Pic(X)_{\Q} \to \Cl(X)_{\Q} \to \bigoplus_{x\in X_{\sing}} \Cl(\widehat{\mathscr{O}}_{X,x})_{\Q} \to  K_{-1} (X)_{\Q} \to 0.$$
We have $K_{-1}(X)_\mathbb{Q} = 0 \iff$ all the Hodge structures on the cohomology groups of $X$ are pure $\iff$ the global $\Q$-factoriality defect of $X$ is equal to the sum of its local analytic $\Q$-factoriality defects.

The description and exact sequence for  $K_{-1} (X)$ already hold over $\ZZ$ if $X$ is in addition a local complete intersection.

\noindent
(ii) If $X$ has rational singularities, then
$$K_{-3}(X)_\Q = K_{-2} (X)_\Q = 0,$$
and, if in addition $X$ has isolated singularities, the description of $K_{-1}(X)$ in part (i) holds if we assume Bloch's conjecture on $0$-cycles on smooth projective surfaces.
\end{intro-theorem}

Recall that a variety is of klt type when it supports an effective $\Q$-divisor such that the pair $(X, \Delta)$ is klt. In this case $X$ has rational singularities by a well-known result of Elkik, and moreover the two notions are equivalent when $X$ is Gorenstein. For a more detailed discussion of the result above, as well as similar results for fourfolds (and curves and surfaces), please see \S\ref{scn:small-dim}. We note that the exact sequence for $K_{-1} (X)$ was already obtained in the case of $cA_n$ singularities in \cite{pavic}, and coincides with the description of $K_{-1}$ for normal surfaces given in \cite{weibel surfaces}.

We now proceed with discussing the general conjectures and results of the paper. We will always assume $n \ge 2$, as $K_{-1}$ of a curve is well understood (see \S\ref{scn:small-dim}).

\noindent
{\bf Overview.}
Since the technical statements may obscure the overall picture, we start by loosely reviewing the main ideas of our approach. Details will come next.

\noindent
(1) We relate the vanishing of negative rational homotopy $K$-theory $KH_j (X)_\Q$, and more generally of cdh-motivic cohomology $H^k_{\rm cdh}(X, \Q(\ell))$, to bounds on the weights of the mixed Hodge structure on the singular cohomology of $X$.

\noindent
(2) The weight bounds needed in (1) are implied, due to the main result of \cite{PP1}, by a duality condition $D_m$ in Du Bois theory, analogous to the perfect pairing between forms of degree $k \le m$ and forms of degree $n-k$ on a complex manifold.

\noindent
(3) The $K$-theory groups $K_j (X)$ coincide with the homotopy $K$-theory groups $KH_j (X)$ in a range depending on the $K$-regularity level of $X$. This is in turn determined by the $m$-Du Bois property of the singularities of $X$, thanks to \cite{rosie}.

\noindent
(4) The theory of higher singularities also studies $m$-rational singularities, for $m\ge 0$, refining the classical notion of rational singularities. A key result is that $m$-rational is equivalent to $m$-Du Bois plus condition $D_m$. Thus $m$-rationality is perfectly suited for predicting the vanishing (and other) properties of negative $K$-groups.
In summary:

\begin{itemize}
\item Bounds on the weights on the mixed Hodge structures on singular cohomology groups, or the $D_m$ property, lead to the vanishing of rational $KH$-groups.
\item The $m$-Du Bois property leads to isomorphisms between $K$ and $KH$-groups.
\item The $m$-rationality property leads to the vanishing of rational $K$-groups.
\end{itemize}

Most of the unconditional results we show are in the setting of rational or klt type singularities, and sometimes (pre-)$1$-rational. Besides hyperresolution (or cdh) and mixed Hodge theory techniques, in the most difficult situations we need to study natural restriction maps on Chow groups of codimension $2$ cycles on log resolutions of singularities. This may be of independent interest; our approach is based on techniques from 
the minimal model program (MMP) and on Bloch's conjecture on $0$-cycles.

\smallskip

\noindent
{\bf $KH$-theory.}
One of the main difficulties in the study of $K$-theory is the fact that it is usually not $\A^r$-invariant (or ``homotopy invariant"); see Definition \ref{def:K-regularity}. To rectify this, Weibel \cite{weibel kh} has introduced the homotopy $K$-theory groups $KH_j (X)$, which do have this property, and admit natural comparison maps $K_j (X) \to KH_j (X)$.
The main step of our program is to try to systematically understand the groups $KH_j (X)$ for $j < 0$, and in fact the finer cdh-motivic cohomology, in terms of Hodge theory and Chow groups, to which they are more directly related.

We make the following conjecture, in terms of the  weight filtration on the mixed Hodge structure on the cohomology of $X$. A stronger and more comprehensive conjecture regarding cdh-motivic cohomology appears below as Conjecture \ref{main-cdh-conjecture}.

\begin{intro-conjecture}\label{main-KH-conjecture}
Fix an integer $i > 0$. If  $W_{j-i} H^j(X, \mathbb{Q}) = 0$ for all $j \le 2n - i$, then 
$$KH_{-i}(X)_{\Q}= 0,$$ 
or in other words $KH_{-i}(X)$ is torsion. 
\end{intro-conjecture}

We note that if we considered $i \le 0$ as well in the conjecture, neither the hypotheses nor the conclusion would hold.

\begin{remark*}[\emph{Projective vs. non-projective}]\label{rmk:KH-nonproj}
The conjecture does not require $X$ to be projective. The only solid evidence we have for the non-projective case is the fact that it is implied by the deep conjectural picture of mixed motives for arbitrary varieties. In the isolated singularities case, it can easily be reduced to the projective case via compactification that does not introduce new singularities. The same applies to the other conjectures we state in what follows, but besides isolated singularities, all the results we can currently show are in the projective setting.
\end{remark*}

For example, Conjecture \ref{main-KH-conjecture} predicts that if all the singular cohomology groups of a projective $X$ have pure rational Hodge structure, then \emph{all} negative $KH$-theory is torsion. The most obvious case when this hypothesis holds is that of rational homology manifolds, for instance varieties with quotient singularities, but it can also happen when Poincar\'e duality fails. For example, by \cite[Theorem A]{PP1} and \cite[Theorem E]{PP2} a threefold with isolated rational singularities satisfies this property if and only if the global $\Q$-factoriality defect is equal to the sum of the local analytic $\Q$-factoriality defects. When all Hodge structures are pure, we extend our study to the groups $KH_j (X)$ for $j \ge 0$, by formulating  Conjecture \ref{conj nonnegative KH} in the Appendix, and proving it for $KH_0$ of a surface. 

Note that Conjecture \ref{main-KH-conjecture} fails integrally, already in the case of surfaces with rational singularities, as many examples in \S\ref{scn:small-dim} and \S\ref{scn:misc} show.

Thanks to results in \cite{PP1}, from the point of view of singularity theory, the hypothesis of Conjecture \ref{main-KH-conjecture} holds for varieties with normal pre-$m$-rational singularities, or more generally satisfying condition $D_m$ of \emph{loc. cit.}, with $m =\left[ \frac{n-1-i}{2}\right]$ ($m = -1$ means no assumptions). See \S\ref{scn:higher-sings}, especially Corollary \ref{cor weight restrictions} (and Remark \ref{rmk:Dm-nonproj}). Overall, when replacing the weight assumptions with stronger singularity assumptions, we make the following more complete proposal.

\begin{intro-conjecture}\label{sings-KH-conjecture}
If $X$ is a normal variety with pre-$m$-rational singularities (or more generally satisfying  condition $D_m$) for some $0 \le m \le \frac{(n-2)}{2}$, then 
$$KH_{-i}(X)_{\Q}= 0 \,\,\,\,\,\,{\rm for~all}\,\,\,\,i \ge n - 2m -1$$
and if $m < \frac{(n-2)}{2}$, then in addition
$$KH_{-n + 2m + 2}(X)_{\Q} \simeq \gr^W_{2m +2} H^n (X, \Q).$$ 
\end{intro-conjecture}

We now list the results on $KH$-groups that we can show either unconditionally, or in one case assuming Bloch's conjecture on the Chow groups of $0$-cycles on smooth projective surfaces.

\begin{intro-theorem}\label{thm:main-KH}
Let $X$ be a complex projective variety of dimension $n\ge 2$.

\noindent
(i) Conjecture \ref{main-KH-conjecture} holds for $i = n, n - 1$. In particular, it holds for surfaces.

\noindent
(ii) The first part of Conjecture \ref{sings-KH-conjecture} holds for $m = 0$. That is, if $X$ has rational singularities, then 
$$KH_{-n}(X)_{\Q} = KH_{-n + 1}(X)_{\Q} = 0.$$
If $X$ has isolated klt type singularities, the second part also holds, i.e. if $n \ge 3$, then
$$KH_{-n + 2}(X)_{\Q} \simeq \gr^W_{2} H^n (X, \Q).$$ 
In particular, Conjecture \ref{sings-KH-conjecture} holds for surfaces with rational singularities, and for threefolds with isolated klt type singularities. It holds for all threefolds with isolated rational singularities if we assume Bloch's conjecture on $0$-cycles on  smooth surfaces.

Under the same assumptions, Conjecture \ref{sings-KH-conjecture} holds for $m = 1$ and $i \ge n-2$;  that is, if $n \ge 4$ and $X$ has isolated singularities of klt type and satisfies $D_1$, then we also have $KH_{-n +2}(X)_{\Q} = 0$.
\end{intro-theorem}

All of these results are consequences of more general facts that essentially provide descriptions of $KH_{-n}$, $KH_{-n+1}$ and $KH_{-n +2}$, sometimes even integrally, for varieties with appropriate singularities. See Ch.\ref{scn:results-K-groups} for details. Even more generally, the computations are in fact for the appropriate cdh cohomology groups, solving the corresponding cases of Conjecture \ref{main-cdh-conjecture} below.

\smallskip

\noindent
{\bf $K$-theory.}
Our main interest is in understanding the standard $K$-theory groups in the negative range. 
We would like to make sense of what they measure, or obstruct.

We make predictions analogous to those for $KH$-groups. This time however, they are necessarily based on singularity assumptions as in Conjecture \ref{sings-KH-conjecture}. The weight assumptions of Conjecture \ref{main-KH-conjecture} do not suffice, since in addition one needs to deal with $K$-regularity issues.

\begin{intro-conjecture}\label{main-K-conjecture}
If $X$ has normal pre-$m$-rational singularities, for some $0 \le m \leq (n-2)/2$, and $s = \dim X_{\rm sing}$, then:

\smallskip

\noindent 
(i) We have
$$K_{-i}(X)_{\Q}= 0 \,\,\,\,\,\,{\rm for~all}\,\,\,\,i \ge {\rm max}~\{s, n - 2m -1\}.$$
In particular, if $X$ has $m$-rational singularities,\footnote{We will not discuss this strong higher rationality notion here, but besides implying normal pre-$m$-rational singularities, a result in \cite{MP-lci} in the lci case, and its very definition in \cite{SVV} in general, imply that $s \le n - 2m -2$. This is another remarkable agreement with the picture of $K$-groups. 
Note that this also says that (ii) should hold when $X$ has $m$-rational singularities.}
 then 
$$K_{-i}(X)_{\Q}= 0 \,\,\,\,\,\,{\rm for~all}\,\,\,\,i \ge  n - 2m -1.$$



\smallskip

\noindent
(ii)  If $s \le n - 2m -2$ and $m < \frac{(n-2)}{2}$, then
$$K_{-n + 2m + 2}(X)_{\Q} \simeq \gr^W_{2m +2} H^n (X, \Q).$$ 
\end{intro-conjecture}


Note that $m = 0$ means precisely rational singularities, and in this case the conjecture predicts that $K_{-n}(X)$ and $K_{-n +1}(X)$ are torsion, plus a precise description of $K_{-n+2}(X)_\Q$ when $n \ge 3$. At the other end, rational homology manifolds (for example varieties with quotient singularities), are pre-$m$-rational for all $m$, hence at least for isolated singularities all of their negative $K$-groups should be torsion.

Using the results on $K$-regularity from \cite{rosie} (and upcoming improvements in \cite{BPS}), which provide criteria for when the natural maps $K_j (X) \to KH_j (X)$ are isomorphisms in terms of higher Du Bois singularities (see \S\ref{scn:K-intro}), Conjecture \ref{main-K-conjecture} is in fact a consequence of Conjecture \ref{sings-KH-conjecture}. Along the same lines, using Theorem \ref{thm:main-KH}, we deduce the following cases of Conjecture \ref{main-K-conjecture}.

\begin{intro-theorem}\label{thm:main-K}
Let $X$ be a projective complex variety of dimension $n\ge 2$.

\noindent 
(i) Conjecture \ref{main-K-conjecture}(i) holds for $m = 0$. In other words, if $X$ has rational singularities, then
$$K_{-n}(X)_{\Q} =  K_{-n +1}(X)_{\Q} = 0.$$

\noindent 
(ii) Conjecture \ref{main-K-conjecture}(ii) holds for $m =0$ if $X$ has isolated klt type singularities. In other words, if $n \ge 3$ and $X$ has isolated klt type singularities, then
$$K_{-n +2} (X)_{\Q} \simeq {\rm gr}_2^W H^n (X, \Q).$$ 
Moreover, if $X$ is a threefold, it holds in general if we assume Bloch's conjecture on $0$-cycles on smooth surfaces.

In particular $K_{-n +2} (X)_{\Q} = 0$ when the singularities are in addition pre-$1$-rational (i.e. part of Conjecture \ref{main-K-conjecture}(i) for $m=1$).
\end{intro-theorem}

As a special case, note that Conjecture \ref{main-K-conjecture} always holds for surfaces (only (i) is relevant). This can already be deduced from \cite{weibel surfaces}.

\smallskip

\noindent
{\bf Other known cases.}
Our conjectures on $KH$-groups and $K$-groups hold in full generality for special classes of varieties. These include varieties with quotient singularities, cones over smooth varieties, toric varieties, or certain varieties with short hyperresolution, like secant varieties under sufficiently positive embeddings. See \S\ref{scn:graded}, \S\ref{scn:cones}, \S\ref{scn:toric} and \S\ref{scn:short-hyp}. Much of this follows from various existing results. The case of quotient singularities seems new in full generality however, although partial results can be found in the literature; it is a consequence of a general result we prove about the behavior of cdh cohomology under finite surjective maps, Proposition 
\ref{prop finite map}.

\smallskip

\noindent
{\bf Cdh-motivic cohomology.}
Conjecture \ref{main-KH-conjecture} follows from a more precise prediction regarding the vanishing of \emph{cdh cohomology groups}. These can be defined and studied in a rather general context, see for instance the comprehensive \cite{BEM}. In the case of complex varieties, they coincide
with the motivic cohomology groups of Friedlander and Voevodsky \cite{FV}, which have been studied quite extensively. See \S\ref{scn:cdh} for more details.\footnote{The reason for using the terminology cdh-motivic cohomology, or simply \emph{cdh cohomology}, is that in the singular case a different motivic cohomology theory has been constructed in the recent \cite{motivic cohomology}, which has a stronger connection with $K$-theory, analogous to that of cdh cohomology to $KH$-theory.}

\begin{intro-conjecture}\label{main-cdh-conjecture}
If $X$ is projective and $j\ge 0$ and $k$ are integers, then:

\noindent 
(i)  We have $H^k_{\rm cdh}  (X, \Q(j)) = 0$ if 
$$W_{b-k + 2j} H^b(X, \mathbb{Q}) = 0 \,\,\,\,\,\,{\rm for ~all } \,\,\,\,  b \leq k.$$
More precisely, it should suffice to assume 
\begin{equation}\label{eqn:sharp-cdh}
H^{r,s} \gr_{b-c}^W H^b(X, \mathbb{Q}) = 0 \,\,\,\,\,\,{\rm for ~all } \,\,\,\,  b \leq k, c \geq k-2j, ~{\rm and}~r \leq b+ j - k.
\end{equation}

\noindent 
(ii) If (\ref{eqn:sharp-cdh}) holds except when $b= k$, $c = k-2j$ and $r = j$, then 
$$H^k_{\rm cdh} (X, \Q(j) ) \simeq \gr^W_{2j} H^k(X, \Q).$$

\noindent
(iii) If $X$ has normal pre-$m$-rational singularities (or more generally satisfies $D_m$), then $H^k_{\rm cdh}  (X, \Q(j)) = 0$ whenever we have $k > 2j$, as well as $m \ge j$ or $m \ge n - k + j$.
\end{intro-conjecture}

Part (iii) is a consequence of part (i) and Corollary \ref{cor weight restrictions}. Part (ii) is the reason for the formulas that go beyond vanishing, predicted in Conjectures \ref{sings-KH-conjecture} and \ref{main-K-conjecture}.

This is the fundamental technical conjecture of the paper, which essentially implies everything else. Note that it predicts more vanishing than is needed in Conjecture \ref{main-KH-conjecture}. Indeed, it includes a conjecture of Beilinson and Soul\'e, stating that $H^k_{\rm cdh} (X, \Q(j)) = 0$ for $k < 0$; see Remark \ref{rmk:BS-conj}.
It also progressively improves the so-called motivic Soul\'e-Weibel vanishing under better and better singularity hypotheses, as in:

\begin{remark*}[\emph{Rational singularities (and beyond)}]
When $X$ has rational singularities, the conjecture predicts that $H^k_{\rm cdh}  (X, \Q (0)) = 0$ for $k > 0$, which we show in \S\ref{scn:weight0}, but also that 
$$H^{j + n}_{\rm cdh}  (X, \Q (j)) = 0 \,\,\,\,\,\, {\rm for} \,\,\,\, j < n.$$ 
This is new ``borderline" vanishing, as for varieties with no conditions on the singularities it is known that $H^i_{\rm cdh}  (X, \Q (j)) = 0$
for $i > j +n$; see Lemma \ref{lem motivic soule-weibel}. We will show this for $j =1$, and also for $j=2$ in the case of isolated klt singularities. When $X$ has pre-$1$-rational singularities, it predicts in addition the vanishing of groups of the type 
$H^{j + n -1}_{\rm cdh}  (X, \Q (j))$, and so on.
\end{remark*}

A natural reindexing that places the relevant groups in the fourth quadrant (when $k \ge 0$) makes the intricate statement of Conjecture \ref{main-cdh-conjecture}  quite easy to visualize. Here is a picture for part (iii), in dimension six:

\begin{figure}[h]
\includegraphics[width=5.8in]{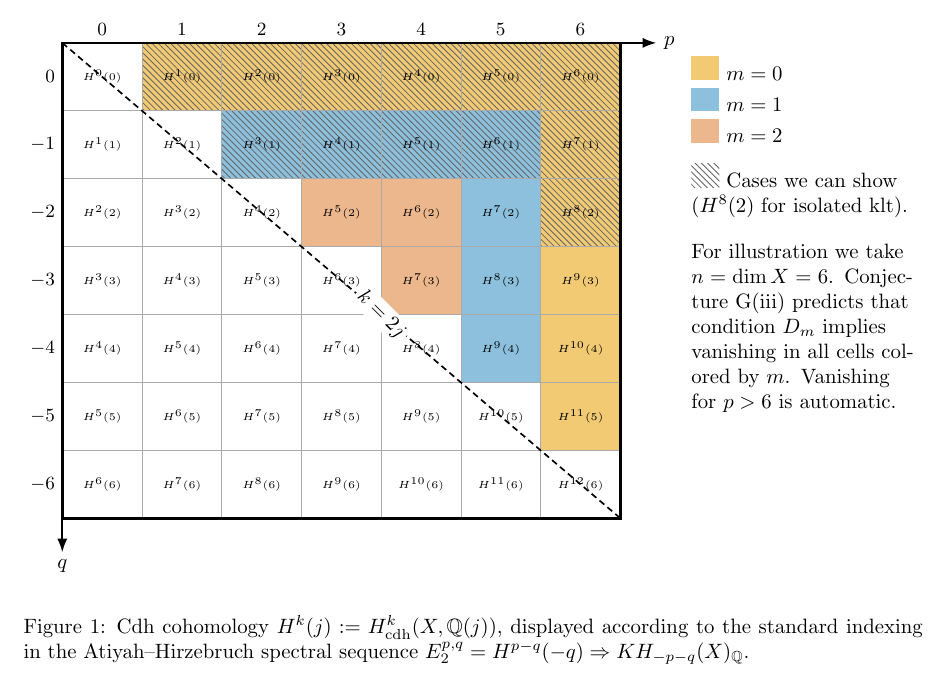}
\end{figure}

The fact that Conjecture \ref{main-cdh-conjecture} implies Conjecture \ref{main-KH-conjecture} is a consequence of the rational $E_2$-degeneration of an Atiyah-Hirzebruch type spectral sequence, which filters $KH_{-i} (X)_{\Q}$ by the cdh-cohomology groups 
$$H^{i + 2j}_{\rm cdh} (X, \Q(j)) \,\,\,\,\,\,{\rm with} \,\,\,\, 0 \le j \le n -i.$$
For complex varieties, the existence and degeneration of this spectral sequence is a consequence of results in \cite{descent properties} and \cite{KP}, respectively. See \S\ref{scn:cdh} for more details.

\begin{remark*}[\emph{Bloch-Beilinson filtration}]
We will see in the Appendix that Conjecture \ref{main-cdh-conjecture} (hence consequently all the conjectures in this paper) is implied by an enhanced form of the conjectural Bloch-Beilinson filtration on (higher) Chow groups, which is in turn predicted by the standard conjectures in the theory of mixed motives. These are of course completely open in general. The refinement deals with the exactness of filtered complexes of higher Chow groups naturally obtained from semisimplicial varieties, especially cubical hyperresolutions; see \S\ref{scn:BB-conj}.
\end{remark*}

As for unconditional results, we have a quite complete picture for weight $0$ and weight $1$ cdh-cohomology groups. We state the case of $\Q$-coefficients for simplicity, but stronger results for $\ZZ$-coefficients appear in the body of the paper.

\begin{intro-theorem}\label{thm:cdh-low-weight}
Let $X$ be a normal complex projective variety of dimension $n$.

\noindent 
(i) For all integers $i$, we have 
$$H^i_{\rm cdh}(X, \Q (0)) \simeq W_0 H^i (X, \Q).$$
In particular, if $X$ has rational singularities, then $H^i_{\rm cdh}(X, \Q (0)) = 0$ for all $i \ge 1$.

\noindent 
(ii)  For all integers $i$, the groups $H^i_{\rm cdh}(X, \Q (1))$ sit in exact sequences whose terms are concretely determined by the mixed Hodge structure of $X$; see Corollary \ref{cor weight one rational}. In particular, if $X$ has rational singularities and $i \ge 3$, then 
$$H_{\cdh}^{i}(X, \mathbb{Q}(1)) \simeq W_2 H^{i}(X, \mathbb{Q}) =  \gr_2^W H^{i}(X, \mathbb{Q}) = H^{1,1} \gr_2^W H^{i}(X, \mathbb{Q}).$$
If moreover $X$ has pre-$1$-rational singularities, then $H^i_{\rm cdh}(X, \Q (1)) = 0$ for all $i \ge 3$.
\end{intro-theorem}


In weight $2$ and higher, Chow groups of cycles of codimension at least two appear in the calculations, and the situation becomes more delicate. 
The connection with Hodge theory is realized via Bloch's conjecture on $0$-cycles, and its generalizations, for various strata of the exceptional locus of a resolution of singularities; for example, thanks to Shokurov's rational connectedness conjecture, established in \cite{hm}, this is known for the exceptional surface components of any log resolution of a klt type singularity. Together with other arguments relying on the Minimal Model Program, this allows us to obtain the following statement about $H^{n+2}_{\rm cdh}(X, \Q (2))$. This is by some margin the most difficult result of the paper, and the main step towards the results on $KH_{-n+2}(X)$ and  $K_{-n+2}(X)$ stated earlier.

\begin{intro-theorem}\label{thm:weight2}
(i) If $X$ is a projective threefold with isolated klt type singularities, then 
$$H_{\rm cdh}^5 (X, \Q (2)) = 0.$$
If in addition $X$ has local complete intersection singularities, then 
$$H^5_{\rm cdh} (X, \ZZ (2)) = 0.$$
The same result holds if we replace klt type by rational singularities, if we assume Bloch's conjecture on $0$-cycles on smooth surfaces.

\noindent
(ii) If $X$ is projective of dimension $n\ge 4$, with isolated klt type singularities, then
$$H_{\rm cdh}^{n+2} (X, \Q (2)) = 0.$$
\end{intro-theorem}

Other weight $2$ cdh-cohomology groups are still mysterious. To treat those, and higher weight groups, we are forced to enter the unknown territory part of the Bloch-Beilinson conjectures.

\smallskip

\noindent
{\bf Quasi-projective varieties with isolated singularities.}
Based on local-to-global principles relying on localization sequences and the Gersten resolution, as described for instance in \cite{weibel isolated}, 
the theorems about projective varieties stated in this Introduction can be extended without much effort to the setting of quasi-projective varieties with isolated singularities, or to the case of local rings of germs of such singularities. We explain this in \S\ref{scn:quasiproj-isolated}.

\smallskip

To conclude, we mention that in the forthcoming \cite{BPS} we will address similar questions for the motivic cohomology groups of complex varieties, as defined in \cite{motivic cohomology}. Besides the connection with mixed Hodge theory and higher singularities, we will show that results in \cite{weibel surfaces}, \cite{pavic} or this paper, relating $K$-theory to factoriality in small dimension, are examples of general behavior for  special motivic cohomology groups.

It is also natural to ask whether some of the conjectures and results in this paper have natural versions in positive characteristic, in view of the fact that a compelling analogue of the theory of higher singularities is currently taking shape \cite{KW1}, \cite{KW2}.

\noindent
{\bf Acknowledgements.}
We would like to thank Jungkai Chen, Brad Dirks, Elden Elmanto, Christopher Hacon, Christian Haesemeyer, Mircea Musta\c t\u a, Scott Nollet, Sung Gi Park, Stefan Schreieder, Vasudevan Srinivas, Anh Duc Vo, Claire Voisin, Charles Weibel, and Kai Xu for various comments and suggestions. 
Fanjun Meng pointed out an inaccuracy related to the proof of Theorem \ref{thm mmp}, and helped fix it.
The second named author is grateful to DPMMS at the University of Cambridge, especially Holly Krieger,  for hospitality and excellent working conditions during part of the preparation of this paper, and Evgeny Shinder for a conversation in which he asked about a possible connection between the results of \cite{PP1}, \cite{PP2} and $K$-theory.

\section{Preliminaries}

\subsection{Mixed Hodge theory and higher singularities.}\label{scn:higher-sings}
For a detailed version of the material in this section, and further references, the reader may consult \cite[Ch. B]{PP1}.

For a complex $n$-dimensional variety $X$, and for each $0\le p \le n$, the \emph{$p$-th Du Bois complex} of $X$ is 
\[\DB^p_X : = \gr^p_F\, \DB^\bullet_X[p],\]
where $(\DB_X^\bullet, F)$ is the \emph{filtered de Rham complex} of $X$, an object in the bounded derived category of filtered differential complexes on $X$ which plays a role analogous to the standard de Rham complex in the smooth case, and which can be computed in terms of the de Rham complexes on a hyperresolution of $X$.
Each $\DB^p_X$ is an object in ${\bf D}^b_{\rm coh}(X)$, the bounded derived category of coherent sheaves.

Note that there exists a Hodge-to-de Rham spectral sequence 
	$$E^{p,q}_1 = \HH^q (X, \DB_X^p) \implies H^{p + q} (X, \C),$$
	which degenerates at $E_1$ if $X$ is projective; see e.g. \cite[Proposition 7.24]{peters steenbrink}. The filtration induced by this spectral sequence coincides with the Hodge filtration $F_{\bullet} H^{p+q} (X, \C)$ associated to Deligne's mixed Hodge structure on cohomology. As for smooth varieties, in the projective case we define 
$$
\underline h^{p,q}(X):=\dim_\C \gr_F^{p}H^{p+q}(X,\C),
$$
and the degeneration mentioned above shows that they are also computed by the hypercohomologies of the Du Bois complexes $\DB_X^p$, in the sense that 
$$\gr_F^{p}H^{p+q}(X,\C) \simeq \HH^q(X,\DB_X^p).$$
We call $\underline h^{p,q}(X)$ the \emph{Hodge-Du Bois numbers} of $X$, in order to emphasize the difference between them and the \emph{Hodge-Deligne numbers} of the pure Hodge structures on the associated graded pieces with respect to the weight filtration
$$h_i^{p,r} (X) : = \dim_{\C} \gr_F^p \gr_{p+r}^W H^i(X, \C).$$
By definition and the strictness of the Hodge filtration, we have 
\begin{align*}
    \underline{h}^{p,i-p} = \sum_{r=0}^{i-p} h_i^{p,r}.
\end{align*}

The following technical result is essentially (a part of) \cite[Lemma $3.23$]{FL}, with a small addition. We include the proof for completeness.

\begin{lemma}\label{lem hodge du bois numbers}
Let $X$ be a projective variety of dimension $n$. For integers $m\ge -1$ and $i\ge 0$, we have
$$\underline{h}^{p,i-p}(X) = \underline{h}^{i-p,p}(X) \,\,\,\,{\rm for~ all} \,\,\,\, p \leq m \iff 
\gr_F^p W_{i-1} H^i(X, \mathbb{Q}) = 0  \,\,\,\,{\rm for~ all} \,\,\,\, p \leq m.$$
Moreover, if this holds and $i \ge 2m +2$, then $H^i(X, \mathbb{Q})$ has weights $ \geq 2m+2$.
\end{lemma}
\begin{proof}
The idea is that there are inequalities $\underline{h}^{p,i-p} \geq \underline{h}^{i-p,p}$ for each $p$ as above, and the differences measure lower weight pieces of $H^i(X, \mathbb{C})$. 
Recall that $H^i(X, \mathbb{Q})$ has weights only between $0$ and $i$.

When $i \le 2m$, both sides of the equivalence are easily seen to mean that $H^i (X, \Q)$ is a pure Hodge structure.

When $i \ge 2m +1$, we induct on $m$. The base case $m = -1$ is vacuous (the last statement holds for arbitrary varieties). By induction, we may assume $h_i^{i-m,r} = 0$ whenever $r < m$. Since there is a duality $h_i^{j,k}= h_i^{k,j}$, we have 
\begin{align*}
    \underline{h}^{m,i-m} - \underline{h}^{i-m,m} =& \sum_{r=0}^{i-m} h_i^{m,r} - \sum_{r=0}^m h_i^{i-m,r} \\
    =&  \sum_{r=0}^{i-m-1} h_i^{m,r}, 
\end{align*}
This sum is the dimension of $\gr_F^m W_{i-1} H^i(X, \mathbb{Q})$, which shows the equivalence.

To deduce the final statement, note that it follows from the strictness of the Hodge filtration that $\gr_F^p W_{2m+1} H^i(X, \mathbb{Q}) = 0$ for all $p \leq m$. By duality, we conclude that $W_{2m+1} H^i(X, \mathbb{Q}) = 0$.
\end{proof}

We are particularly interested here in the vanishing of certain weight spaces. In order to apply Lemma \ref{lem hodge du bois numbers}, we need the symmetry of various 
Hodge-Du Bois numbers. This is studied in \cite{PP1}, in terms of higher singularities, which we now introduce.

Following \cite{MOPW,JKSY,FL, MP-LC}, if $X$ is a local complete intersection (LCI) subvariety of a smooth variety $Y$, we say that it has \textit{$m$-Du Bois singularities} if the canonical morphisms 
$$\Omega^p_X \to \DB^p_X$$ 
are isomorphisms for all $0\le p\le m$, and 
 \textit{$m$-rational singularities} if the canonical compositions 
 $$\Omega^p_X \to \DB_X^p \to \mathbf{D}_X(\DB^{n-p}_X)$$  
 are isomorphisms for all $0\le p\le m$, where $\mathbf{D}_X (\cdot) : = R\H om (\, \cdot \, , \omega_X)$.
 For $m =0$, the LCI condition is not needed, and these definitions coincide with the classical notions of Du Bois and rational singularities, respectively. These singularity notions have found many uses recently, due to the fact that there are numerical criteria to detect them. We only
 recall this in the case of hypersurfaces (which can be assumed to be in $\C^n$ without loss of generality), but there exist similar criteria for LCI varieties of higher codimension \cite{MP-LC}, \cite{CDM}.
 
 \begin{theorem}
If $X = (f = 0)$ is a singular hypersurface in $\C^n$, and $m \ge 0$, then:
 
 \noindent
 (i) (\cite{MOPW}, \cite{JKSY})~~$X$ has $m$-Du Bois singularities $\iff \widetilde{\alpha} (f) \ge m+1$.
 
 \noindent
 (ii) (\cite{saito-appendix}, \cite{MP-lci})~~$X$ has $m$-rational singularities $\iff \widetilde{\alpha} (f) > m+1$.
 \end{theorem}
 
Here  $\widetilde{\alpha} (f)$ is the \emph{minimal exponent} of $f$, meaning the negative of the largest root of the reduced Bernstein-Sato polynomial $\tilde{b}_f (s)$. See \cite[\S2.5]{MOPW} for a detailed review of the main properties of $\widetilde{\alpha} (f)$, and references. A crucial point is that there exists programs computing it, or even $\tilde{b}_f (s)$, while for certain classes of singularities there are concrete formulas.

\begin{example}\label{ex:min-exp}
(i) If $f$ is a weighted homogenous polynomial of total weight $1$, with $x_i$ of weight $w_i$, then 
$\widetilde{\alpha} (f) = w_1 + \cdots + w_n$.

\noindent
(ii) If $f$ has an ordinary singularity of multiplicity $m$ (e.g. the cone over a hypersurface of degree $m$ in $\mathbb{P}^{n-1}$), then $\widetilde{\alpha} (f) = \frac{n}{m}$.
\end{example}

For non-LCI varieties and $m > 0$, the definitions above turn out to be too restrictive, as explained in \cite{SVV}, where new definitions are introduced in the general setting.
It is also pointed out in \emph{loc. cit.} that often it suffices to consider weaker notions obtained by focusing only on cohomological degree $> 0$.

\begin{definition}\label{definition:pre-k-DB-rational}
	One says  that $X$ has \textit{pre-$m$-Du Bois} singularities if 
		\[ \H^i\DB_X^p=0 \,\,\,\,{\rm for~ all } \,\,\,\,i>0 \,\,\,\, {\rm and} \,\,\,\, 0\le p\le m,\]
and that it has \textit{pre-$m$-rational} singularities if
		\[ \H^i (\mathbf{D}_X(\DB^{n-p}_X))=0 \,\,\,\,{\rm for~ all } \,\,\,\,i>0 \,\,\,\, {\rm and} \,\,\,\, 0\le p\le m.\]
\end{definition}

It is known that (pre)-$m$-rational singularities are (pre)-$m$-Du Bois (see \cite[Theorem B]{SVV}), which generalizes the well known fact that rational singularities are Du Bois. A more precise statement, which combines  \cite[Proposition 9.4]{PSV} and \cite[\S4,5]{DOR}, is the following:

\begin{theorem}
A normal variety $X$ has pre-$m$-rational singularities if and only if it has pre-$m$-Du Bois singularities and satisfies condition $D_m$. 
\end{theorem}

Condition $D_m$, considered in \cite{PP1}  and in \cite{DOR} (under a different name in the latter), stands for duality up to level $m$; concretely, it means that the canonical duality morphisms 
$$\DB_X^p \to \mathbf{D}_X(\DB^{n-p}_X)$$ 
are isomorphisms for $p \le m$.\footnote{Note that, unlike for smooth varieties, in general this is not usually the case.}

Going back to the question of the symmetry of Hodge-Du Bois numbers, we have the following:

\begin{theorem}[{\cite[Theorem A]{PP1}}]\label{thm hodge symmetry}
If $X$ be a normal complex projective variety of dimension $n$ which satisfies condition $D_m$, then
\begin{align*}
    \underline{h}^{p,q}(X) = \underline{h}^{q,p}(X) = \underline{h}^{n-p, n-q}(X) = \underline{h}^{n-q, n-p}(X)
\end{align*}
for all $0 \leq p \leq m$ and $0 \leq q \leq n$.
\end{theorem}

Combining this with Lemma \ref{lem hodge du bois numbers}, we obtain the following consequence that is important for this paper, but was not stated explicitly in \cite{PP1}. This is the key link between singularity conditions and weight conditions.
The case $m = -1$, meaning no assumptions on the singularities of $X$, is well known (see \cite[Table $5.1$]{peters steenbrink}).

\begin{corollary}\label{cor weight restrictions}
Let $X$ be a projective variety of dimension $n$. Suppose $X$ has pre-$m$-rational singularities (or more generally satisfies condition $D_m$), with $2m + 2 \le n$, and is normal if 
$m \geq 0$. Then:

\noindent 
(i) If $i \le n$, then 
$$H^{p, j-p} \gr_j^W H^i (X, \Q) = 0, \,\,\,\,\,\,{\rm for ~all} \,\,\,\,p \le m, ~j \le i -1.$$
In particular, $H^i(X, \mathbb{Q})$ has weights $\geq \min \{ 2m+2, i \}$.

\noindent 
(ii) If $i \ge n$, then 
$$H^{p, j-p} \gr_j^W H^i (X, \Q) = 0, \,\,\,\,\,\,{\rm for ~all} \,\,\,\,p \le i - n +m, ~ j \le i -1.$$
In particular, $H^i(X, \mathbb{Q})$ has weights $ \geq \min \{ 2i-2n+2m+2, i \}$. 
\end{corollary}

\begin{proof}
Suppose $i \leq n$. The hypotheses of Lemma \ref{lem hodge du bois numbers} are satisfied by Theorem \ref{thm hodge symmetry}, to give the vanishing of the corresponding Hodge-Deligne numbers. (Those for $j < i -1$ vanish by the strictness of the Hodge filtration.) In particular,  $H^i(X, \mathbb{Q})$ has weights $\geq 2m+2$ when $i \ge 2m +2$.

Suppose now $i \geq n$, and write $i = (n-p) + (n-q)$, so that Theorem \ref{thm hodge symmetry} gives the symmetry 
$$\underline{h}^{n-q,i-n+q}(X) = \underline{h}^{i-n+q,n-q}(X)\,\,\,\,\,\,{\rm for ~all} \,\,\,\, n-q \le i - n + m.$$ 
Then the same argument as above applies. In particular, if $i \leq 2n-2m-2$, then $H^i(X, \mathbb{Q})$ has weights $\geq 2(i-n+m)+2$.
\end{proof}

\begin{example}[{\bf Rational singularities}]\label{ex:rational}
Theorem \ref{thm hodge symmetry} and Corollary \ref{cor weight restrictions} contain new information already in the classical case of rational singularities, i.e. $m =0$. First, if $i \le n$, then $H^i (X, \Q)$ must have weights at least $2$. But we also have 
$$H^{0,j} \gr_j^W H^i (X, \Q) =0  \,\,\,\,\,\,{\rm for ~all} \,\,\,\, j \le i -1.$$
In particular
$$\gr_2^W H^i (X, \Q) = H^{1,1} \gr_2^W H^i (X, \Q).$$
Moreover, if $i \ge  n +1$, we have 
$$H^{p,j- p} \gr_j^W H^i (X, \Q) =0  \,\,\,\,\,\,{\rm for ~all} \,\,\,\,p\le i -n, j \le i -1.$$
As a first example, if $n \ge 4$, then
$$\gr_4^W H^i (X, \Q) = H^{2,2} \gr_4^W H^i (X, \Q).$$
\end{example}

\begin{remark}\label{rmk:purity-sing}
Recall also that in the upper part of the Hodge diamond, some purity holds for a more basic reason:  if $\dim X_{\rm sing} = s$, then the Hodge structure on $H^k (X,\Q)$ is pure for 
$k > n+ s$; see e.g. \cite[Theorem 6.33]{peters steenbrink}.
\end{remark}

\begin{remark}[{\bf The non-projective case}]\label{rmk:Dm-nonproj}
One can check that the machinery in \cite{PP1} yields an enhancement of Corollary \ref{cor weight restrictions} that holds in the non-projective setting as well. Namely, if $X$ satisfies condition $D_m$, then the trivial Hodge module $\Q_X^H [n]$ has weights $\ge 2m +2$. It is not hard to see that this implies in turn the hypothesis of Conjecture \ref{main-KH-conjecture}, for an appropriate $i$; see the discussion before Conjecture 
\ref{sings-KH-conjecture}.
\end{remark}

\subsection{Weights in integral Hodge theory}\label{scn:weights-integral}

In this section, we discuss a weight filtration defined on the integral cohomology of a projective variety. Recall that Deligne constructed a weight filtration on the singular cohomology with rational coefficients of any complex algebraic variety. Let us briefly review the construction in the case where $X$ is projective.

Choose a cubical hyperresolution $\epsilon_\bullet \colon X_\bullet \rightarrow X$ of $X$. 
Recall that $X_\bullet$ is a smooth cubical variety, built by resolving X and then successively resolving the singular loci and exceptional divisors that occur at each stage. It satisfies cohomological descent, that is the constant sheaf $\mathbb{Z}_X$ is quasi-isomorphic to the iterated cone
\begin{align}\label{eqn hyperres cone}
    \mathbb{Z}_X = \bigg( R \epsilon_{0*} \mathbb{Z}_{X_0} \rightarrow R \epsilon_{1*} \mathbb{Z}_{X_1} \rightarrow \ldots \bigg).
\end{align}
For later use, we recall that such a  cubical hyperresolution can be assumed  to satisfy $\dim X_i \leq n-i$ for each $i$; this can always be done by \cite[Theorem 2.15]{GNPP}.

After tensoring with $\mathbb{Q}$, we obtain a (cohomological) weight spectral sequence
\begin{align}\label{eqn:weight-ss-rational}
    E_1^{p,q} = H^q(X_p, \mathbb{Q}) \Rightarrow H^{p+q}(X, \mathbb{Q}).
\end{align}
The terms $E_1^{p,q}$ naturally carry pure Hodge structures of weight $q$, and the boundary maps $d_1$ respect the Hodge filtration. It follows that the subquotients $E_2^{p,q}$ are pure of weight $q$ as well. The boundary maps $d_2: E_2^{p,q} \rightarrow E_2^{p+2, q-1}$ are between pure Hodge structures of different weight and they preserve the Hodge filtration. Therefore, they all vanish. The same argument shows the boundary maps $d_r$ for $r \geq 2$ all vanish. This shows that the spectral sequence degenerates at the $E_2$ page.

The weight filtration is, by definition, the induced filtration on the cohomology groups $H^{p+q}(X, \mathbb{Q})$. Concretely, for the associated graded factors we have:
\begin{align}\label{eqn:weight-hyper}
    \gr_q^W H^{p+q}(X, \mathbb{Q}) = \frac{ \ker \bigg( H^q(X_p, \mathbb{Q}) \rightarrow H^q(X_{p+1}, \mathbb{Q}) \bigg) }{ \im \bigg( H^q(X_{p-1}, \mathbb{Q}) \rightarrow H^q(X_p, \mathbb{Q}) \bigg) }
\end{align}
See  \cite[Chapter $5$]{peters steenbrink}.

Let us work now with integral coefficients. The iterated cone (\ref{eqn hyperres cone}) gives rise to a cohomology spectral sequence
\begin{align}\label{eqn:weight-ss-integral}
    E_1^{p,q} = H^q(X_p, \mathbb{Z}) \Rightarrow H^{p+q}(X, \mathbb{Z}).
\end{align}
This time, given the lack of an integral Hodge filtration, the spectral sequence need not degenerate. Nevertheless, it is an invariant of $X$.

\begin{proposition}[{\cite[Theorem $3$]{gillet soule}}]\label{prop integral weight}
Let $X$ be a complex projective variety. The spectral sequence (\ref{eqn:weight-ss-integral}) is an invariant of $X$, from the $E_2$-page onwards. It defines a weight filtration $W_\ell H^i(X, \mathbb{Z})$ on the cohomology of $X$ with integral coefficients, which agrees with Deligne's weight filtration after tensoring with $\mathbb{Q}$. The filtration respects pullback by proper morphisms, and there are Mayer-Vietoris sequences from the $E_2$-page onwards.
\end{proposition}

\begin{remark}
There is in fact an integral weight filtration for arbitrary complex algebraic varieties $X$; see \cite[Theorem  3]{gillet soule}. When $X$ is not projective the construction is less explicit, and we will not use it here.
\end{remark}

We will use the notation
$$H_q^{p+q}(X, \mathbb{Z}) := E_2^{p,q}.$$
This is a subquotient of $\gr_q^W H^{p+q}(X, \mathbb{Z})$, and they coincide after tensoring with $\mathbb{Q}$.

Let us note a few elementary properties of this integral weight filtration.

\begin{lemma}\label{lem integral weight properties}
\begin{enumerate}
    \item If $X$ is normal, then $H^1(X, \mathbb{Z})$ is pure.
    \item If $X$ has rational singularities, then $H^2(X, \mathbb{Z})$ is pure. 
    \item If $X$ is an integral homology manifold, then each $H^i(X, \mathbb{Z})$ is pure.
    \item If $E$ is a simple normal crossings divisor and $\mathcal{D}(E)$ is its dual complex, then $H_0^i(E, \mathbb{Z}) \simeq H^i(\mathcal{D}(E), \mathbb{Z})$.
\end{enumerate}
\end{lemma}

\begin{proof}
$(1)$ It is well known that $W_0 H^1(X, \mathbb{Q}) = 0$; see e.g. \cite[Example 8.11]{PP1} and the references therein. Since $H^1(X, \mathbb{Z})$ is torsion-free by the universal coefficient theorem, $W_0 H^1(X, \mathbb{Z})$ must vanish as well.

$(2)$ If $Y \rightarrow X$ is a resolution, then for example \cite[Lemma $2.1$]{bl} shows the pullback $H^2(X, \mathbb{Z}) \hookrightarrow H^2(Y, \mathbb{Z})$ is injective by analyzing the Leray spectral sequence.

$(3)$  We are assuming that $\mathbb{Z}_X \rightarrow \mathrm{IC}_{X, \mathbb{Z}}$ is an isomorphism. For a resolution $f\colon Y \rightarrow X$, Poincaré duality shows $f_* \circ f^* = \id$. Hence, each $f^*\colon H^i(X, \mathbb{Z}) \rightarrow H^i(Y, \mathbb{Z})$ is injective.

$(4)$ We obtain a cubical hyperresolution by $E$ by taking $E_0$ to be the disjoint union of the irreducible components of $E$, $E_1$ to be the disjoint union of the pairwise intersections of components of $E$, and so on. Then $H^i_0(E, \mathbb{Z})$ is the $i$th cohomology group of the complex
\begin{align}\label{eqn snc wt zero}
   0\to H^0(E_0, \mathbb{Z}) \rightarrow H^0(E_1, \mathbb{Z}) \rightarrow \cdots.
\end{align}
Recall that the dual complex $\mathcal{D}(E)$ is a CW complex whose vertices correspond with the irreducible components of $E$, whose edges correspond to the pairwise intersection components of the components of $E$, and so on. We observe that the complex is precisely the simplicial cochain complex of $\mathcal{D}(E)$. Hence, $H^i_0(E, \mathbb{Z}) \simeq H^i(\mathcal{D}(E), \mathbb{Z})$.
\end{proof}

The weight filtration on integral cohomology carries nontrivial information about the singularities of $X$, even on torsion subgroups, as the following example illustrates.

\begin{example}[{\cite[$3.1.3$]{gillet soule}}]\label{example kummer surface}
Let $S$ be a Kummer surface, namely the quotient of an abelian surface $A$ by the involution $\iota \colon x \mapsto -x$. The surface $S$ has $16$ isolated singularities, each of type $A_1$. The cohomology group $H^3(S, \mathbb{Z}) = (\mathbb{Z}/2 \mathbb{Z})^5$ is pure of weight two.
We extend this to Kummer varieties of arbitrary dimension in Example \ref{ex:Kummer}.
\end{example}

We have seen in Corollary \ref{cor weight restrictions} that quite a bit is known about the rational weight filtration, considered as an invariant of singularities.  The preceeding example shows that this does not extend immediately to the integral setting. Forthcoming work of the first author \cite{burke dual complex} clarifies this picture to some extent, using singularities of the minimal model program.

\begin{theorem}[{\cite{burke dual complex}}]\label{thm wt zero klt}
Suppose $X$ is projective and of klt type.
Then $H_0^i(X, \mathbb{Z}) = H_1^i(X, \mathbb{Z}) = 0$ for $i > 0$. In particular, 
\begin{align*}
    H^i(X, \mathbb{Z}) {\rm~ has ~weights} \geq \min \{ i, 2\}.
\end{align*}
\end{theorem}

\begin{remark}[{\bf Klt type}]\label{rmk:klt}
We recall that a variety \emph{of klt type} is a normal variety $X$ such that there exists an effective $\Q$-divisor $\Delta$ on $X$ with the property that the 
pair $(X, \Delta)$ is klt; see \S\ref{scn:MMP} for more on this notion. By a well-known result of Elkik, if $X$ is of klt type, then it has rational singularities. If $X$ is in addition Gorenstein, then the two notions are equivalent, since in this case rational and canonical singularities are the same. Using klt rather than rational singularities will be crucial at times in this paper, but this is therefore relevant only for non-Gorenstein varieties.
\end{remark}

\begin{remark}
If $X$ has isolated klt type singularities, then \cite{dFKX} prove the dual complex $\mathcal{D}(E)$ is contractible. Using \ref{lem integral weight properties}(4), we deduce $H_0^i(X, \mathbb{Z}) = 0$ for $i > 0$. The above theorem is an extension of this result. Note also that varieties of klt type have rational singularities, and the result holds over $\Q$ for all rational singularities by Corollary \ref{cor weight restrictions}.
\end{remark}

\subsection{$K$-theory and $KH$-theory}\label{scn:K-intro}
In this section, $X$ will always be a complex quasi-projective variety of dimension $n$. Essentially all of the content recalled here can be found in Weibel's $K$-book \cite{k-book}.

\smallskip

\noindent
{\bf Basics of (negative) $K$-Theory.}
As is standard, $K_0(X)$ denotes the group completion of the additive monoid of vector bundles on $X$. The $K$-groups $K_i(X)$ can be defined for all 
$i \in \mathbb{Z}$, for instance by constructing a (non-connective) spectrum $\mathcal{K}(X)$ called the $K$-theory spectrum of $X$, and setting $K_i(X) = \pi_i \mathcal{K}(X)$. 
We have the following ``Fundamental Theorem of $K$-theory" of Bass and Quillen, which says in particular that higher $K$-groups determine the lower ones in a suitable sense.

\begin{theorem}[{\cite[Theorem V.8.3]{k-book}}]\label{thm fund}
There exists a four-term exact sequence
\begin{align*}
    0 \rightarrow K_i(X) \xrightarrow{\Delta} K_i(X[t]) \oplus K_i(X[t^{-1}]) \xrightarrow{\pm} K_i(X[t, t^{-1}]) \rightarrow K_{i-1}(X) \to 0.
\end{align*}
\end{theorem}

\smallskip

The map $\pm$ is the difference of the maps induced by the open inclusions
\begin{align*}
    X[t,t^{-1}] := X \times \mathop{\mathrm{Spec}}\mathbb{C}[t, t^{{-1}}] \rightarrow X[t]  = X \times \mathop{\mathrm{Spec}}\mathbb{C}[t]
\end{align*}
and
\begin{align*}
    X[t,t^{-1}] = X \times \mathop{\mathrm{Spec}}\mathbb{C}[t, t^{{-1}}] \rightarrow X[t^{-1}]  = X \times \mathop{\mathrm{Spec}}\mathbb{C}[t^{-1}],
\end{align*}
respectively, while $\Delta$ is the diagonal map. We use the above notation to emphasize that, although $X[t]$ and $X[t^{-1}]$ are isomorphic schemes, they are viewed as distict partial compactifications of $X[t,t^{-1}]$. The geometric picture to keep in mind is that these are the two standard copies of $X \times \A^1$ covering $X \times \mathbb{P}^1$, whose intersection is  $X \times \mathbb{G}_m$.

In this paper we are almost exclusively interested in non-positive $K$-groups. In this case the picture can be seen in a more elementary way, in the sense that  
we can view Theorem \ref{thm fund} as a definition, and negative $K$-groups can be built only out of the $K_0$-groups on various spaces.\footnote{One of course needs
homotopy-theoretic methods in order to prove many of the fundamental facts we will use.}

\begin{definition}
We define the negative $K$-groups $K_{-i}(X)$ for $i > 0$ inductively. Assuming $K_{-i + 1}$ has already been defined, we set
\begin{align}\label{eqn K-theory recursive}
    K_{-i}(X) = \coker \bigg( K_{-i+1}(X[t]) \oplus K_{-i+1}(X[t^{-1}]) \xrightarrow{\pm} K_{-i+1}(X[t,t^{-1}] )\bigg).
\end{align}
\end{definition}

Note that iterating \ref{eqn K-theory recursive} leads to the promised description only in terms of $K_0$-groups.
\begin{align}\label{eqn K-i iterative}
    K_{-i}(X) = \coker \bigg( \bigoplus_{j=1}^i  \big( K_0(X[t_1^{\pm 1}, \ldots t_j, \ldots , t_i^{\pm 1}])& \oplus K_0(X[t_1^{\pm 1}, \ldots t_j^{-1}, \ldots , t_i^{\pm 1}]) \big) \\ & \xrightarrow[]{\oplus (\pm 1)}  K_0(X[t_1^{\pm 1}, \ldots , t_i^{\pm 1}] \bigg).
\end{align}
In words, $K_{-i}(X)$ is the quotient of $K_0(X \times \mathbb{G}_m^i)$, modulo the subgroup generated by virtual vector bundles which arise as differences of classes on one of the $2i$ standard partial compactifications $\mathbb{A}^1 \times \mathbb{G}_m^{i-1}$ of $\mathbb{G}_m^i$.

If $X$ is smooth, then it has no negative $K$-groups by another well-known result going back to Grothendieck.

\begin{theorem}[{\cite[Theorem II.$7.8$]{k-book}}]\label{thm:reg-smooth}
If $X$ is smooth, then 
$$K_0(X) \simeq K_0(X[t]) \simeq K_0(X[t^{\pm 1}]).$$
In particular, $K_{-i} (X) = 0$ for $i > 0$.
\end{theorem}

\smallskip

\noindent
{\bf $K$-regularity.}
Theorem \ref{thm:reg-smooth} implies that negative $K$-theory may appear only in the singular setting. At the same time, it suggests a coarser numerical invariant that has been studied in the literature.

\begin{definition}\label{def:K-regularity}
We say that $X$ is $K_i$-regular if the natural maps
\begin{align*}
    K_i (X) \rightarrow K_i (X \times \mathbb{A}^r)
\end{align*}
are isomorphisms for all $r \geq 1$.\footnote{It is known that requiring isomorphisms for all $r \geq 1$ rather than just $r=1$  is necessary for a good definition.}
\end{definition}

In this language, Theorem \ref{thm:reg-smooth} is part of a more general statement saying that a smooth variety is $K_i$-regular for all $i$. For singular varieties, it is important to give bounds on the largest index $i$ for which this property holds, as regularity makes calculations more tractable. Concretely,  it will allow us to take advantage of a homotopy-invariant version of $K$-theory, recalled in the next subsection.

If $X$ is $K_m$-regular, then it is also $K_{m-1}$-regular by \cite{vorst} and \cite{dayton weibel}. In the case of affine varieties Vorst's conjecture predicts, conversely, that $K_{n+1}$-regularity implies that $X$ is smooth. This is verified in \cite[Theorem 0.1(c)]{vorst conj char 0} over a characteristic zero field, and more generally for perfect fields in \cite[Corollary 4.7]{vorst conj char p}.

The following fundamental result was conjectured by Weibel \cite[Question $2.9$]{weibel conj}, and proved in \cite[Theorem 6.2]{k weibel} and \cite[Theorem B]{descent}. (Both the conjecture and the theorem hold in a more general setting than that of complex varieties.) The first statement is in fact a consequence of the second, thanks to Corollary \ref{weibel conj homotopy} below and the surrounding discussion.

\begin{theorem}\label{cor:weibel-conj}
If $X$ is a variety of dimension $n$, then $K_{-i}(X) = 0$ for $i > n$. Moreover, $X$ is $K_{-n}$-regular.
\end{theorem}

A recent discovery, relying on results from \cite{k weibel,bass nk, haesemeyer weibel}, is that regularity for other $K$-groups goes hand in hand with the higher singularities notions introduced in \S\ref{scn:higher-sings}.

\begin{theorem}[{\cite[Theorem C]{rosie}}]\label{thm rosie}
Let $X$ be a complex variety of dimension $n$ with $\dim X_{\sing}=s$.

\noindent
(i) If $X$ has Du Bois singularities, then it is $K_{-n+1}$-regular.

\noindent
(ii) If $X$ has pre-$m$-Du Bois singularities, then it is $K_{-t}$-regular for $t\ge \max\{1+s, n-2m-2+s\}$. If $X$ is in addition seminormal, then the conclusion extends to $t=s$.

\noindent 
(iii) If $X$ is a affine local complete intersection and $X$ is $K_{- n + 2m +1 + s}$-regular, then it is pre-$m$-Du Bois.

If $X$ is a local complete intersection with isolated singularities, then 
\begin{align*}
\text{$X$ is $m$-Du Bois} &\iff \text{$X$ is $K_{-n+2m+2}$-regular} \\
    &\iff \text{$X$ is $K_{-n+2m+1}$-regular}
\end{align*}
\end{theorem}

For example, while regularity always holds for $K_{-n}$, the next simplest fact covered by Theorem \ref{thm rosie} is that the standard notion of Du Bois singularities implies $K_{-n +1}$-regularity, and even $K_{-n +2}$-regularity if the singularities are isolated.

\begin{remark}
In the upcoming \cite{BPS} we will show that under the seminormal and pre-$m$-Du Bois hypothesis, the bound in (ii) can be 
improved to $t\ge \max\{s, n-2m-2\}$. This will be a consequence of finer comparison statements between motivic and cdh cohomology.
\end{remark}

\smallskip

\noindent
{\bf Homotopy $K$-theory.}
In \cite{weibel kh}, Weibel defined the homotopy $K$-theory $KH_i (X)$ as the homotopy groups $\pi_i \mathcal{KH} (X)$ of a homotopy $K$-theory spectrum. These satisfy formal properties analogous to those of the usual $K$-theory groups, but by contrast, and by their very definition, they are $\A^1$-invariant.

The non-positive $KH$-groups again admit a more elementary description; see  \cite[Section IV.$12$]{k-book}.
Let $KH_0(X)$ be the quotient of $K_0(X)$ by the subgroup generated by $i_0^*(\alpha) - i_1^*(\alpha)$, where $i_0$ (resp. $i_1$) are the inclusions of $X$ into $X \times \mathbb{A}^1$ as $X \times \{0\}$ (resp. $X \times \{1\}$) and $\alpha \in K_0(X \times \mathbb{A}^1)$. 
Then, as before:

\begin{definition}
We define the negative $K$-groups $KH_{-i}(X)$ for $i > 0$ inductively. Assuming $KH_{-i + 1}$ has already been defined, we set
\begin{align}\label{eqn KH-theory recursive}
    KH_{-i}(X) = \coker \bigg( KH_{-i+1}(X[t]) \oplus KH_{-i+1}(X[t^{-1}]) \xrightarrow{\pm} KH_{-i+1}(X[t,t^{-1}] )\bigg).
\end{align}
\end{definition}

Note that by construction, for each $i$ we have a natural morphism 
$$\varphi_i \colon K_i (X) \to KH_i (X).$$
It is known that $K_m$-regularity implies $\varphi_i$ is an isomorphism for all  $i \leq m$; see \cite[Corollary IV.$12.3.2$]{k-book}.
Thus for smooth varieties we have $K_i (X) = KH_i (X)$ for all $i$, as we have seen in Theorem \ref{thm:reg-smooth} that in this case $K$-theory is $\A^1$-invariant. In the singular case, once $K_i$-regularity is established for some index $i$, it suffices to focus on $KH_i (X)$.

An important fact, proved by Haesemeyer \cite{descent properties}, is that $\mathcal{KH}(\cdot)$ is a sheaf in the cdh topology.
As a consequence, we are free to use Mayer-Vietoris to compute $KH_i (X)$. 
Recall that we can choose a  cubical hyperresolution $\epsilon_\bullet \colon X_\bullet \rightarrow X$ such that $\dim X_i \leq n-i$ for each $i$. There is a spectral sequence
\begin{align}\label{spectral sequence hyperres}
    E_1^{p,q} = K_{q}(X_p) = KH_q(X_p) \Rightarrow KH_{q-p}(X).
\end{align}

Note that for complex varieties the homotopy-invariant version of the vanishing in Theorem \ref{cor:weibel-conj}, first noted in \cite[Theorem 7.1]{descent properties}, is an immediate consequence. Indeed, in the set-up above, we have $X_i = \emptyset$ for $i>n$. Thus $E_1^{p,q} = 0$ for $p > n$ and for $q < 0$, and therefore:

\begin{corollary}\label{weibel conj homotopy}
$KH_i(X) = 0$ for $i < -n$.
\end{corollary}

We will approach the computation of $KH$-groups by means of cdh-motivic cohomology and the analogue of the Atiyah-Hirzebruch spectral sequence, as described 
in the next section.

\subsection{Cdh-motivic cohomology}\label{scn:cdh}
In this section we review some basics on the motivic cohomology of arbitrary complex varieties, which refines the study of $K$-groups.

\smallskip

\noindent
{\bf Higher Chow groups and motivic cohomology on smooth varieties.}
We first recall some important notions that are  building blocks in the $K$-theory of smooth varieties.

Let $X$ be a smooth variety of dimension $n$. For $k \geq 0$, denote
\begin{align*}
    \Delta^k = \mathop{\mathrm{Spec}} \bigg( \mathbb{C}[t_0, \ldots , t_k] / \big( \sum t_i - 1 \big) \bigg),
\end{align*}
the algebraic $k$-simplex. Abstractly, $\Delta^n$ is isomorphic to affine space $\mathbb{A}^n$. However, it also carries natural faces and degeneracies. Given an increasing map $\rho: \{0, \ldots , m\} \rightarrow \{0, \ldots , k\}$, define $\Tilde{\rho}: \Delta^m \rightarrow \Delta^k$ by $\Tilde{\rho}^*(t_i) = \sum_{\rho(j)=i} t_j$ and $\Tilde{\rho}^*(t_i) = 0$ if $\phi^{-1}(i) = \emptyset$. If $\rho$ is injective, we say $\Tilde{\rho}(\Delta^m) \subset \Delta^k$ is a face. If $\rho$ is surjective, $\Tilde{\rho}$ is a degeneracy.

For $0 \le k \le n$, define $z^\bullet (X,k)$ to be the free abelian group (graded by codimension) which is generated by the irreducible subvarieties of $X \times \Delta^k$ meeting each face $X \times \Delta^m \subset X \times \Delta^k$ properly. It is a standard check that $z^\bullet (X, \cdot)$ is preserved under pullback by face and degeneracy maps. As we vary $k$, it assembles into a simplicial complex of graded abelian groups $z^\bullet (X, \cdot)$.

The differential $\partial_i$ is pullback along the standard face map
\begin{align*}
    (t_0, \ldots , t_{k-1}) \mapsto (t_0, \ldots , t_{i-1}, 0 , t_i, \ldots t_{k-1})
\end{align*}
for $0 \leq i \leq k$, and the degeneracy $s_i$ is pullback along the standard degeneracy
\begin{align*}
    (t_0, \ldots , t_k) \mapsto (t_0, \ldots , t_{i-1}, t_i+t_{i+1} , t_{i+2}, \ldots t_k)
\end{align*}
for $0 \leq i \leq k-1$. The higher Chow groups $\CH^\bullet (X,k)$ are defined to be the homology of the associated chain complex. Concretely,
\begin{align*}
    \CH^\bullet (X,k) = \frac{\ker \bigg( \sum (-1)^i \partial_i : z^\bullet (X,k) \rightarrow z^\bullet (X,k-1) \bigg)}{\im \bigg( \sum (-1)^i \partial_i : z^\bullet (X,k+1) \rightarrow z^\bullet (X,k) \bigg)}
\end{align*}
For example, $\CH^\bullet (X,0)$ is the the quotient of $z^\bullet (X)$ by the subgroup generated by all cycles $Z(0)-Z(1)$, where $Z$ is a cycle on $X \times \mathbb{A}^1$, meeting the fibers over $0$ and over $1$ properly, i.e. the usual Chow group $\CH^\bullet (X)$.

There is an Atiyah-Hirzebruch spectral sequence
\begin{align*}
    E_2^{p,q} = \CH^{-q}(X, -p-q) \Rightarrow K_{-p-q}(X)
\end{align*}
relating higher Chow groups with $K$-theory \cite{Levine}, \cite{FS}. After tensoring with $\mathbb{Q}$, it degenerates at the $E_2$-page, thus providing a decomposition
\begin{equation}\label{eqn:K-decomposition}
    K_i(X)_\mathbb{Q} \simeq \bigoplus_p \CH^p(X, i)_\mathbb{Q},
\end{equation}
extending the well-known decomposition of $K_0(X)_\mathbb{Q}$ in terms of Chow groups, induced by the topological filtration (see \cite[15.2.16]{fulton}).

In the smooth case, motivic cohomology can be seen as a repackaging of higher Chow groups. One can define the weight $j$ motivic cohomology complex of  $X$ to be the chain complex of abelian groups 
$$\mathbb{Z}(j)(X) := z^j(X, \cdot )[-2j],$$ 
and then motivic cohomology is defined as the cohomology of this complex, i.e.
$$H^i(X, \mathbb{Z}(j)) := H^i(\mathbb{Z}(j)(X)).$$ 
Therefore we have $H^i(X, \mathbb{Z}(j)) \simeq \CH^j(X, 2j-i)$. It is clear that we have 
\begin{equation}\label{eqn:van-smooth-1}
H^i(X, \mathbb{Z}(j)) = 0 \,\,\,\,\,\,{\rm for} \,\,\,\, i > 2j.
\end{equation} 
We also have 
\begin{equation}\label{eqn:van-smooth-2}
H^i(X, \mathbb{Z}(j)) = 0 \,\,\,\,\,\,{\rm for} \,\,\,\, i > j + n.
\end{equation}
Indeed, $CH^j(X, 2j - i)$ is a subquotient of the group of codimension $j$ cycles on $X \times \mathbb{A}^{2j - i}$. 
However, $j$ is greater than the dimension $n + (2j - i)$ of the ambient space. See Lemma \ref{lem motivic soule-weibel} for a more general statement.

\smallskip

\noindent
{\bf Cdh-motivic cohomology on singular varieties.}
Let $\mathbb{Z}(j)^{\cdh}(\cdot)$ be the functor on arbitrary complex varieties defined as the cdh-sheafification of the functor $\mathbb{Z}(j)(\cdot)$ defined as above on smooth varieties. For an arbitrary complex variety $X$, following \cite{BEM} we denote 
$$H^i_{\cdh}(X, \mathbb{Z}(j)) : = H^i\big(\mathbb{Z}(j)^{\cdh}(X)\big).$$
This is called the \emph{cdh-motivic cohomology} of $X$, or we will simply say the \emph{cdh cohomology} of $X$. Over the complex numbers, it coincides with the motivic cohomology introduced by Friedlander and Voevodsky \cite{FV}.  By definition, it can be computed in terms of the motivic cohomology of the smooth constituents of a cubical hyperresolution of $X$. Concretely, if $\epsilon_\bullet \colon  X_\bullet \to X$ is any such hyperresolution, then  for each $j$ there is a \emph{descent spectral sequence}
\begin{equation}\label{eq:ss-cdh}
E_1^{p,q} = H^q(X_p, \mathbb{Z}(j)) \Rightarrow H^{p+q}_{\cdh}(X, \mathbb{Z}(j)),
\end{equation}
where the terms on the left denote the motivic cohomology defined in the previous subsection. Thus in the smooth case we have 
$$H^i_{\cdh}(X, \mathbb{Z}(j)) = H^i (X, \mathbb{Z}(j)),$$
i.e. cdh cohomology coincides with motivic cohomology.

The crucial point is that this is an $\A^1$-invariant cohomology theory, that plays the same role for $KH$-theory that motivic cohomology plays for the $K$-theory of smooth varieties. We recall some important results in this direction, originally due to a number of authors at a level of generality that includes complex varieties as needed here. (Recently these results have been extended to arbitrary qcqs schemes in \cite[Theorem~1.1]{BEM}.)

\begin{theorem}\label{thm motivic cohom}
Let $X$ be a complex algebraic variety.
\begin{enumerate}
\item There exists an Atiyah-Hirzebruch type spectral sequence \cite[Theorem 7.13]{descent properties}
\begin{align*}
    E_2^{p,q} = H_{\cdh}^{p-q}(X, \mathbb{Z}(-q)) \Rightarrow KH_{-p-q}(X), 
\end{align*}
which degenerates rationally at the $E_2$ page by \cite[Theorem 4.27]{KP}.

\smallskip

\item The map $\mathbb{Z}(j)^{\cdh}(X) \rightarrow \mathbb{Z}(j)^{\cdh}(X \times \mathbb{A}^1)$ is an equivalence \cite[\S9]{FV}.

\smallskip

\item There are equivalences $\mathbb{Z}(0) = R\Gamma_{\cdh}(X, \mathbb{Z})$ and $\mathbb{Z}(1) = R\Gamma_{\cdh}(X, \mathbb{G}_m)[-1]$; see \cite[Theorem 4.1]{MVW}.
Therefore, if $X$ is smooth and connected, $\mathbb{Z}(0) = \mathbb{Z}$, considered as an abelian group concentrated in degree zero, and
\[
H^i_{\mathrm{cdh}}(X,\mathbb{Z}(1)) \cong 
\begin{cases}
H^0\bigl(X,\mathscr{O}_X^*\bigr) & \text{if } i = 1,\\[4pt]
\operatorname{Pic}(X) & \text{if } i = 2,\\[4pt]
0 & \text{otherwise.}
\end{cases}
\]
\end{enumerate}
\end{theorem}

\begin{remark}\label{rmk:Z-copy}
We note a useful fact about the spectral sequence in the theorem, namely that
$$E_2^{0,0} \simeq E_\infty^{0,0}.$$
It is straightforward to reduce this to the case when $X$ is connected, when $E_2^{0,0} \simeq H_{\rm cdh}^0 (X, \ZZ (0))\simeq \ZZ$.
We use the fact that the spectral sequence is ``multiplicative", meaning there are products 
$$E_r^{p,q} \otimes E_r^{a,b} \to E_r^{p+a, q+b}$$ 
which satisfy a Leibniz rule with respect to the differential; see \cite[Theorem 7.12(1)]{BEM}.
In particular, taking the element $1 \in E_r^{0,0}$, we have $d_r(1) = d_r( 1 \otimes 1) = d_r(1)\otimes1 + 1\otimes d_r(1)$, so $d_r(1) = 0$. 
Intuitively, this copy of $\ZZ$ survives at the $E_\infty$ stage because it appears in the surjective rank homomorphism ${\rm rk}: K_0(X) \to \ZZ$.
\end{remark}

The cdh cohomology groups $H^i_{\cdh}(X, \mathbb{Z}(j))$ do not necessarily vanish anymore when $i > 2j$, and their study in this range is essentially the main point in what follows. On the other hand, using the spectral sequence (\ref{eq:ss-cdh}) and the vanishing properties of the motivic cohomology of smooth varieties, we have the following immediate consequence:

\begin{lemma}\label{lem motivic soule-weibel}
We always have 
$$H^i_{\cdh}(X, \mathbb{Z}(j)) = 0 \,\,\,\,\,\, {\rm  for}\,\,\,\, i > j + n.$$
\end{lemma}

Indeed, if we choose a cubical hyperresolution with the property that $\dim X_p \le n -p$ for all $p$, and take $p + q > j + n$, then the motivic cohomology $H^q (X_p, \ZZ (j))$ always vanishes, because $q > j + n-p\ge j + \dim X_p$, and we apply (\ref{eqn:van-smooth-2}).

To avoid any confusion, we conclude by noting that in the singular setting we will use the terminology \emph{motivic cohomology} for a different cohomology theory $H^i_{\mot}(X, \mathbb{Z}(j))$, introduced in \cite{motivic cohomology}, that typically differs from cdh-motivic cohomology, and plays a similar role with respect to the standard $K$-theory. For this cohomology theory, the vanishing in Lemma \ref{lem motivic soule-weibel} still holds with a more involved argument, and is called ``motivic Soul\'e-Weibel vanishing"; see \cite[Theorem~1.14]{motivic cohomology}. Results on motivic cohomology parallel to those in this paper will appear in \cite{BPS}.

\subsection{Singularities of pairs and MMP}\label{scn:MMP}
Finally, in preparation for \S\ref{scn:weight2} and \S\ref{scn:chow} especially, we briefly review singularities of pairs and some results from the MMP. For a detailed introduction, please see \cite{kollar mori} or \cite{mmp}.

A \emph{pair} $(X, \Delta)$ consists of a normal variety $X$, together with a $\mathbb{Q}$-divisor $\Delta$ on $X$ such that $K_X + \Delta$ is $\mathbb{Q}$-Cartier.
If $f\colon Y \rightarrow X$ is a proper birational morphism, there is a unique $\mathbb{Q}$-divisor $\Delta_Y$ on $Y$, called the \emph{log pullback} of $\Delta$, defined by the formulas
\begin{align}\label{eqn delta_Y}
    K_Y + \Delta_Y \sim_\mathbb{Q} f^*(K_X + \Delta) \hspace{.2cm} \text{ and } \hspace{.2cm} f_* \Delta_Y = \Delta. 
\end{align}
Given a divisor $E \subset Y$, its \emph{discrepancy}, denoted $a(E, X, \Delta)$, is the negative of the coefficient of $E$ in $\Delta_Y$.

A pair $(X, \Delta)$ is called \emph{lc} (resp. \emph{klt}) if $a(E, X, \Delta) \geq -1$ (resp. $a(E, X, \Delta) > -1)$ for every birational morphism 
$f\colon Y \rightarrow X$ and every divisor $E$ on $Y$. Moreover, it is called \emph{plt} if it is lc and $a(E, X, \Delta) > -1$ for every exceptional divisor $E$ over X. An irreducible subvariety $Z \subset X$ is an \emph{lc center} if there exists a birational morphism $f\colon Y \rightarrow X$ and a divisor $E \subset Y$ such that $a(E, X, \Delta)= -1$ and $f(E) = Z$.

We say that the pair $(X, \Delta)$ is a \emph{simple normal crossing pair} if $X$ is smooth and $\Delta$ is a reduced simple normal crossing divisor.  
Given any pair $(X, \Delta)$, where $\Delta$ has coefficients between $0$ and $1$, there exists a largest open set $X^{\rm snc} \subset X$, called the \emph{snc locus} such that $(X^{\rm snc}, \Delta|_{X^{\mathrm{snc}}})$ is a simple normal crossing pair. An lc pair $(X, \Delta)$ is called  \emph{dlt} if none of the lc centers of $(X, \Delta)$ are contained in $X \setminus X^{\rm snc}$. In this case, the lc centers of $(X, \Delta)$ are in bijective correspondence with the strata of the snc divisor $\Delta|_{X^{\mathrm{snc}}}$.

A \emph{log resolution} of $(X, \Delta)$ is a resolution $f\colon Y \rightarrow X$ of $X$ such that the pair $(Y, {\rm Supp}(\mathrm{Exc}(f) + \widetilde{\Delta}))$ is a  simple normal crossing  pair, where $\widetilde{\Delta}$ is the proper transform of $\Delta$. We say that $f$ is a \emph{strong log resolution} if in addition it is an isomorphism over the locus where $(X, {\rm Supp} (\Delta))$ is an snc pair.

Suppose $(Y, \Gamma)$ is a dlt pair and $E = \Gamma^{=1}$, the sum of the components with coefficient $1$. Every stratum (or, equivalently, lc center) $S$ inherits a natural $\mathbb{Q}$-divisor $\Gamma_S = \mathrm{Diff}_S^* \Gamma$ such that $(S, \Gamma_S)$ is dlt and $K_S + \Gamma_S \sim_\mathbb{Q} \left(K_Y + \Gamma)\right|_S$.  See \cite[\S4.1]{mmp} for details. In general, $\Gamma_S$ can be complicated. However, if all of the divisors in $\Gamma$ have coefficient $1$ and if $K_Y + \Gamma$ is Cartier, then the same holds for $\Gamma_S$, and $\Gamma_S$ consists of the intersections of $S$ with those irreducible components of $\Gamma$ that do not contain $S$. This construction will be useful for transferring information about $Y$ to the strata $S$.

\smallskip

We have the well-known existence result needed in order to run the minimal model program.

\begin{theorem}\cite[Theorem $1.2$]{bchm}\label{thm bchm}
Let $(Y, \Gamma)$ be a klt pair and let $g\colon Y \rightarrow Z$ be a projective morphism of quasi-projective varieties. If $K_Y + \Gamma$ is $g$-big (for example, if $g$ is birational), we may run a $(K_Y + \Gamma)$-MMP over $Z$ ending with a log terminal model over $Z$.
\end{theorem}



Recall from Remark \ref{rmk:klt} that a variety $X$ is called \emph{of klt type} if there exists an effective $\Q$-divisor $\Delta$ on $X$ such that the pair $(X, \Delta)$ is klt. We will require the existence of a nice boundary for such a variety. The proof is inspired by \cite[Theorem $5.4$]{dFH} and discussions with C. Hacon.

\begin{lemma}\label{lem new boundary}
Suppose $X$ is a quasi-projective variety of klt type. Then there exists a $\mathbb{Q}$-divisor $\Delta$ on $X$ with smooth support away from the singular locus such that $(X, \Delta)$ is a klt pair.
\end{lemma}

\begin{proof}
Choose a $\mathbb{Q}$-divisor $\Delta_1 \geq 0$ such that $(X, \Delta_1)$ is klt, and let $m \geq 2$ be such that $B := m \Delta_1$ is integral and $m(K_X + \Delta_1)$ is Cartier. 

Let $D$ be an effective Weil divisor for which $K_X - D$ is Cartier. Take $f\colon Y \rightarrow X$ to be a log resolution of $(X, \Delta_1)$; in addition, we require that the sheaves $\mathscr{O}_X(-mD) \cdot \mathscr{O}_Y$ and $\mathscr{O}_X(B) \cdot \mathscr{O}_Y = \mathscr{O}_Y(\widetilde{B}+E)$ be invertible, where $\widetilde{B}$ is the proper transform of $B$ and $E \geq 0$ is exceptional, and that all such divisors have simple normal crossings.

Let $\mathscr{L}$ be an invertible sheaf such that $\mathscr{L} \otimes \mathscr{O}_X(-mD)$ is globally generated, where  $\mathscr{O}_X(-mD)$ is the reflexive rank one sheaf associated to the integral divisor $mD$, and let $G$ be a general element in the linear system $\{ L \in | \mathscr{L}| : L - mD \geq 0 \}$. We write $G = M + mD$, where $M$ is an effective divisor. Finally, set
\begin{align*}
    \Delta := \frac{1}{m} M.
\end{align*}
By construction, $K_X + \Delta$ is $\mathbb{Q}$-Cartier. Moreover, away from the singular locus of $X$, the sheaf 
$\mathscr{L} \otimes \mathscr{O}_X(-mD)$ is invertible and globally generated, so the restriction of $M$ is smooth by Bertini's theorem.

It remains to show that $(X, \Delta)$ is klt. By the generality of $M$, $f$ is a log resolution of $(X, \Delta)$. Since $\Delta = \frac{1}{m} M$ and $\Delta_1 = \frac{1}{m} B$, we compute
\begin{align}\label{eqn klt}
    K_Y + \widetilde{\Delta} - f^*(K_X + \Delta) 
    = \bigg( K_Y + \widetilde{\Delta}_1 - f^*(K_X + \Delta_1) \bigg) + 
    \frac{1}{m} \bigg( f^*(B-M) -  \widetilde{B} + \widetilde{M} \bigg).
\end{align}
Since $B - M = m(K_X + \Delta_1) - m(K_X + \Delta)$ is Cartier, we have $f^*(B-M) = \widetilde{B} + E - \widetilde{M}$. Thus, the second term on the right of (\ref{eqn klt}) is effective. Since $(X, \Delta_1)$ is klt, the first term on the right of (\ref{eqn klt}) has coefficients $> -1$. We conclude that the left side of (\ref{eqn klt}) has coefficients $> - 1$. Therefore, $(X, \Delta)$ is klt. 
\end{proof}

Next we recall a special type of morphism, useful in what follows (see e.g. \cite[Definition $2.5$]{moraga plt}).

\begin{definition}\label{defn plt blow-up}
A \emph{plt blow-up} of a klt pair $(X, \Delta)$ at a point $x \in X$ is a projective birational morphism $h\colon Z \rightarrow X$, such that:
\begin{enumerate}
    \item $Z$ is a quasi-projective normal variety,
    \item the exceptional locus of $h$ is an irreducible divisor $D$ whose image on $X$ is $x$,
    \item the pair $(Z, \widetilde{\Delta}+D)$ is plt, and
    \item $-D$ is a $\Q$-Cartier divisor which is ample over $X$.
\end{enumerate}
Plt blow-ups always exist by \cite[Lemma 1]{Xu}; cf. also\cite[Proposition $2.7$]{moraga plt}.

We will say a log resolution $f\colon Y \rightarrow X$ factors through a plt blow-up if locally around each singular point $x \in X$, there exists a plt blow-up $h\colon Z \rightarrow X$ of $x$ such that  $f = h \circ g$, where $g\colon Y \rightarrow Z$ is a morphism.
\end{definition}

We will at times assume that $f$ factors through a plt blow-up to ensure that any preimage of the $\mathbb{Q}$-Cartier divisor $D = h^{-1}(x)$ is divisorial. Without this assumption, some variety over $X$ arising in the MMP may have exceptional components of codimension $\geq 2$, and these are difficult to control.

We also require the following version of the connectedness theorem.

\begin{proposition}\label{prop connectness}\cite[Proposition $25$]{dFKX}
Let $(Y, \Gamma)$ be a dlt pair and let $g\colon Y \rightarrow Z$ be a projective morphism such that $g_* \mathcal{O}_Y \simeq \mathcal{O}_Z$ and $-(K_Y+\Gamma)$ is $g$-ample. Then for any (possibly non-closed) point $z \in Z$, the set of all lc centers of $(Y, \Gamma)$ which intersect $g^{-1}(z)$ is either empty or contains a unique minimal element with respect to inclusion.
\end{proposition}

\section{General results on cdh cohomology}

\subsection{Weight zero cdh cohomology}\label{scn:weight0}

In this subsection, we compute weight zero cdh cohomology over the integers. 
 As an application, in the case of rational coefficients we obtain a proof of Theorem \ref{thm:cdh-low-weight}(i).

In what follows we use the notions around the weight filtration on integral cohomology introduced in \S\ref{scn:weights-integral}.

\begin{proposition}\label{prop weight zero}
If $X$ is a projective variety, then
\begin{align*}
    H^i_{\cdh} \big( X, \mathbb{Z}(0) \big) \simeq H^i_0(X, \mathbb{Z}).
\end{align*}
If $X$ is of klt type, this vanishes for all $i > 0$.
\end{proposition}

\begin{proof}
Choose a cubical hyperresolution $\epsilon_\bullet \colon X_\bullet \rightarrow X$ of $X$. Then
\begin{align*}
   \mathbb{Z}_{\cdh}(0)(X) =& \bigg( \mathbb{Z}_{\cdh}(0)(X_0) \rightarrow \mathbb{Z}_{\cdh}(0)(X_1) \rightarrow \cdots \rightarrow \mathbb{Z}_{\cdh}(0)(X_n) \bigg).
\end{align*}
Applying Theorem \ref{thm motivic cohom}(3) to each term separately, we observe that this complex is of the form
\begin{align*}
    \bigg( H^0(X_0, \mathbb{Z}) \rightarrow H^0(X_1,\mathbb{Z}) \rightarrow \cdots \rightarrow H^0(X_n, \mathbb{Z}) \bigg).
\end{align*}
By construction, its cohomology groups $H^i_{\cdh}(X, \mathbb{Z})$ are isomorphic to $H_0^i(X, \mathbb{Z}(0))$ (see Proposition \ref{prop integral weight}). The last statement is a consequence of Theorem \ref{thm wt zero klt}.
\end{proof}

\begin{remark}
Recall that Gorenstein rational singularities are canonical, hence klt, so the result above applies. For non-Gorenstein rational singularities, see Corollary \ref{cor:cdh-weight-zero} below.
\end{remark}

The following corollary gives a more concrete interpretation for cohomology in sufficiently large degree.

\begin{corollary}\label{corollary dual complex}
Let $X$ be a normal projective variety. Let $f \colon Y \rightarrow X$ be a log resolution of $X$, which is an isomorphism away from the singular locus $X_{\sing}$ and has snc exceptional divisor $E$. If $i \geq \dim X_{\sing}+2$, then 
$$H^i_{\cdh}(X, \mathbb{Z}(0)) \simeq H^{i-1}(\mathcal{D}(E), \mathbb{Z}),$$ 
where $\mathcal{D}(E)$ is the dual complex of $E$.
\end{corollary}

\begin{proof}
Consider the weight zero part of the Mayer-Vietoris long exact sequence
\begin{align*}
    \cdots \rightarrow H^{i-1}_0&(Y, \mathbb{Z}) \oplus H^{i-1}_0(X_{\sing}, \mathbb{Z}) \rightarrow H^{i-1}_0(E, \mathbb{Z}) \\& \rightarrow H^i_0(X, \mathbb{Z}) \rightarrow H^{i}_0(Y, \mathbb{Z}) \oplus H^{i}_0(X_{\sing}, \mathbb{Z}) \rightarrow \cdots
\end{align*}
The end terms vanish whenever $i \geq \dim X_{\sing}+2$, so $H^i_0(X, \mathbb{Z}) \simeq H^{i-1}_0(E, \mathbb{Z})$. On the other hand, we saw that 
$H^{i-1}_0(E, \mathbb{Z}) \simeq H^{i-1}(\mathcal{D}(E), \mathbb{Z})$ in Lemma \ref{lem integral weight properties}(4), while $H^i_{\cdh}(X, \mathbb{Z}(0)) \simeq 
H^i_0(X, \mathbb{Z})$ by Proposition \ref{prop weight zero}.
\end{proof}

Working with rational coefficients, the situation is simpler. This concludes the proof of Theorem \ref{thm:cdh-low-weight}(i).

\begin{corollary}\label{cor:cdh-weight-zero}
If $X$ is a projective variety, then 
$$H^i_{\cdh}(X, \mathbb{Q} (0)) \simeq W_0 H^i(X, \mathbb{Q}).$$ 
These spaces vanish for all $i > 0$ if $X$ has rational singularities, and for $i = 1$ if $X$ is normal.
\end{corollary}

\begin{proof}
For the first statement, recall that $H^i_0(X, \mathbb{Z}) \otimes \Q \simeq W_0 H^i(X, \mathbb{Q})$, and use Proposition \ref{prop weight zero}.
The second follows from Corollary \ref{cor weight restrictions}, as well as the well-known fact (already used in Lemma \ref{lem integral weight properties}) that normal projective varieties have pure $H^1 (X, \Q)$.
\end{proof}

\subsection{Weight one cdh cohomology}\label{scn:weight1}
In this subsection, we compute weight one cdh-motivic cohomology over the integers. As an application, in the case of rational coefficients we obtain a proof of Theorem \ref{thm:cdh-low-weight}(ii).

Choose a cubical hyperresolution $\epsilon_\bullet \colon X_\bullet \rightarrow X$ of $X$. We obtain an expression for $\mathbb{Z}(1)^{\cdh}(X)$ as an iterated cone
\begin{align}\label{eqn wt one iterated cone}
    \mathbb{Z}(1)^{\cdh}(X) = \bigg( \mathbb{Z}(1)^{\cdh}(X_0) \rightarrow \mathbb{Z}(1)^{\cdh}(X_1) \rightarrow \cdots \rightarrow \mathbb{Z}(1)^{\cdh}(X_n) \bigg).
\end{align}

Each smooth piece $X_i$ of the hyperresolution is smooth and projective, so Theorem \ref{thm motivic cohom}(3) computes the cohomology of the complex $\mathbb{Z}(1)^{\cdh}(X_i)$. In particular, there exist distinguished triangles
\begin{align}\label{eqn weight one smooth}
    H^0(X_i, \mathscr{O}_{X_i}^*)[-1] \rightarrow \mathbb{Z}(1)^{\cdh}(X_i) \rightarrow \Pic(X_i)[-2] \xrightarrow[]{+1}
\end{align}

To consolidate this information, let us form the iterated cones
\begin{align*}
    H^0(X_\bullet, \mathscr{O}_{X_\bullet}^*) : = \bigg(H^0(X_0, \mathscr{O}_{X_0}^*) \rightarrow H^0(X_1, \mathscr{O}_{X_1}^*) \rightarrow \cdots \rightarrow H^0(X_n, \mathscr{O}_{X_n}^*) \bigg).
\end{align*}
and
\begin{align*}
    \Pic(X_\bullet) : = \bigg( \Pic(X_0) \rightarrow \Pic(X_1) \rightarrow \cdots \rightarrow \Pic(X_n) \bigg).
\end{align*}
Analogously, define $\Pic^0(X_\bullet)$ and ${\rm NS} (X_\bullet)$.

\begin{proposition}\label{prop weight one}
If $X$ is projective, there exist distinguished triangles of complexes
\begin{align*}
     H^0(X_\bullet, \mathscr{O}_{X_\bullet}^*)  [-1] \rightarrow \mathbb{Z}(1)^{\cdh}(X) \rightarrow  \Pic(X_\bullet)  [-2] \xrightarrow[]{+1}
\end{align*}
and 
\begin{align*}
    \Pic^0(X_\bullet) \rightarrow \Pic(X_\bullet) \rightarrow {\rm NS}(X_\bullet) \xrightarrow[]{+1}
\end{align*}
Moreover, there exist short exact sequences
\begin{align*}
    0 \rightarrow H^i_0(X, \mathbb{Z}) \otimes_\mathbb{Z} \mathbb{C}^* \rightarrow H^i \big( H^0(X_\bullet, \mathscr{O}_{X_\bullet}^*) \big) \rightarrow \mathrm{Tor}_1^\mathbb{Z} \big( H^{i+1}_0(X, \mathbb{Z}), \mathbb{C}^* \big) \rightarrow 0
\end{align*}
and
\begin{align*}
    0 \rightarrow H^i_1(X, \mathbb{Z}) \otimes_\mathbb{Z} \mathbb{R}/\mathbb{Z} \rightarrow H^i \big( \Pic^0(X_\bullet) \big) \rightarrow \mathrm{Tor}_1^\mathbb{Z} \big( H^{i+1}_1(X, \mathbb{Z}), \mathbb{R} / \mathbb{Z} \big) \rightarrow 0
\end{align*}
Finally, if $X$ is of klt type, then 
$$H^i_{\cdh}(X, \mathbb{Z}(1)) \simeq H^{i-2}( \Pic X_\bullet) \simeq H^{i-2}( {\rm NS}(X_\bullet) ) \,\,\,\,\,\,{\rm for ~all} \,\,\,\, i \geq 2.$$
\end{proposition}

Note that there does not seem to be a direct Hodge-theoretic interpretation of the complex ${\rm NS}(X_\bullet)$, because strictness can fail integrally.

\begin{proof}
The first distinguished triangle is a consequence of (\ref{eqn wt one iterated cone}) and (\ref{eqn weight one smooth}). The second follows by gluing together each of the short exact sequences
\begin{align*}
    0 \rightarrow \Pic^0(X_i) \rightarrow \Pic(X_i) \rightarrow {\rm NS}(X_i) \rightarrow 0.
\end{align*}
Note that the complex $H^0(X_\bullet, \mathscr{O}_{X_\bullet}^*)$ is precisely the tensor product with $\mathbb{C}^*$ of the bottom row of the $E_2$-page weight spectral 
sequence (\ref{eqn:weight-ss-integral}). The first short exact sequence follows from the universal coefficient theorem.

For the second short exact sequence, note the complex $\Pic(X_\bullet)$ is, as a complex of abelian groups, the tensor product with $\mathbb{R}/\mathbb{Z}$ of the weight one row of the $E_2$-page weight spectral sequence. We again apply the universal coefficient theorem.

Based on the above, the final claim is a consequence of Theorem \ref{thm wt zero klt}.
\end{proof}

Over the rational numbers, the situation is simpler.

\begin{corollary}\label{cor weight one rational}
If $X$ is projective, there exist long exact sequences of the form
\begin{align*}
    \cdots \rightarrow W_0 H^{i-1}(X, \mathbb{Q}) \otimes_\mathbb{Z} \mathbb{C}^* \rightarrow H^i_{\cdh}(X, \mathbb{Q}(1)) \rightarrow H^{i-2} ( \Pic(X_\bullet)_\mathbb{Q} ) \rightarrow \cdots  
\end{align*}
and
\begin{align*}
    \cdots \rightarrow \gr_1^W H^{i-1}(X, \mathbb{Q}) \otimes_\mathbb{Z} \mathbb{R}/\mathbb{Z} \rightarrow H^{i-2}( \Pic(X_\bullet)_\mathbb{Q}) \rightarrow H^{1,1} \gr_2^W H^i(X, \mathbb{Q}) \rightarrow \cdots
\end{align*}
\end{corollary}

\begin{proof}
We take tensor product of the distinguished triangles of Proposition \ref{prop weight one} with the rational numbers $\mathbb{Q}$.

The complex $H^0(X_\bullet, \mathscr{O}_{X_\bullet}^*) \otimes_\mathbb{Z} \mathbb{Q}$ is equal to $H^0(X_\bullet, \mathbb{Q}) \otimes_\mathbb{Z} \mathbb{C}^*$, the tensor product with a complex of $\mathbb{Q}$-vector spaces. Since $\mathbb{Q}$-vector spaces are flat over $\mathbb{Z}$, the cohomologies of this complex are precisely 
$$H^i(H^0(X_\bullet, \mathbb{Q})) \otimes_\mathbb{Z} \mathbb{C}^* = W_0 H^i(X, \mathbb{Q}) \otimes_\mathbb{Z} \mathbb{C}^*.$$

Similarly, we have an isomorphism of complexes 
$$\Pic^0(X_\bullet) \otimes_\mathbb{Z} \mathbb{Q}\simeq  H^1(X_\bullet, \mathbb{Q}) \otimes_\mathbb{Z} \mathbb{R}/\mathbb{Z}.$$ 
It follows that 
\begin{align*}
    H^{i-2}(\Pic^0(X_\bullet) \otimes_\mathbb{Z} \mathbb{Q}) \simeq H^{i-2}(H^1(X_\bullet, \mathbb{Q})) \otimes_\mathbb{Z} \mathbb{R}/\mathbb{Z} \simeq \gr_1^W H^{i-1}(X, \mathbb{Q}) \otimes_\mathbb{Z} \mathbb{R}/\mathbb{Z}.
\end{align*}

Finally, we note that ${\rm NS} (X_\bullet)_\mathbb{Q} = H^{1,1}(X_\bullet, \mathbb{Q})$. Using the strictness of the Hodge filtration, the cohomologies of this complex are precisely
\begin{align*}
    H^{i-2} ( {\rm NS} (X_\bullet)_\mathbb{Q} ) \simeq H^{1,1} \gr_2^W H^i(X, \mathbb{Q}).
\end{align*}
\end{proof}

Let us list a few immediate consequences.

\begin{corollary}\label{cor cdh 1 vanishing}
If $W_1 H^{i-1}(X, \mathbb{Q}) = \gr_2^W H^i(X, \mathbb{Q}) = 0$, then $H^i_{\cdh}(X, \mathbb{Q}(1)) = 0$.
\end{corollary}

\begin{corollary}\label{cor H^3(1)}
If $X$ has rational singularities and $i \geq 3$, then 
\begin{align*}
    H_{\cdh}^{i}(X, \mathbb{Q}(1)) \simeq H^{i-2}(\Pic X_\bullet)_\mathbb{Q} \simeq \gr_2^W H^{i}(X, \mathbb{Q}).
\end{align*}
If moreover $X$ has pre-$1$-rational singularities, then $H_{\cdh}^{i}(X, \mathbb{Q}(1)) = 0$.
\end{corollary}

\begin{proof}
We apply Corollary \ref{cor weight one rational}. It remains to observe that since $X$ has rational singularities, Corollary \ref{cor  weight restrictions} implies each $H^j(X, \mathbb{Q})$ has weights $\geq \min \{ j, 2 \}$. Moreover, $W_2 H^i(X, \mathbb{Q})$ is of Hodge-type $(1,1)$ for $i \ge 3$ by Lemma \ref{lem hodge du bois numbers}. Finally, if $X$ has pre-$1$-rational singularities, Corollary \ref{cor  weight restrictions} applies again, to say that $W_2 H^i(X, \mathbb{Q}) = 0$.
\end{proof}

This completes the proof of Theorem \ref{thm:cdh-low-weight}(ii).

\subsection{The weight two cdh cohomology $H^{n+2}_{\rm cdh} (X, \Q (2))$}\label{scn:weight2}
In weight two we will only address the ``borderline" cdh cohomology group $H^{n+2}_{\rm cdh} (X, \Q (2))$ (and its integral version), which is the key case for computing $KH_{-n+2} (X)$, and only for isolated singularities. This is significantly more difficult than the weight zero and one cases, as (higher) Chow groups of higher codimension cycles come into play. Beyond this case it seems necessary to appeal to deep conjectures on Chow groups.

Concretely, the goal of this section and the next is to prove the following Theorem \ref{thm:weight2} in the Introduction. This states that for  a projective variety $X$ of dimension $n \ge 3$ with isolated klt type singularities we have 
$$H^{n+2}_{\rm cdh} (X, \Q (2)) = 0,$$
with some improvements in the case of threefolds: the same result holds integrally when $X$ is a local complete intersection, and it holds 
more generally for rational singularities if one assumes Bloch's conjecture \ref{conj:Bloch} on $0$-cycles on smooth surfaces.

\begin{example}\label{example kummer threefold}
There exist klt threefolds with isolated (but not local complete intersection) singularities such that $H^5_{\cdh}(X,\mathbb{Z}(2)) \neq 0$.
We will in fact show in Example \ref{example kummer threefold details} that this is the case for the Kummer threefold $X = A/\iota$, where $A$ is an abelian threefold and $\iota: x \mapsto -x$ is the negation involution.
Thus we cannot hope to improve the statement of Theorem \ref{thm:weight2} to the integral setting without extra hypotheses. 
\end{example}

Let then $X$ be a complex projective variety of dimension $n\ge 3$, of klt type and with isolated singularities. For the notions that follow, please refer
to \S\ref{scn:MMP}. 
Choose a boundary $\Delta$ on $X$ such that $(X, \Delta)$ is klt and $\mathrm{Supp}(\Delta)$ is smooth away from the singular locus $Z_0$ of $X$; this exists by Lemma \ref{lem new boundary}. We can then say equivalently that $(X, \Delta)$ has \emph{isolated singularities}.
Take a strong log resolution $f\colon Y \rightarrow X$, i.e. such that $f$ is an isomorphism away from $Z_0$ and $E \cup \mathrm{Supp}(\widetilde{\Delta})$ has simple normal crossings, where $E = \cup E_a$ is the exceptional divisor and $\widetilde{\Delta}$ is the proper transform of $\Delta$.
 Later on, we will impose another technical condition, namely that $f$ factor through a plt blow-up (see Definition \ref{defn plt blow-up}). Since such a blow-up always exists, there will be no loss of generality in assuming it. This is in any case an issue only when $X$ is not $\Q$-factorial.

Denote $E^{[m]} = \amalg ( E_{a_1} \cap \cdots \cap E_{a_m})$ and let $Z_m \subset Z_0$ be the image of $E^{[m+1]}$. The standard algorithm in \cite{GNPP} provides a semi-simplicial hyperresolution of $X$ given by 
\begin{align*}
      \cdots E^{[2]} \amalg Z_2
    \mathrel{\substack{\rightarrow\\[-0.6ex]\rightarrow\\[-0.6ex]\rightarrow}}
    E^{[1]} \amalg Z_1 \rightrightarrows Y \amalg Z_0.
\end{align*}
See \cite[Example $2.4$]{abw} for details. Set $X_0 = Y \amalg S_0$ and $X_p = E^{[p]} \amalg Z_p$ for $p \geq 1$.

We analyze $H^{n+2}_{\cdh}(X, \mathbb{Z}(2))$ via the descent spectral sequence 
$$E_1^{p,q} = H^q(X_p, \mathbb{Z}(2)) \Rightarrow H_{\cdh}^{p+q}(X, \mathbb{Z}(2)),$$
associated to this hyperresolution.

Since the motivic cohomology $H^i (W, \ZZ(j))$ of a smooth variety $W$ vanishes for $i > 2j$ (see \S\ref{scn:cdh}), the only terms $E_1^{p,q}$ with $p+q= n+2$ which are possibly nonzero are those with $p = n-2$, $n-1$, $n$. The relevant differentials at the $E_1$ page are:

\[
\begin{tikzcd}
H^4(X_{n-3}, \ZZ(2)) \arrow[r, "\alpha"] & H^4(X_{n-2}, \ZZ (2)) \arrow[r] & 0 \\
H^3(X_{n-2}, \ZZ(2)) \arrow[r, "\beta"] & H^3(X_{n-1}, \ZZ(2)) \arrow[r] & 0 \\
H^2(X_{n-1}, \ZZ(2)) \arrow[r, "\gamma"] & H^2(X_n, \ZZ(2)) \arrow[r] & 0 \\
\end{tikzcd}
\]
Note that the zeros on the right are due to another standard vanishing for motivic cohomology mentioned in (\ref{eqn:van-smooth-2}), namely $H^i (W, \ZZ (j)) = 0$ for $i > j + \dim W$, given that $\dim X_i = n- i$.

The goal is to show that these three $E_1$-terms die at the $E_2$-page, under suitable assumptions. This is the content of the following Lemmas \ref{lem gamma}, \ref{lem beta}, and \ref{lem alpha}, addressing the maps $\gamma$, $\beta$ and $\alpha$ (this is the order of difficulty), after translating this in terms of higher Chow groups according to the identification
$$H^i(Y, \mathbb{Z}(j)) \simeq \CH^j(Y, 2j-i)$$ 
described in see \S\ref{scn:cdh}. 
For simplicity, we shall denote by $T_i$ the threefold strata of $E$, $S_j$ the surface strata, $C_k$ the curve strata, and $P_\ell$ the point strata. We do not index them according to which exceptional components they belong to, as this will not be relevant. Recall that we are assuming $n \ge 3$.

\begin{lemma}\label{lem gamma}
The map
$$\gamma \colon \bigoplus_k CH^2(C_k,2) \oplus CH^2(Z_{n-1}, 2)\rightarrow \bigoplus_\ell CH^2(P_\ell, 2)$$ 
is surjective if $W_0 H^n(X, \mathbb{Q}) = 0$, so in particular if $X$ has rational singularities.
\end{lemma}

\begin{proof}
The target is a direct sum of factors $CH^2(\operatorname{Spec}\mathbb{C}, 2)$, one for each point stratum $P_\ell$. The terms $CH^2(C_k, 2)$ in the domain can be more complicated, but the structure morphism $C_k \rightarrow \operatorname{Spec}\mathbb{C}$ induces a map
\begin{align*}
    CH^2(\operatorname{Spec}\mathbb{C}, 2) \rightarrow CH^2(C_k, 2)
\end{align*}

Consider the composition
\begin{align*}
    \bigoplus_{k} CH^2(\operatorname{Spec}\mathbb{C}, 2) \oplus CH^2(Z_{n-1}, 2) \rightarrow \bigoplus_{k} CH^2(C_{k},2) \oplus CH^2(Z_{n-1}, 2) \xrightarrow[]{\gamma} \bigoplus_{\ell} CH^2(P_{\ell}, 2)
\end{align*}
It can be identified with the tensor product  over $\Q$ of the map $\theta: H^0(X_{n-1}, \mathbb{Q}) \rightarrow H^0(X_n,\mathbb{Q})$ with $CH^2(\operatorname{Spec}\mathbb{C}, 2)$. To make sense of this tensor product, we are using the fact that $CH^2(\operatorname{Spec}\mathbb{C}, 2) = K_2(\mathbb{C})$ is uniquely-divisible by \cite[Theorems III$.6.1$, III.$6.4$]{k-book}, hence it naturally inherits the structure of a $\mathbb{Q}$-vector space. Since $W_0 H^n(X, \mathbb{Q}) = 0$, (\ref{eqn:weight-hyper}) shows $\theta$ is surjective. But then the composition is surjective, hence so is $\gamma$.
\end{proof}

For what follows, note that since $Z_p$ are unions of points, we have 
$$CH^2(Z_p, 1) = CH^2(Z_p, 0)  = 0,$$
hence they do not appear in the statements. The second vanishing is clear, and the first follows for instance from (\ref{eqn:van-smooth-2}).

\begin{lemma}\label{lem beta}
The map
$$\beta \colon \bigoplus_j CH^2(S_j, 1) \rightarrow \bigoplus_k CH^2(C_k, 1)$$ 
is surjective if ${\rm gr}_1^W H^{n} (X, \Q) = {\rm gr}_2^W H^{n+1} (X, \Q) = 0$, so in particular if $X$ has rational singularities.
\end{lemma}

\begin{proof}
The product operation for higher Chow groups induces a commutative diagram
\[\begin{tikzcd}
    \bigoplus_j CH^1(S_j,0) \otimes CH^1(\operatorname{Spec}\mathbb{C}, 1) \ar[r, "\psi"] \ar[d] & \bigoplus_{k} CH^1(C_{k},0) \otimes CH^1(\operatorname{Spec}\mathbb{C}, 1) \ar[d] \\
        \bigoplus_j CH^2(S_j,1) \ar[r, "\beta"] & \bigoplus_{k} CH^2(C_{k},1)
    \end{tikzcd}\]
We claim the top arrow $\psi$ is surjective. Indeed, it can be identified with the tensor product of the natural map 
$$\varphi \colon \bigoplus_j \Pic(S_j) \rightarrow \bigoplus_{k} \Pic(C_{k})$$ 
with $CH^1(\operatorname{Spec}\mathbb{C}, 1) = K_1(\mathbb{C}) = \mathbb{C}^*$. Since the tensor product is right exact, we have $\coker \psi \simeq (\coker \varphi) \otimes_\mathbb{Z} \mathbb{C}^*$. But, under our assumptions, Corollary \ref{cor weight one rational} implies that 
$$\coker \varphi_\mathbb{Q} \simeq H^{n-1}( \Pic(X_\bullet)_\mathbb{Q}) = 0.$$ 
(Note that if $X$ has rational singularities, the vanishing assumptions are satisfied by Corollary \ref{cor weight restrictions}.) Hence we are assuming that $\coker \varphi$ is torsion. The tensor product of a torsion group with the divisible group $\mathbb{C}^*$ is zero, so the claim follows.

The right arrow is surjective as well by the following Lemma \ref{lemma CH^2(1)}. This implies that the composition is surjective, hence so is $\beta$.
\end{proof}

Our proof of Lemma \ref{lem beta} requires the following general fact about higher Chow groups on curves.

\begin{lemma}[{\cite[\S 3.5]{decomposable}}]\label{lemma CH^2(1)}
Let $C$ be a smooth complex projective curve. Then the product operation for higher Chow groups is a surjection
\begin{align*}
    m\colon \Pic(C) \otimes_{\mathbb{Z}} \mathbb{C}^* \rightarrow CH^2(C,1).
\end{align*}
\end{lemma}

\smallskip

The final necessary result is the hardest; this is in fact the most involved part of the paper. We will separate this into the cases $n = 3$ (when the threefold stratum is the whole of $Y$), and $n \ge 4$. Since these are statements that may be of general interest, involving Chow groups, we will address them independently in the next section.

\begin{lemma}\label{lem alpha}
(i) (= Theorem \ref{thm:Chow-threefolds}) If $n = 3$, then 
$$\alpha_\Q: CH^2(Y)_\mathbb{Q} \rightarrow \bigoplus_j CH^2(S_j)_\mathbb{Q}$$
is surjective if $X$ has isolated klt type singularities. The same holds for isolated rational singularities, if we assume Bloch's conjecture on $0$-cycles on smooth projective surfaces. Moreover, the result holds integrally, i.e. $\alpha$ itself is surjective, if we assume in addition that $X$ is a local complete intersection.

\noindent
(ii) (= Theorem \ref{thm:Chow-higher}) Suppose $f$ factors through a plt blow-up (see Definition \ref{defn plt blow-up}). If $n \geq 4$, then 
$$\alpha_\Q: \bigoplus_i CH^2(T_i)_\mathbb{Q} \rightarrow \bigoplus_j CH^2(S_j)_\mathbb{Q}$$
 is surjective if $X$ has isolated klt singularities. 
\end{lemma}

Putting together all of these lemmas gives the proof of Theorem \ref{thm:weight2}.

\subsection{Restriction maps on Chow groups of codimension two cycles}\label{scn:chow}
We start with the case of threefolds, which needs a separate treatment based on Bloch's conjecture and Hodge-theoretic methods.

\begin{theorem}\label{thm:Chow-threefolds}
Let $X$ be a projective threefold with isolated singularities of klt type, and let $f \colon Y \to X$ be a log resolution which is an isomorphism away from the singular locus of $X$. 
If $S_j$ are the surface components of the exceptional locus of $f$, then the restriction map
$$\alpha_\Q: CH^2(Y)_\mathbb{Q} \rightarrow \bigoplus_j CH^2(S_j)_\mathbb{Q}$$
is surjective.  The same holds for isolated rational singularities, if we assume Bloch's conjecture on $0$-cycles on smooth projective surfaces.\footnote{We recall that this is an issue only when $X$ is not Gorenstein, as in the Gorenstein case the notions of klt, canonical and rational singularities are all equivalent.} Moreover, the result holds integrally, i.e. $\alpha$ itself is surjective, if we assume in addition that $X$ is a local complete intersection.
\end{theorem}
\begin{proof}
Note to begin with that the irreducible components $S_j$ of the exceptional divisor $E$ are smooth projective surfaces with $h^{2,0}(S_j) = 0$. Indeed, since $X$ has rational singularities, we have 
$$R^2 f_* \mathscr{O}_Y \simeq H^2(E, \mathscr{O}_E) \simeq \bigoplus_j H^2(S_j, \mathscr{O}_{S_j}),$$
the latter since $H^2$ of the structure sheaf of an intersection of $S_j$'s vanishes. For such surfaces, Bloch's conjecture (see Conjecture \ref{conj:Bloch} in the Appendix) asserts that the Abel-Jacobi map 
\begin{equation}\label{eqn:bloch}
CH_0(S_i)_{\hom} \to \Alb(S_i)
\end{equation}
is an isomorphism; this is known for all surfaces that are not of general type. If $X$ is of klt type, then $E$ is rationally chain connected by \cite{hm}, hence each $S_i$ is birationally ruled, so the conjecture holds (and is rather straightforward). If $X$ is only assumed to have rational singularities, some $S_j$ may be of general type, and in this case we will assume the conjecture. Either way, from now on we assume that (\ref{eqn:bloch}) is an isomorphism for all $j$.

Consider now the commutative diagram 
\begin{equation}\label{eqn gr3W-H4}
\begin{tikzcd}
  CH^2(Y)_\mathbb{Q} \ar[r, "\alpha_{\Q}"] \ar[d, "\cl"] &
  \bigoplus_{i} CH^2(S_i)_{\mathbb{Q}} \ar[d, "\cl"] \\
  H^{2,2}(Y)_\Q \ar[r] &
  \bigoplus_{i} H^{2,2}(S_i)_\Q
\end{tikzcd}
\end{equation}
where the vertical maps are cycle class maps; here $H^{2,2} (Y)_\Q$ is notation for $H^{2,2} (Y) \cap H^4 (Y, \Q)$, and same for the $S_j$.

Since $X$ has isolated singularities, $H^5(X, \mathbb{Q})$ is a pure Hodge structure (see Remark \ref{rmk:purity-sing}), hence in particular $\gr_4^W H^5(X, \mathbb{Q}) = 0$.\footnote{Since $X$ has rational singularities, Corollary \ref{cor weight restrictions} says that $\gr_4^W H^5(X, \mathbb{Q}) = 0$ even without the isolated singularities assumption, but this will be needed in another part of the argument.}
Computing this in terms of the standard hyperresolution for isolated singularities described in \S\ref{scn:weight2}, according to (\ref{eqn:weight-hyper}) we have that the pullback map 
$$H^4(Y, \mathbb{Q}) \rightarrow \bigoplus_j H^4(S_j, \mathbb{Q})$$ 
is surjective, and note that on the surfaces $S_j$ the $H^{2,2}$ space fills up the entire $H^4$. Since this is a map of Hodge structures, we conclude that the bottom arrow of (\ref{eqn gr3W-H4}) is surjective. Moreover, the Hodge conjecture, which is known for threefolds, implies the vertical left arrow in the diagram is surjective as well, and of course so is the right vertical one.

By the Snake Lemma, it suffices then to show the surjectivity of the restriction of $\alpha$ to the null-homologous cycles
\begin{align*}
    \alpha \colon CH^2(Y)_{\hom} \rightarrow \bigoplus_i CH^2(S_i)_{\hom}.
\end{align*}
The only null-homologous one-cycles we understand  on $Y$ are those supported on the exceptional divisor. Therefore we aim to show surjectivity of the composition
\begin{align*}
    M\colon \bigoplus_i CH_1(S_i)_{\hom} \overset{ (a_{i*})_i}{\longrightarrow} CH_1(Y)_{\hom} = CH^2(Y)_{\hom} \overset{\oplus a_j^*}{\longrightarrow}  \bigoplus_j CH^2(S_j)_{\hom}
\end{align*}
where $a_j\colon S_j \rightarrow Y$ denotes the inclusion. 
We will prove the surjectivity of $M$ by exploiting the Hodge theory of $Y$.

Let us make the identifications $CH_1(S_i)_{\hom} \simeq \Pic^0(S_i)$, viewed as a quotient of $H^1(S_i, \mathbb{Q})$, and $CH^2(S_j)_{\hom} \simeq \Alb(S_j)$ by Bloch's conjecture as described above, viewed as a quotient of $H^3(S_j, \mathbb{Q})$. The block-matrix $M$ lifts to a map of cohomology groups of the form
\begin{align*}
     \bigoplus_i H^1(S_i, \mathbb{Q}) \xrightarrow{(a_{i*})_i} H^3(Y, \mathbb{Q}) \xrightarrow{\oplus a_j^*} \bigoplus_j H^3(S_j, \mathbb{Q}), 
\end{align*}
where $a_{j*}\colon H^1(S_j, \mathbb{Q}) \rightarrow H^3(Y, \mathbb{Q})$ is the Gysin pushforward and $a_j^*\colon H^3(Y, \mathbb{Q}) \rightarrow H^3(S_j, \mathbb{Q})$ is the pullback. 
It suffices to show the surjectivity (equivalently, bijectivity) of this composition.

Since these two types of maps are Poincaré dual to each other, this is equivalent to the non-degeneracy of the Gysin pushforward map composed with the intersection 
product on $Y$:
\begin{align*}
    \bigoplus_j H^1(S_j, \mathbb{Q}) \otimes \bigoplus_i H^1(S_i, \mathbb{Q}) \xrightarrow{(a_{i*}) \oplus (a_{j*})} H^3(Y, \mathbb{Q}) \otimes H^3(Y, \mathbb{Q}) \rightarrow \mathbb{Q}.
\end{align*}
Similar to an earlier argument in the proof, since $X$ has isolated singularities, $H^4(X, \mathbb{Q})$ is a pure Hodge structure by Remark \ref{rmk:purity-sing}, hence in particular $\gr_3^W H^4(X, \mathbb{Q}) = 0$.\footnote{Again, this could also be deduced from Corollary \ref{cor weight restrictions} by the rational singularities assumption, even without assuming that the singularities are isolated.}
Writing this in terms of the standard hyperresolution, from (\ref{eqn:weight-hyper}) we deduce that the pullback map 
$$\oplus a_j^* \colon  H^3(Y, \mathbb{Q}) \rightarrow \bigoplus_i H^3(S_j, \mathbb{Q})$$ 
is surjective, hence its dual 
$$(a_{j*{}})_j: \bigoplus_j H^1(S_j, \mathbb{Q}) \rightarrow H^3(Y, \mathbb{Q})$$ 
is injective.  We next show that the image $\im ( (a_{j*})_j )$ of this last map lies in the primitive cohomology $H^3(Y, \mathbb{Q})^{\mathrm{prim}}$, taken with respect to an ample class $h \in H^2(Y, \mathbb{Q})$. Since the intersection product is a polarization on $H^3(Y, \mathbb{Q})^{\mathrm{prim}}$, it then follows that it restricts to a polarization on $\im ( (a_{j*})_j)$, hence it is non-degenerate as claimed.

To this end, fix an index $j$ and let $x \in H^1(S_j, \mathbb{Q})$. We prove that $a_{j*} x \cup h = 0$, hence $a_{j*} x \in H^3(Y, \mathbb{Q})^{\mathrm{prim}}$. For an arbitrary class $y \in H^1(Y, \mathbb{Q})$, the projection formula implies
\begin{align*}
    (a_{j*} x \cup h \cup y)_Y = (x \cup a_j^*h \cup a_j^*y)_{S_j}
\end{align*}
This is zero because, we claim, $a_j^*\colon H^1(Y, \mathbb{Q}) \rightarrow H^1(S_j, \mathbb{Q})$ is the zero map. To show this, consider the commutative diagram
\[
\begin{tikzcd}[row sep=3.2em, column sep=3.6em]   
H^1(X, \mathbb{Q}) \arrow[r] \arrow[d, "f^*"]  &  H^1(X_{\sing}, \mathbb{Q}) \arrow[d]  = 0 & \\
H^1(Y, \mathbb{Q}) \arrow[r]   &   H^1(E, \mathbb{Q}) \arrow[r]   & H^1(S_j, \mathbb{Q}) 
\end{tikzcd}
\]
where $a_j^*$ is the composition of the bottom horizontal maps. 
Note first that $H^1(X_{\sing}, \mathbb{Q}) = 0$ because $X$ has isolated singularities. On the other hand, the pullback $f^*\colon H^1(X, \mathbb{Q}) \rightarrow H^1(Y, \mathbb{Q})$ is an isomorphism, since $X$ has rational singularities. Indeed, in this case $H^1(X, \mathbb{Q})$ is pure of weight one, and there is an isomorphism 
$$H^{0,1}(X) \simeq H^1(X, \mathscr{O}_X) \simeq H^1(Y, \mathscr{O}_Y) \simeq H^{0,1}(Y),$$
Here a priori $H^{0, 1} (X)$ is the Hodge-Du Bois space $\HH^1 (X, \DB_X^0)$, but $\DB_X^0 \simeq  \mathscr{O}_X$ for rational (and more generally Du Bois) singularities, while the isomorphism in the middle also uses the rational singularities hypothesis.  The vanishing we want follows then from commutativity, and this concludes the proof in the general setting. We treat the last statement for local complete intersections in Proposition \ref{prop:threefold-lci} below.
\end{proof}

\begin{example}\label{example kummer threefold details}
We can now provide details regarding Example \ref{example kummer threefold}. The phenomenon we highlight is the only fact in the proof of the Theorem \ref{thm:weight2} that fails over $\ZZ$.

Let $X = A/\iota$, where $A$ is an abelian threefold and $\iota: x \mapsto -x$ is the negation involution. The cohomology group $H^5(X, \mathbb{Z}) \simeq (\ZZ/ 2 \ZZ)^7$ is pure of weight $4$ by Proposition \ref{proposition kummer example} in the Appendix. But this means
\begin{align*}
    0 \neq \gr_4^W H^5(X, \mathbb{Z}) \simeq \coker \big( H^4(Y, \mathbb{Z}) \rightarrow \oplus_j H^4(S_j, \mathbb{Z}) \big),
\end{align*}
so the composition in (\ref{eqn gr3W-H4}) is not surjective when we consider integral coefficients. It follows that the map
$$\alpha\colon CH^2(Y) \rightarrow \bigoplus_j CH^2(S_j)$$
 in the proof of Theorem \ref{thm:Chow-threefolds} is not surjective either. Tracing back through the proof of that theorem, we find $H^5_{\cdh}(X,\mathbb{Z}(2)) \simeq \coker \alpha \neq 0$.
\end{example}

\begin{proposition}\label{prop:threefold-lci}
Suppose $X$ is a projective threefold with isolated, rational,  local complete intersection singularities. Then $H^5_{\cdh}(X, \mathbb{Z}(2)) = 0$.
\end{proposition}
\begin{proof}
As noted above, the only extra fact that we need to show is the surjectivity of the composition of the cycle class map $\cl\colon CH^2(Y) \rightarrow H^4(Y, \mathbb{Z})$ with the pullback $H^4(Y, \mathbb{Z}) \rightarrow \oplus_j H^4(S_j, \mathbb{Z})$, in the notation of Theorem \ref{thm:Chow-threefolds}.

Since $X$ is LCI, its links $L_x$ at singular points are $1$-connected by \cite[Korollar $1.3$]{hamm}, so $H_1(L_x, \mathbb{Z}) = 0$. They are compact $5$-manifolds, so by Poincar\'e duality 
\begin{align*}
    0 = H_1(L_x, \mathbb{Z}) = H^4(L_x, \mathbb{Z}) 
\end{align*}
We denote by $U$ the regular locus of $X$. In view of the above, and excision, we deduce
$$ H^5(X, U; \mathbb{Z}) = 0,$$ 
and therefore the restriction map $H^5(X, \mathbb{Z}) \rightarrow H^5(U, \mathbb{Z})$ is injective.

Consider the Mayer-Vietoris long exact sequence
$$
\begin{tikzcd}
\cdots \arrow[r]
  &  H^4(Y, \mathbb{Z}) \arrow[r] & H^4(E, \mathbb{Z}) \arrow[r] 
  & H^5(X, \mathbb{Z}) \arrow[r] \arrow[d, hook]
  & H^5(Y, \mathbb{Z}) \arrow[r] \arrow[d]
  & \ldots \\
 {} & {} & {}
  & H^{5}(U;\mathbb Z) \arrow[r,equal]
  & H^{5}(U;\mathbb Z)
  & {}
\end{tikzcd}
$$
The bottom row views $U$ as an open subvariety of both $X$ and $Y$. Using the injectivity shown above, the commutative square implies that the pullback map $H^5(X, \mathbb{Z}) \rightarrow H^5(Y, \mathbb{Z})$ is injective, hence the restriction 
$$H^4(Y, \mathbb{Z}) \rightarrow H^4(E, \mathbb{Z}) \simeq \bigoplus_j H^4(S_j, \mathbb{Z})$$ 
is surjective. We need to upgrade this to surjectivity at the level of the Chow groups of $Y$.

For this, consider the diagram 
$$
\begin{tikzcd}
 & {} & H^{4}(Y,U;\mathbb Z) \arrow[d] &  {} & {} \\
\cdots \arrow[r]
  & H^{4}(X;\mathbb Z) \arrow[r] \arrow[d,two heads]
  & H^{4}(Y;\mathbb Z) \arrow[r] \arrow[d]
  & H^4(E, \mathbb{Z}) = \bigoplus_i H^{4}(S_i;\mathbb Z) \arrow[r]
  & 0 \\
 {} 
  & H^{4}(U;\mathbb Z) \arrow[r, equal]
  & H^{4}(U;\mathbb Z)
  & {} & {}
\end{tikzcd}
$$
The row comes from Mayer-Vietoris and the column is the long exact sequence of the pair $(Y, U)$. The map $H^4(X, \mathbb{Z}) \rightarrow H^4(U, \mathbb{Z})$ is surjective because $H^5(X, U; \mathbb{Z}) = 0$.

A diagram chase shows that any class in $\oplus_j H^4(S_j, \mathbb{Z})$ can be lifted to a class in $H^4(Y,U; \mathbb{Z})$. Let us analyze this last group. Poincaré-Lefschetz duality \cite[Theorem B.$28$]{peters steenbrink} implies $H^4(Y, U; \mathbb{Z}) \simeq H_2(E, \mathbb{Z})$. There is a homology spectral sequence 
$$E^1_{p,q} = H_q(E^{[p+1]}, \mathbb{Z}) \Rightarrow H_{p+q}(E, \mathbb{Z}),$$ 
whose $E_1$-page is pictured:  
$$
\begin{tikzcd}[row sep=large, column sep=large]
\displaystyle\bigoplus_j H_{2}(S_j,\mathbb Z) &
\displaystyle\bigoplus_k H_{2}(C_k,\mathbb Z) \arrow[l] &
0 \arrow[l] \\
\displaystyle\bigoplus_j H_{1}(S_j,\mathbb Z) &
\displaystyle\bigoplus_k H_{1}(C_k,\mathbb Z) \arrow[l,"a"] &
 0 \arrow[l] \\
\displaystyle\bigoplus_j H_{0}(S_j,\mathbb Z) &
\displaystyle\bigoplus_k H_{0}(C_k,\mathbb Z) \arrow[l] &
\displaystyle\bigoplus_{\ell} H_{0}(P_{\ell},\mathbb Z) \arrow[l,"b"]
\end{tikzcd}
$$
We claim $\ker b = \ker a = 0$. Indeed, since $W_0 H^3(X, \mathbb{Q}) = 0$, the map $\oplus_k H^0(C_k, \mathbb{Q}) \rightarrow \oplus_\ell H^0(P_\ell, \mathbb{Q})$ is surjective. Hence, its dual is injective. This says that $\ker b$ is torsion. But $\ker b$ is a subgroup of a torsion-free group, so it must vanish.

Similarly, since $\gr_1^W H^3(X, \mathbb{Q}) = 0$, the map $\oplus_j H^1(S_j, \mathbb{Q}) \rightarrow \oplus_k H^1(C_k, \mathbb{Q})$ is surjective. Hence, its dual is injective. This says that $\ker a$ is torsion. But $\ker a$ is a subgroup of a torsion-free group, so it must vanish.

Therefore, 
\begin{align*}
    H_2(E, \mathbb{Z}) = \coker \bigg(\bigoplus_k H_{2}(C_k,\mathbb Z) \rightarrow \bigoplus_j H_{2}(S_j,\mathbb Z) \bigg).
\end{align*}

In particular, using the compatibility between the Poincar\'e-type dualities we used, there is a natural surjection $\oplus_j H^2(S_j, \mathbb{Z}) \rightarrow H^4(Y,U; \mathbb{Z})$. From this we deduce the surjectivity the composition $\oplus_j H^2(S_j, \mathbb{Z}) \rightarrow \oplus_j H^4(S_j, \mathbb{Z})$. Note that, as in the proof of Theorem \ref{thm:Chow-threefolds},  the surfaces $S_j$ are birationally ruled since LCI and rational implies klt.
Therefore $H^2(S_j, \mathbb{Z})$ is generated by classes of curves. The composition factors as
\begin{align*}
    \bigoplus_j H^2(S_j, \mathbb{Z}) = \bigoplus_j CH^1(S_j) \rightarrow CH^2(Y) \rightarrow \bigoplus_j H^4(S_j, \mathbb{Z}).
\end{align*}
In particular, $CH^2(Y) \rightarrow \oplus_j H^4(S_j, \mathbb{Z})$ is surjective.
\end{proof}

\smallskip

We now move on to the case $n\ge 4$, where the picture is somewhat different, as the threefold strata are now proper subvarieties. Our treatment is self-contained, but it borrows ideas from the forthcoming work \cite{burke dual complex} of the first author, which develops these techniques more systematically. The approach uses deep results from the minimal model program and is inspired in part by the analysis of the dual complex of singularities in \cite[Theorems $3$, $19$]{dFKX}.

\begin{theorem}\label{thm:Chow-higher}
Let $(X, \Delta)$ be a projective klt pair of dimension $n \geq 4$ with isolated singularities, and let $f: Y \rightarrow X$ be a strong log resolution of 
$(X, \Delta)$. Suppose $f$ factors though a plt blow-up (see Definition \ref{defn plt blow-up}). If $T_i$ are the threefold components, and $S_j$ the surface components of the exceptional locus of $f$, then the restriction map
$$\alpha_\Q\colon \bigoplus_i CH^2(T_i)_\mathbb{Q} \rightarrow \bigoplus_j CH^2(S_j)_\mathbb{Q}$$
is surjective. 
\end{theorem}
\begin{proof}
Initially, this proceeds as the proof of Theorem \ref{thm:Chow-threefolds}.
Consider the diagram of cycle class maps
\begin{equation}\label{eqn class surj}
\begin{tikzcd}
  \bigoplus_i CH^2(T_i)_\mathbb{Q} \ar[r, "\alpha_\Q"] \ar[d, "\cl"] &
  \bigoplus_{j} CH^2(S_j)_{\mathbb{Q}} \ar[d, "\cl"] \\
  \bigoplus_i H^{2,2}(T_i)_\mathbb{Q} \ar[r] &
  \bigoplus_{j} H^{2,2}(S_j)_\mathbb{Q}
\end{tikzcd}
\end{equation}
Since $X$ has isolated singularities, we have $\gr_4^W H^{n+2}(X, \mathbb{Q}) = 0$ by Remark \ref{rmk:purity-sing}. Therefore the pullback map of Hodge structures
$$\bigoplus_i H^4(T_i, \mathbb{Q}) \rightarrow \bigoplus_j H^4(S_j, \mathbb{Q})$$ 
is surjective, hence so is the bottom arrow of the diagram. The Hodge conjecture  for threefolds implies that  the left arrow in the diagram is surjective as well. By commutativity, so is the composition $\cl \circ \alpha_\Q$.

It therefore suffices to prove the surjectivity of the restriction of $\alpha$ to the null-homologous parts
\begin{align*}
    \alpha \colon \bigoplus_i CH^2(T_i)_{\hom} \rightarrow \bigoplus_j CH^2(S_j)_{\hom}
\end{align*}

Consider the commutative diagram
\[
\begin{tikzcd}[row sep=1.3em, column sep=4.2em]
0 \arrow[d] & 
0 \arrow[d] &
0 \arrow[d] \\
\displaystyle\bigoplus_m F^{2}CH^{2}(F_m) \arrow[r] \arrow[d] &
\displaystyle\bigoplus_i F^{2}CH^{2}(T_i) \arrow[r, "a"] \arrow[d] &
\displaystyle\bigoplus_j F^{2}CH^{2}(S_j) \arrow[d] \\
\displaystyle\bigoplus_m CH^{2}(F_m)_{\mathrm{hom}} \arrow[r] \arrow[d, "AJ_F"] &
\displaystyle\bigoplus_i CH^{2}(T_i)_{\mathrm{hom}} \arrow[r, "\alpha"] \arrow[d, "AJ_T"] &
\displaystyle\bigoplus_j CH^{2}(S_j)_{\mathrm{hom}} \arrow[d, "AJ_S"] \\
\displaystyle\bigoplus_m J^{3}(F_m) \arrow[r] \arrow[d] &
\displaystyle\bigoplus_i J^{3}(T_i) \arrow[r, "b"] \arrow[d] &
\displaystyle\bigoplus_j J^{3}(S_j) \arrow[d] \\
\displaystyle\bigoplus_m \operatorname{coker}(AJ_{F_m}) \arrow[r, "c"] \arrow[d] &
\displaystyle\bigoplus_i \operatorname{coker}(AJ_{T_i}) \arrow[r] \arrow[d] &
0 \\
0  &
0  &
\end{tikzcd}
\]
Here $F_m$ are the fourfold strata of the exceptional divisor of the resolution, and $AJ$ stands for the respective Abel-Jacobi maps. (We use the notation $F^2 CH^2$ for the kernels
of these maps, as suggested by the conjectural Bloch-Beilinson filtration.)
The columns are exact by definition; here we are using the fact that $\coker (AJ_{S_j}) = 0$, as the Abel-Jacobi map is always surjective in the case of a surface.
Moreover, since $\gr_3^W H^{n+1}(X, \mathbb{Q}) = 0$, again by the isolated singularities hypothesis, the pullback of Hodge structures
\begin{align*}
    \bigoplus_i H^3(T_i, \mathbb{C}) \rightarrow \bigoplus_j H^3(S_j, \mathbb{C})
\end{align*}
is surjective. Viewing the intermediate Jacobian $J^3(-)$ as a quotient of $H^3(-, \mathbb{C})$, the map $b$ is surjective as well. A diagram chase shows it suffices to prove the surjectivity of $a$ and $c$.

Let us focus on $c$. By Lemma \ref{lemma coker AJ} below, the uniruled threefold strata satisfy $\coker (AJ_{T_i}) = 0$, as their $CH_0$ is supported on a surface. We now use the independent Theorem \ref{thm mmp}, a consequence of the MMP which is proved in the next section. Using the numbering in the statement of that theorem, 
it suffices to show surjectivity of the restriction
\begin{align*}
    c: \bigoplus_{q=1}^p \coker (AJ_{F_q}) \rightarrow \bigoplus_{r=1}^p \coker (AJ_{T_r}).
\end{align*}
We view $c$ as a block-matrix, with entries $c_{rq}\colon \coker (AJ_{F_q}) \rightarrow \coker(AJ_{T_r})$ plus or minus the pullback of the inclusion map if $T_q \subset F_r$ and zero otherwise. We claim the diagonal entries $i_q^*\colon \coker (AJ_{F_q}) \rightarrow \coker(AJ_{T_q})$ are surjective. Indeed, it is straightforward to see that $\coker(AJ)$ is a birational invariant of smooth projective varieties. Therefore, $i_q^*$ is split by the pullback of the rational fibration $F_q \dashrightarrow T_q$ of $(1)$ in Theorem \ref{thm mmp}(i), so certainly surjective. Moreover, $c$ is upper-triangular. Indeed, $(2)$ in that same theorem says precisely that the entries $c_{rq}$ with $q > r$ below the diagonal are all zero. The claim follows.

An identical argument using Theorem \ref{thm mmp}(ii) shows that $a$ is surjective.
\end{proof}

We required the following lemma in the proof of Theorem \ref{thm:Chow-higher}.

\begin{lemma}\label{lemma coker AJ}
If $Y$ is a smooth projective variety such that $CH_0(Y)$ is supported on a surface, then $AJ_Y: CH^2(Y)_{\mathrm{hom}} \rightarrow J^3(Y)$ is surjective.
\end{lemma}

\begin{proof}
This is a standard application of the decomposition of the diagonal. Since $CH_0(Y)$ is supported on a surface $S$, \cite[Proposition $1$]{bloch srinivas} implies there is an equality of cycle classes
\begin{align*}
    N \Delta_Y = Z_1 + Z_2 \in CH^n (Y \times Y)
\end{align*}
for some integer $N > 0$, where $Z_1, Z_2$ are codimension $n$ cycles with
\begin{align*}
    \mathrm{Supp} (Z_1) \subset D \times Y, \hspace{.2cm} D \subsetneq Y, \hspace{.5cm} \mathrm{Supp} (Z_2) \subset Y \times S.
\end{align*}
Denote by $\widetilde{D}$ a resolution of $D$, with induced map $a\colon \widetilde{D} \rightarrow Y$, and by $\widetilde{S}$ a resolution of $S$, with induced map $b\colon \widetilde{S} \rightarrow Y$. Note that $Z_1$ is the pushforward of a cycle on $\widetilde{D} \times Y$ and $Z_2$ is the pushforward of a cycle on $Y \times \widetilde{S}$.

Viewing these cycles as correspondences, the map $N \cdot \mathrm{Id}\colon \coker(AJ_Y) \rightarrow \coker (AJ_Y)$ is a sum of maps
\begin{align*}
    \coker ( AJ_Y)  \xrightarrow{b^*} \coker ( AJ_{\widetilde{S}} ) \rightarrow \coker ( AJ_Y )
\end{align*}
and 
\begin{align*}
    \coker ( AJ_Y)  \rightarrow \coker \big( CH^1(\widetilde{D})_{\mathrm{hom}} \rightarrow \Pic^0(\widetilde{D}) \big) \xrightarrow{a_*} \coker ( AJ_Y ).
\end{align*}
Observe that $\coker ( AJ_{\widetilde{S}} ) = 0$ because the Abel-Jacobi map to the Albanese of a surface is always surjective. Moreover, $\coker \big( CH^1(\widetilde{D})_{\mathrm{hom}} \rightarrow \Pic^0(\widetilde{D}) \big)$ is certainly zero. It follows that $\coker (AJ_Y)$ is $N$-torsion. But it is also divisible because $J^3(Y)$ is a torus. Therefore, it is equal to zero.
\end{proof}

\noindent
{\bf Ordering of exceptional strata.}\label{section mmp}
In the proof of Theorem \ref{thm:Chow-higher} we made crucial use of the following result of independent interest, which requires running a
 suitable minimal model program. The proof relies heavily on the material in \S\ref{scn:MMP}.

\begin{theorem}\label{thm mmp}
Let $(X, \Delta)$ be a quasi-projective klt pair of dimension $n \geq 4$, with isolated singularities, and let $f \colon Y \to X$ be a strong log resolution of $(X, 
\Delta)$. Suppose $f$ factors through a plt blow-up (see Definition \ref{defn plt blow-up}) at each of the singular points. 
Denote by $F_m$  the fourfold strata, by $T_i$ the threefold strata, and by $S_j$ the surface strata of the exceptional locus of $f$. Then:

\noindent
(i)  There exists an ordering $T_1, \ldots, T_p$ of the non-uniruled threefold strata, together with an ordering $F_1, \ldots ,F_p$ of some subset of fourfold strata, such that:
\begin{enumerate}
    \item for each $1 \leq q \leq p$, there is an inclusion $i_q\colon T_q \hookrightarrow F_q$, and $i_q$ is a section of a rational fibration $F_q \dashrightarrow T_q$, and
    \item $T_q \subset F_r$ only if $q \leq r$.
\end{enumerate}

\noindent
(ii) Similarly, there exists an ordering $S_1, \ldots, S_p$ of the non-uniruled surface strata, together with an ordering $T_1, \ldots, T_p$ of some subset of threefold strata, such that:
\begin{enumerate}
    \item for each $1 \leq q \leq p$, there is an inclusion $i_q\colon S_q \hookrightarrow T_q$, and $i_q$ is a section of a rational fibration $T_q \dashrightarrow S_q$, and
    \item $S_q \subset T_r$ only if $q \leq r$.
\end{enumerate}
\end{theorem}

\begin{proof}[Proof of Theorem \ref{thm mmp}]
We will prove the first statement. The second is entirely analogous. Recall that the set-up is precisely that described after Example \ref{example kummer threefold}.

 By shrinking $X$ and dealing with each isolated singularity separately, we may assume that the pair $(X, \Delta)$ (hence also $X$) has a unique singularity $x \in X$ and that $f = h \circ g$ for a plt blow-up $h\colon Z \rightarrow X$. Thus, the pair $(Z, \Delta_Z)$ is plt, where $\Delta_Z = D+\widetilde{\Delta}$. Set $\Gamma = E' + \widetilde{\Delta}_Z$, where $E'$ is the union of the $g$-exceptional divisors on $Y$. By assumption the pair $(Y, \Gamma)$ is dlt, and we claim that we can run an MMP on the pair $(Y, \Gamma)$ over $Z$, just as in the klt case appearing in Theorem  \ref{thm bchm}.\footnote{Note that currently it is not known that one can do this for any dlt pair, unless we assume Special Termination.  This in turn relies inductively on knowing the full MMP in lower dimension, hence is known unconditionally only in dimension $\le 5$; cf. \cite{Birkar}. We thank Fanjun Meng for this observation.}

To see this, consider the new $\Q$-divisor 
$$G  : = \Gamma - \epsilon g^* D = E' + \widetilde{\Delta}_Z - \epsilon g^* D$$ 
on $Y$, where $\epsilon > 0$ is a sufficiently small rational number. Using the assumption that the pair $(X, \Delta)$ is log smooth away from $x$, we see that $g^*D$ is supported precisely on the exceptional locus of the log resolution $f$. Therefore the pair $(Y, G)$ becomes klt, and we can run an MMP for this pair over $Z$ by Theorem \ref{thm bchm}. On the other hand, the divisors $\Gamma$ and 
$G$ are numerically equivalent over $Z$, as they differ by a multiple of $g^*D$. Therefore running a MMP for $(Y, G)$ relative over $Z$ is in practice the same thing as running one for $(Y, \Gamma)$.

Running one such MMP yields a factorization of $g$ as a composition 
\[
Y=Y_0 \overset{g_0}{\dashrightarrow} Y_1 \overset{g_1}{\dashrightarrow} \cdots \dashrightarrow
Y_r {\rightarrow} Z .
\]
The pairs $(Y_i, \Gamma_i)$ are dlt, where $\Gamma_i$ is the pushforward of $\Gamma$ to $Y_i$. For each $0 \leq i < r$, the map $g_i: Y_i \dashrightarrow Y_{i+1}$ is either a divisorial contraction or a flip corresponding to a $(K_{Y_i}+\Gamma_i)$-negative extremal ray $R_i$. Denote by $j_i: Y_i \rightarrow Z$ the morphism to $Z$. Since $Y_r$ is a log terminal model, $K_{Y_r} + \Gamma_r$ is relatively nef over $Z$.

We claim that $Y_r \rightarrow Z$ has no exceptional divisors. For this, consider the auxiliary $\mathbb{Q}$-divisor
\begin{equation}\label{eqn A}
\begin{split}
A &= E' +\widetilde{\Delta}_Z - \Delta_Y \\
  &\sim_{\mathbb{Q}} K_Y + \Gamma - g^*(K_Z+\Delta_Z).
\end{split}
\end{equation}
Here, we are using the log pullback $\Delta_Y$ of $\Delta_Z$ defined by \ref{eqn delta_Y}. Since $(Z, \Delta_Z)$ is plt, the coefficients of $\Delta_Y$ are $<1$, except for $\widetilde{D}$. Hence, $A$ is an effective divisor with support consisting precisely of the exceptional components of $g$.

On the other hand, the second line of (\ref{eqn A}) shows that the pushforward $A_r$ of $A$ to $Y_r$ is $\mathbb{Q}$-linearly equivalent to $K_{Y_r} + \Gamma_r$ over $Z$. Since $Y_r$ is a log terminal model, this is relatively nef over $Z$. But then $A_r$ is effective, exceptional, and relatively nef, so it must be trivial by the negativity lemma \cite[Lemma $3.39$]{kollar mori}. In other words, all exceptional components of $g$ are contracted.

For each $0 \leq i < r$, fix an irreducible component $E_i$ of $\Gamma_i^{=1}$ such that $(E_i \cdot R_i) > 0$. Its existence is a consequence of \cite[Lemma $21$]{dFKX}, because the pullback $j_i^* D$ has support equal to $\Gamma_i^{=1}$. In particular, $E_i$ intersects each contracted curve of $g_i$ and the proper transform $g_{i*}E_i$ contains the exceptional locus $\mathrm{Ex}(g_i^{-1})$.

Let $T$ be a non-uniruled threefold stratum on $Y$. Consider its successive proper transforms $T_i$ in $Y_i$. These remain lc centers on $Y_i$ up until the first map $g_i$ such that $T_i \subset \mathrm{Ex}(g_i)$, and then they are no longer lc centers \cite[Lemma $3.38$]{kollar mori}. We study the associated contraction $h_i: Y_i \rightarrow Z_i$. Thus, $Y_{i+1} = Z_i$ if $g_i$ is a divisorial contraction, and $h_i^+:Y_{i+1} \rightarrow Z_i$ is the flip of $h_i$ otherwise.

We claim that $T_i \subset E_i$. For this, apply Proposition \ref{prop connectness} to the morphism $h_i$ and a generic point $\eta \in h_i(T_i)$. Since $E_i$ meets every contracted curve, this implies $(T_i \cap E_i)_\eta \neq \emptyset$. Since $T_i$ is non-uniruled, the fiber $(T_i)_\eta$ is zero-dimensional, so it follows that $T_i \subset E_i$.

Consider the unique fourfold stratum $F_i$ in $Y_i$ such that $T_i$ is an irreducible component of $F_i \cap E_i$. Again by Proposition \ref{prop connectness}, we have $T_i = F_i \cap E_i$ over a neighborhood of $h(T_i)$. We claim that $F_i$ lies in the exceptional locus $\mathrm{Ex}(g_i)$. If $g_i$ is a divisorial contraction, this is clear. We require an argument to cover the case where $g_i$ is a flip. Suppose for the sake of contradiction that $F_i \not\subset \mathrm{Ex}(g_i)$. The intersection $g_{i*}F_i \cap g_{i*} E_i$ of lc centers on $Y_{i+1}$ is non-empty since $g_{i*}E_i \supset\mathrm{Ex}(g_i^{-1})$. Therefore, it is a union of lc centers. But it is contained in $\mathrm{Ex}(g_i^{-1})$, at least over a neighborhood of $h(T_i)$. Since $\mathrm{Ex}(g_i^{-1})$ contains no lc centers, this is a contradiction. Thus, $F_i$ is exceptional, and $F_i \cap E_i$ has no other components than $T_i$ by Proposition \ref{prop connectness}.

If $F_i \xrightarrow[]{k_i} k_i(F_i) \rightarrow h_i(F_i)$ is the Stein factorization of the restriction $h_i|_{F_i}$, we see that $T_i$ maps birationally onto $k_i(F_i)$. The inclusion $T_i \hookrightarrow F_i$ is a rational section of $k_i$.

Now let us pair the threefold stratum $T$ with the proper transform $F \subset Y$ of $F_i$. We order these pairs in the same order they are contracted (if multiple pairs are contracted at the same step, choose arbitrarily). We show this satisfies the conditions of Theorem \ref{thm mmp}. The fourfolds $F$ are all different, because we can recover $T$ by the equality $T_i = F_i \cap E_i$. We have seen that $T_i \hookrightarrow F_i$ is a section of a rational contraction, so the same holds for their transforms $T \hookrightarrow F$, from which property $(1)$ follows. For property $(2)$, note that whenever a fourfold stratum $F$ is contracted, all threefold strata $T \subset F$ are contracted along with it. Thus, $T_q  \not\subset F_r$ when $q > r$. Indeed, there are no containments between non-paired threefold and fourfold strata contracted by the same $g_i$, because of the identity $T_i = F_i \cap E_i$.
\end{proof}

\section{General results on $KH$ and $K$-groups}\label{scn:results-K-groups}

\subsection{The case $i = n$}
We have the following description of the lowest $K$-group of $X$. When $X$ has isolated Cohen-Macaulay singularities, this was proved by 
Haesemeyer \cite[Corollary $8.15$]{descent properties}. An intermediate statement that will appear in the proof, namely that 
$K_{-n}(X) \simeq KH_{-n}(X) \simeq H_{\cdh}^n(X,\mathbb{Z}(0))$, is already contained in \cite[Theorem 6.2]{k weibel}.

\begin{theorem}\label{theorem K-n}
If $X$ is  a complex variety of dimension $n \geq 1$, then 
$$K_{-n}(X) \simeq KH_{-n}(X) \simeq H^n_0(X, \mathbb{Z}).$$ 
When $X$ is normal, this is isomorphic to $H^{n-1}(\mathcal{D}(E), \mathbb{Z})$, where $\mathcal{D}(E)$ is the dual complex of the exceptional divisor of a log resolution. 
Moreover, $K_{-n}(X) = 0$ if $X$ is of klt type.
\end{theorem}

\begin{proof}
Thanks to  Corollary \ref{cor:weibel-conj} we have $K_{-n}(X) \simeq KH_{-n}(X)$. Consider the Atiyah Hirzebruch type spectral sequence
\begin{align*}
    E_2^{p,q} = H_{\cdh}^{p-q}(X, \mathbb{Z}(-q)) \Rightarrow KH_{-p-q}(X),
\end{align*}
as in Theorem \ref{thm motivic cohom}. The term $E_2^{n,0} = H_{\cdh}^n (X,\mathbb{Z}(0))$ contributes to $KH_{-n}(X)$, while Lemma 
\ref{lem motivic soule-weibel} shows all other terms contributing to $KH_{-n}(X)$ vanish. Moreover, it shows that $E_2^{n,0} = E_\infty^{n,0}$. It follows that 
$$KH_{-n}(X) \simeq H_{\cdh}^n(X,\mathbb{Z}(0)).$$ 
The latter is equal to $H^n_0(X, \mathbb{Z})$ by Proposition \ref{prop weight zero}. The statement about the dual complex is a consequence of Corollary \ref{corollary dual complex}, since normal singularities are regular in codimension one.
\end{proof}

\begin{corollary}\label{corollary weight zero}
If $X$ is projective, then
$$K_{-n}(X)_\mathbb{Q} \simeq H^n_{\cdh}(X, \mathbb{Q}(0)) \simeq W_0 H^n(X, \mathbb{Q}).$$ 
In particular, $K_{-n}(X)_\mathbb{Q} = 0$ if $X$ has rational singularities.
\end{corollary}

\begin{proof}
After applying Theorem \ref{theorem K-n} and tensoring with $\Q$, the statement follows from Corollary \ref{cor:cdh-weight-zero}.
\end{proof}

\subsection{The case $i = n -1$}
 We describe the $K$-theoretic consequences  of our results for the next to last degree. First, we prove 
Conjecture \ref{main-KH-conjecture} for $i = n-1$. Together with the previous section, this also completes the proof of Theorem \ref{thm:main-KH}(i) and 
Theorem \ref{thm:main-K}(i).

\begin{corollary}\label{cor K -n+1}
Let $X$ be a projective variety of dimension $n \ge 2$, and assume that 
$$W_0 H^{n-1}(X, \mathbb{Q}) = W_1 H^n(X, \mathbb{Q}) = \gr^W_2 H^{n+1}(X, \mathbb{Q}) = 0.$$
Then $KH_{-n+1}(X)_\mathbb{Q} = 0$. In particular, this holds if $X$ has rational singularities. 
\end{corollary}

\begin{proof}
The spectral sequence in Theorem \ref{thm motivic cohom}(1) gives a decomposition
\begin{align*}
    KH_{-n+1}(X)_\mathbb{Q} \simeq H^{n-1}_{\rm cdh} (X, \mathbb{Q}(0)) \oplus H^{n+1}_{\rm cdh} (X, \mathbb{Q}(1)).
\end{align*}
We have $H^{n-1}_{\cdh}(X, \mathbb{Q}(0)) \simeq W_0 H^{n-1}(X, \mathbb{Q})$ by Corollary \ref{cor:cdh-weight-zero}. Moreover, by Corollary \ref{cor cdh 1 vanishing}, $H^{n+1}(X, \mathbb{Q}(1))$ vanishes if $W_1 H^n (X, \mathbb{Q}) = \gr_2^W H^{n+1}(X, \mathbb{Q}) = 0$. The fact that this holds for rational singularities is a consequence of Corollary \ref{cor weight restrictions} (see also Example \ref{ex:rational}). Finally, rational singularities are Du Bois, hence in this case we have $K_{-n+1}(X) \simeq KH_{-n+1}(X)$ by Theorem \ref{thm rosie}(i).
\end{proof}

Next we obtain a stronger result for singularities of klt type, computing $KH_{-n+1}$ integrally, using the notation in \S\ref{scn:weight1}.

\begin{theorem}\label{thm:K_{-n+1}}
If $X$ is projective of dimension $n \ge 2$, and of klt type, then 
$$K_{-n+1}(X)  \simeq KH_{-n+1}(X) \simeq H^{n-1}(\Pic(X_\bullet)) \simeq H^{n-1}( {\rm NS}(X_\bullet)).$$
\end{theorem}
\begin{proof}
Since the singularities of $X$ are of klt type, they are rational, hence as above $K_{-n+1}(X) = KH_{-n+1}(X)$.
The only terms contributing to $KH_{-n+1}(X)$ in the spectral sequence in Theorem \ref{thm motivic cohom}(1) are $E_2^{n-1,0} = H^{n-1}_{\cdh}(X, \mathbb{Z}(0))$ and $E_2^{n,-1} = H^{n+1}_{\cdh}(X, \mathbb{Z}(1))$. We have $H^{n-1}_{\cdh}(X, \mathbb{Z}(0)) = 0$ by Proposition \ref{prop weight zero}. Moreover, Lemma \ref{lem motivic soule-weibel}
shows $E_2^{n,-1} = E_\infty^{n,-1}$; this is immediate for $n \ge 3$, while for $n =2$ we also need to apply Remark \ref{rmk:Z-copy}, which 
implies that the differential $E^{0,0}_2 \to E^{2, -1}_2$ is trivial.
Hence 
$$KH_{-n+1}(X) \simeq H^{n+1}_{\cdh}(X, \mathbb{Z}(1)),$$ 
and we may apply the last statement of Proposition \ref{prop weight one}.
\end{proof}

\begin{remark}
Note that even without the klt assumption, by Lemma \ref{lem motivic soule-weibel} we have $E_2^{n-1,0} = E_\infty^{n-1,0}$, while $E^{n, -1}_\infty \simeq E^{n, -1}_2/ E^{n-2, 0}_2$, 
hence although the result is not as clean, at least theoretically the results of \S\ref{scn:weight0} and \S\ref{scn:weight1} provide a method for computing $KH_{-n+1}(X)$. 
\end{remark}

 Let us make the statement about $\Pic X_\bullet$ more explicit.  Let $f \colon Y \to X$ is a log resolution which is an isomorphism away from the singular locus, and denote by $S_j$  the surface strata, and by $C_k$ the curve strata; again, for simplicity, we do not index them according to which exceptional components they belong to, but we always have maps $\Pic S_j \to \Pic C_k$ and other analogous such restrictions, whenever 
$C_k$ is contained in $S_j$, etc. Then
\begin{equation}\label{eqn:h^{n-1}Pic}
  H^{n-1}( \Pic X_\bullet ) = \coker \big( \bigoplus_j \Pic S_j \rightarrow \bigoplus_k \Pic C_k \big)
\end{equation}

\subsection{The case $i = n - 2$}
Moving on to $K_{-n +2} (X)$, we need to impose singularity conditions from the start, hence we can only address Conjectures \ref{sings-KH-conjecture} and \ref{main-K-conjecture}.
The following statement completes the proof of Theorems \ref{thm:main-KH}(ii) and \ref{thm:main-K}(ii).

\begin{theorem}\label{thm K_-n+2}
(i) Let $X$ be a projective variety of dimension $n \geq 3$, with isolated singularities of klt type. Then
$$K_{-n+2}(X)_\mathbb{Q} \simeq KH_{-n+2}(X)_\mathbb{Q} \simeq H^{n-2}( \Pic (X_\bullet) )_\mathbb{Q} \simeq  \gr_2^W H^n(X, \mathbb{Q}).$$
In particular, $K_{-n+2}(X)_\mathbb{Q} = 0$ if $X$ satisfies $D_1$, so e.g. if it has pre-$1$-rational singularities.

\noindent
(ii) If $n = 3$, the same result holds more generally for isolated rational singularities, assuming Bloch's conjecture on $0$-cycles on smooth surfaces. If in addition 
the threefold is a local complete intersection, it holds integrally, i.e.
$$K_{-1} (X) \simeq KH_{-1} (X) \simeq H^1( \Pic (X_\bullet) ).$$
\end{theorem}

As in the previous subsection, we make this more explicit as in \S\ref{scn:weight1}. Let $f \colon Y \to X$ be a log resolution which is an isomorphism away from the singular locus, and denote by  $T_i$  the threefold strata of the exceptional locus, by $S_j$  the surface strata, and by $C_k$ the curve strata.  Then
\begin{equation}\label{eqn:h^{n-2}Pic}
  H^{n-2}( \Pic X_\bullet ) = \frac{\ker \big( \bigoplus_j \Pic S_j \rightarrow \bigoplus_k \Pic C_k \big)}{\im \big( \bigoplus_i \Pic T_i \rightarrow \bigoplus_j \Pic S_j \big)}
\end{equation}

\smallskip

\begin{remark}
On close inspection, all but one point in the proof of Theorem \ref{thm K_-n+2} works integrally, namely the surjectivity of the composition (\ref{eqn class surj}). This is a serious obstacle, as demonstrated by the case of the Kummer variety associated to an abelian threefold in Example \ref{example kummer threefold} and \ref{example kummer threefold details} (where $K_{-1} (X)$ is torsion, but nonzero, as explained in Example \ref{example-K-Kummer}).
\end{remark}

We now explain how to deduce Theorem \ref{thm K_-n+2} from the results we have already discussed in \S\ref{scn:weight2} and prior.

Since $X$ has isolated Du Bois singularities, it is $K_{-n+2}$-regular by Theorem \ref{thm rosie}(ii). In particular, $K_{-n+2}(X) \simeq KH_{-n+2}(X)$, and as usual we analyze the latter. There are three terms contributing to $KH_{-n+2}(X)$  which may be nontrivial in the Atiyah-Hirzebruch type spectral sequence in Theorem \ref{thm motivic cohom}(1).

The first term $H_{\cdh}^{n-2}(X, \mathbb{Z}(0))$ vanishes by Proposition \ref{prop weight zero}, under the klt assumption. Moreover, if $X$ has rational singularities, we have 
$H_{\cdh}^{n-2}(X, \mathbb{Q}(0)) = 0$ by Corollary \ref{cor:cdh-weight-zero}.

The second term $H^n_{\cdh}(X, \mathbb{Z}(1))$ is isomorphic to the desired answer $H^{n-2}( \Pic X_\bullet)$ under the klt assumption, by Proposition \ref{prop weight one}, and it is straighforward to see that it survives to the $E_\infty$-page. When $X$ has rational singularities,  by Corollary \ref{cor H^3(1)} we have in addition 
$$H^n_{\cdh}(X, \mathbb{Q}(1)) \simeq H^{n-2}( \Pic (X_\bullet) )_\mathbb{Q} \simeq  \gr_2^W H^n(X, \mathbb{Q}).$$
This is $0$ when $X$ has pre-$1$-rational singularities, by Corollary \ref{cor weight restrictions} .

The final term is $H^{n+2}_{\cdh}(X, \mathbb{Z}(2))$. This was treated in \S\ref{scn:weight2}, and it vanishes rationally under the klt assumption, and also integrally under the local complete intersection rational singularities assumption for threefolds.

\subsection{$K$-groups of quasi-projective varieties with isolated singularities}\label{scn:quasiproj-isolated}
The case of quasi-projective varieties with isolates singularities can be reduced to the projective case due to the following basic fact derived 
using localization sequences in $K$-theory; see \cite[Theorem 1.2(i)]{weibel isolated}. See  also \emph{loc. cit.}
for more on local-to-global results.

\begin{lemma}\label{lemma:negative-k-semi-local}
    Let $Y$ be a quasi-projective variety with isolated singularities. Then we have
 \[K_{-i}(X)\simeq K_{-i}(R) \,\,\,\,\,\,{\rm for ~all} \,\,\,\, i > 0,\]
    where $R$ is the semi-local ring of $X$ at its singular points. In particular, if $X$ is a projective compactification with the same singularities, 
    then
    \[K_{-i}(Y)\simeq K_{-i}(X) \,\,\,\,\,\,{\rm for ~all} \,\,\,\, i > 0.\]
 \end{lemma}

If $S$ is the local analytic germ of an isolated singularity, by Artin approximation there exists an affine variety $Y$ and $y \in Y$ such that $y$ is the only singular point of $Y$ and $R := \O_{X,x}$ satisfies $\widehat{R} \simeq S$.  For every such local ring $R$ of an isolated affine singularity, 
given a compactification $X$ as in Lemma \ref{lemma:negative-k-semi-local}, we have
$$K_{-i}(R)\simeq K_{-i}(Y) \simeq K_{-i} (X) \,\,\,\,\,\,{\rm for ~all} \,\,\,\, i > 0.$$

In any event, we then have that the case of quasi-projective varieties with isolated singularities encompasses all of the above, and by 
Lemma \ref{lemma:negative-k-semi-local} all the conjectures and results in this paper remain the same as in the projective setting. We summarize here some of the results that hold due to this fact, without aiming for an exhaustive list.

\begin{corollary}
Let $X$ be a complex quasi-projective variety with isolated singularities, of dimension $n$. Then:

\noindent
(i) If $X$ has rational singularities, then 
$$K_{-n} (X)_\Q = KH_{-n}(X)_\Q = 0\,\,\,\,{\rm and}\,\,\,\,K_{-n+1} (X)_\Q = KH_{-n+1}(X)_\Q = 0.$$

\noindent
(ii) If $X$ is a threefold with klt type singularities, then 
$${\rm rk}~ K_{-1} (X)_\Q = \sum_{x \in X_{\rm sing}} \sigma^{\rm an} (X; x) - \sigma (X).$$
In particular, $K_{-1} (X)_\Q = 0$ if and only of the global $\Q$-factoriality defect is equal to the sum of the local analytic ones.

\noindent 
(iii) If $n \ge 3$ and $X$ has klt type and pre-$1$-rational singularities,\footnote{The hypothesis may sometimes be redundant. For instance, a $1$-rational local complete intersection is automatically terminal.}  then $K_{-n+2} (X)_\Q = 0$.
\end{corollary}

A special case of the last condition in (ii) is when $X$ is locally analytically $\Q$-factorial, or equivalently a rational homology manifold by \cite[Theorem D]{PP1}.
The pre-$1$-rational condition in (iii) can be replaced by the weaker $D_1$, and this also holds when $X$ is a rational homology manifold; for threefolds and fourfolds it is equivalent to this latter condition, by \cite[Theorem A]{PP1}.

\section{Small dimension, and further results and examples}

In this chapter we describe the case of varieties of dimension up to four, and also other general classes of varieties, in order to exemplify concretely when our conjectures are known to hold. We apply the results of this paper, but in a few instances we also rely on existing results in the literature. As a general outline, the conjectures hold in dimension at most $2$, almost fully hold in dimension $3$, and partially hold in dimension $4$. They hold for quotient singularities, toric varieties, conical varieties, and other special examples. We also give examples of non-vanishing $K$-groups in borderline cases, showing that our results cannot be improved to the integral setting. Further examples that require more work are provided in the Appendix.

\subsection{Varieties of dimension up to four}\label{scn:small-dim}
We describe what our results show, and what conjectures are verified, in dimension up to four.

\noindent
{\bf Curves and surfaces.}
~Let $C$ be a complex curve. The only relevant negative cohomology group is $K_{-1}$, and Theorem \ref{theorem K-n} says that
$$K_{-1}(C) \simeq  KH_{-1} (C) \simeq H^1_0 (C, \ZZ).$$ 
In particular, suppose that $C$ is irreducible and has singular points $p_i$ with $d_i$ local branches. Then $K_{-1}(C)$ is free abelian of rank $\sum (d_i-1)$.

This last statement is already known; it is discussed in terms of the cohomology of a graph in \cite[Section 2]{weibel surfaces}.

We also fully understand the case of normal surfaces, with especially clean results in the rational or klt case. Many of the results in the normal or rational singularities case can already be deduced from the foundational \cite{weibel surfaces}, which approaches the problem on surfaces via different methods that do not appeal to Hodge theory.

Theorems \ref{thm:main-KH} and \ref{thm:main-K} show that the conjectures in the Introduction hold for surfaces. In particular, for a surface $X$ with rational singularities, $K_i (X) \simeq KH_i(X)$ for $i < 0$, 
and 
$$K_{-2}(X)_\Q = K_{-1} (X)_\Q = 0.$$
We can be more explicit integrally.  For any projective surface $X$ we have
$$K_{-2} (X) \simeq KH_{-2} (X) \simeq  H^2_0 (X, \ZZ)$$
by Theorem \ref{theorem K-n}.  When $X$ is normal, this is isomorphic to $H^1 (\mathcal{D}(E), \mathbb{Z})$, where $\mathcal{D}(E)$ is the dual complex of the exceptional divisor of a log resolution.  This is $0$ when $X$ is of klt type.\footnote{Recall the for surfaces the klt condition is equivalent to having quotient singularities.}  Under this last assumption, we also have 
$$K_{-1} (X) \simeq KH_{-1} (X) \simeq H^1 (\Pic (X_\bullet)) \simeq {\rm coker} \big(\Pic Y \to \bigoplus_k \Pic C_k \big)$$
by Theorem \ref{thm:K_{-n+1}}. Here $f \colon Y \to X$ is a log resolution of $X$ which is an isomorphism away from the singular locus, and $C_k$ are the components of the reduced exceptional divisor.

For any normal surface $X$, the main result of \cite{weibel surfaces} says that there exists an exact sequence 
 \begin{equation}\label{eqn:Weibel-sequence}
 0 \to \Pic(X) \to \Cl(X) \to \bigoplus_{x\in X_{\sing}} \Cl(\widehat{\mathscr{O}}_{X,x}) \to K_{-1}(X) \to 0.
 \end{equation}
This can also be seen as an application of the  motivic cohomology techniques of this paper. We do not include the argument here, as a suitable generalization to arbitrary dimension will appear in \cite{BPS}. We remark that this recovers the fact that $K_{-1} (X)_\Q = 0$ for rational surface singularities, as they are known to be locally analytically $\Q$-factorial.

\begin{example}
There are simple examples of surfaces with quotient singularities, hence klt, such that $K_{-1}(X) \neq 0$. An example from \cite{projective quotients} appears in \S\ref{scn:graded} below. For a different perspective, among ADE singularities $E_8$ is the only type that has 
 $\Cl(\widehat{\mathscr{O}}_{X,x}) = 0$, hence by the exact sequence above any $X$ with such singularities satisfied $K_{-1} (X) = 0$. On the other hand,  by \cite{srinivas2} there exists a quasi-projective surface $X$ with a singularity of any ADE type, for which ${\rm Cl} (\mathscr{O}_{X,x}) = 0$; when the type is not $E_8$, it follows that $K_{-1} (X) \neq 0$.
\end{example}

\noindent
{\bf Threefolds.}
~When $X$ be a complex projective threefold, the paper solves a substantial part of our conjectures, and provides results over $\ZZ$.  At the same time, there are a few loose ends.

The case of $K_{-3}$ and $K_{-2}$ is covered by the results that hold in arbitrary dimension. More precisely $KH_{-3}$ and $KH_{-2}$ are covered by Theorem \ref{thm:main-KH}, while if $X$ has rational singularities, Theorem \ref{thm:main-K} gives 
$$K_{-3}(X)_\Q = K_{-2} (X)_\Q = 0.$$
Integrally, we have 
$$K_{-3}(X) \simeq KH_{-3} (X) \simeq H^3_0 (X, \ZZ)$$
by Theorem \ref{theorem K-n}. When $X$ is normal, this is isomorphic to $H^2 (\mathcal{D}(E), \mathbb{Z})$, where $\mathcal{D}(E)$ is the dual complex of the exceptional divisor of a log resolution of $X$, and this is $0$ when $X$ is of klt type.

When $X$ is of klt type, we also have 
$$K_{-2} (X) \simeq KH_{-2} (X) \simeq H^2 (\Pic (X_\bullet)) \simeq \coker \big( \bigoplus_j \Pic S_j \rightarrow \bigoplus_k \Pic C_k \big)$$
by Theorem \ref{thm:K_{-n+1}}. The last isomorphism is (\ref{eqn:h^{n-1}Pic}); here $S_j$ and $C_k$ are the surface and curve strata of the exceptional divisor of a strong resolution of singularities.

\begin{example}\label{ex:cA1}
When $X$ has ordinary singularities (for instace $cA_1$), so there is only one exceptional surface over each singular point, the formula immediately shows $K_{-2} (X) = 0$.  In Example \ref{example-K2-cAn} we will see that $K_{-2} (X) = 0$ for all $cA_n$ threefold singularities, 

On the other hand, in Example \ref{exmp 3-fold with cayley} we construct a klt threefold with isolated singularities for which $K_{-2} (X) \neq0$.
\end{example}

Finally, we focus on $K_{-1}$. Note first that for this group Conjecture \ref{main-KH-conjecture} is still open, without singularity assumptions.
On the other hand, we have a good understanding of Conjectures \ref{sings-KH-conjecture} and \ref{main-K-conjecture}, and of further more precise results, except for the last part of the former in the non-isolated singularities case.

 If $X$ has isolated klt singularities, by Theorem \ref{thm K_-n+2}(i) we have
$$K_{-1}(X)_\mathbb{Q} \simeq KH_{-1}(X)_\mathbb{Q} \simeq H^1( \Pic (X_\bullet) )_\mathbb{Q} \simeq  \gr_2^W H^3(X, \mathbb{Q}).$$
Since $W_1 H^3(X, \mathbb{Q}) = 0$ for rational singularities, this is saying that $K_{-1}(X)_\mathbb{Q} = 0$ if and only if $H^3 (X, \Q)$ is pure. 
By \cite[Theorem E]{PP2}, we deduce the equivalence
\begin{equation}\label{eqn:threefold-equiv}
K_{-1}(X)_\mathbb{Q} = 0 \iff \sigma (X) = \sum_{x \in X_{\rm sing}} \sigma^{\rm an} (X; x) \iff \underline{h}^{1,2} (X) = \underline{h}^{2,1} (X),
\end{equation}
where $\sigma (X)$ is the global $\Q$-factoriality defect of $X$, $ \sigma^{\rm an} (X; x)$ is the local analytic $\Q$-factoriality defect at each singular point, 
and $\underline{h}^{1,2} (X)$ and $\underline{h}^{2,1} (X)$ are Hodge-Du Bois numbers.

\begin{example}
The results above say that whether $K_{-1} (X)$ is torsion or not depends on a delicate local to global interaction.

A threefold node $x \in X$ is not locally analytically $\Q$-factorial; in fact $\sigma^{\rm an} (X; x) = 1$. If we exhibit this singularity as the projective cone over a smooth quadric in $\mathbb{P}^3$, then $\sigma (X) = 1$ as well, and so by the above $K_{-1}(X)$ is torsion. In fact $K_{-1} (X) = 0$, as we will see in \S\ref{scn:cones}. (For a situation when $K_{-1}(X)$ is torsion but nonzero, see Example \ref{example-K-Kummer} in the Appendix.)

On the other hand, if 
$X \subset \mathbb{P}^4$ is a nodal hypersurface of degree $d$ with at most $(d-1)^2 -1$ nodes, then $X$ is factorial by \cite{Cheltsov}. 
Therefore for such a hypersurface we have $K_{-1} (X)_{\Q} \neq 0$.
\end{example}

If we weaken the klt assumption to rational singularities in the non-Gorenstein case, the same statements hold conditional on Bloch's conjecture on $0$-cycles on smooth projective surfaces. 
In the opposite direction, if $X$ is in addition a local complete intersection, we have integrally
$$K_{-1} (X) \simeq KH_{-1} (X) \simeq H^1( \Pic (X_\bullet) ).$$

For dimension reasons, the threefold stratum on the resolution is $Y$ itself, and the formula  in (\ref{eqn:h^{n-2}Pic}) becomes
\begin{align*}
    H^{1}( \Pic X_\bullet )  = \frac{\ker \big( \bigoplus_j \Pic S_j \rightarrow \bigoplus_k \Pic C_k \big)}{\im \big( \Pic Y \rightarrow \bigoplus_j \Pic S_j \big)}.
\end{align*}
This can be written more intrinsically, directly in terms of $X$, as follows; there is an exact sequence 
\[0 \to \Pic(X) \to \Cl(X) \to \bigoplus_{x\in X_{\sing}} \Cl(\widehat{\mathscr{O}}_{X,x}) \to H^{1}( \Pic X_\bullet ) = K_{-1} (X) \to 0.\]
This formula for $K_{-1} (X)$ is similar to the case of normal surfaces discussed above; for  threefolds with at most isolated $cA_n$ singularities, it was derived 
in \cite[Proposition 3.6 and Corollary 3.8]{pavic} with the aid of Kn\"orrer periodicity. 
As mentioned in the previous subsection, rather than including the argument for this last exact sequence here,  a generalization to arbitrary dimension will appear in \cite{BPS}.  Finally, note that the sequence gives another explanation for the equivalence (\ref{eqn:threefold-equiv}).

\noindent
{\bf Fourfolds.}
~The picture is rather clean for fourfolds as well, except not entirely so for $K_{-1}$ at the moment, which shows where the first serious new obstructions arise.

The groups $K_{-4}$ and $K_{-3}$ are again covered by the results that hold in arbitrary dimension. More precisely $KH_{-4}$ and $KH_{-3}$ are covered by Theorem \ref{thm:main-KH}, while if $X$ has rational singularities, Theorem \ref{thm:main-K} gives 
$$K_{-4}(X)_\Q = K_{-3} (X)_\Q = 0.$$
Integrally, we have 
$$K_{-4}(X) \simeq KH_{-4} (X) \simeq H^4_0 (X, \ZZ)$$
by Theorem \ref{theorem K-n}, and as before this is $0$ when $X$ is of klt type. Under this last assumption we also have 
$$K_{-3} (X) \simeq KH_{-3} (X) \simeq H^3 (\Pic (X_\bullet)) \simeq \coker \big( \bigoplus_j \Pic S_j \rightarrow \bigoplus_k \Pic C_k \big)$$
by Theorem \ref{thm:K_{-n+1}}, with the same notation as in the previous paragraph on threefolds. For example, we have $K_{-3} (X) = 0$ if a log resolution of $X$ does not contain triple intersections of exceptional divisors. This happens for instance for $cA_1$ (or any ordinary singularity) and $cA_2$ singularities, simply because these have one, respectively two, exceptional divisors in a log resolution.

We next focus on $K_{-2}$. When $X$ has isolated rational (or even Du Bois) singularities, we have $K_{-2}$-regularity by Theorem \ref{thm rosie}.
If moreover the singularities are of klt type, then 
$$K_{-2}(X)_\mathbb{Q} \simeq KH_{-2}(X)_\mathbb{Q} \simeq H^2 ( \Pic (X_\bullet) )_\mathbb{Q} \simeq  \gr_2^W H^n(X, \mathbb{Q}),$$
by Theorem \ref{thm K_-n+2}, and this can also be expressed as in (\ref{eqn:h^{n-2}Pic}). In particular, if $X$ has pre-$1$-rational singularities, then 
$$K_{-2}(X)_\mathbb{Q} = 0.$$

\begin{example}
Assume $X$ is a fourfold with hypersurface singularities given locally by $x_1^2 + x_2^2 + x_3^2 + x_4^2 + x_5^{n+1} = 0$, so of  $cA_n$ type. Since 
as in Example \ref{ex:min-exp} we have 
$$\widetilde{\alpha} (X) = 2 + \frac{1}{n+1} > 2,$$
$X$ has $1$-rational singularities, which are also klt since $X$ is a hypersurface. (Some other $cA$ and $cD$ type fourfold singularities have this property as well.) We thus have 
$$K_{-4} (X) = 0 \,\,\,\,\,\, {\rm and} \,\,\,\,\,\, K_{-3}(X)_{\Q} = K_{-2}(X)_{\Q} = 0.$$
Moreover, if $n \le 2$, then in fact $K_{-3}(X) =  0$. We don't know yet if, or when, the other vanishing statements hold integrally.
\end{example}

We are thus left with understanding $K_{-1}$. If $X$ is pre-$1$-rational (or even pre-$1$-Du Bois) with singular locus of dimension at most $1$ (as for instance in the case of terminal singularities), then $X$ is $K_{-1}$-regular, again by Theorem \ref{thm rosie}. By the Atiyah-Hirzebruch type spectral sequence we have a decomposition
$$KH_{-1} (X)_\Q \simeq H^1_{\rm cdh} (X, \Q (0)) \oplus H^3_{\rm cdh} (X, \Q (1)) \oplus H^5_{\rm cdh} (X, \Q (2)) \oplus 
H^7_{\rm cdh} (X, \Q (3)).$$
The first two summands are covered by the methods of this paper. When $X$ has rational singularities, we have 
$H^1_{\rm cdh} (X, \Q (0)) = 0$ by Corollary \ref{cor:cdh-weight-zero}, and
$$H^3_{\rm cdh} (X, \Q (1)) \simeq {\rm gr}_2^W H^3 (X, \Q)$$
by Corollary \ref{cor H^3(1)}. Moreover, the latter is $0$ if $X$ has pre-$1$-rational singularities, by the same result.

The last two summands are currently just beyond the limit of our methods. By Conjecture \ref{main-cdh-conjecture}, the term $H^7_{\rm cdh} (X, \Q (3))$ is predicted to be $0$ under the rational singularities assumption, while the term $H^5_{\rm cdh} (X, \Q (2)) $ is  predicted to be $0$ under the pre-$1$-rational 
singularities assumption. Overall, conjecturally we also have $K_{-1}(X)_{\Q} = 0$ when $X$ has pre-$1$-rational singularities.

\subsection{Graded rings and quotient singularities}\label{scn:graded}
Weibel shows in \cite{weibel kh}  that if $X = \mathop{\mathrm{Spec}}R$ for a graded ring $R$, then $KH_{i}(X) = 0$ for $i < 0$.

In particular, this applies when $X$ is the quotient of $\A^n$ by the linear action of a reductive group $G$. In the special case when $G$ is a finite group,  this is stronger than Conjecture \ref{main-KH-conjecture}, which predicts that $KH_{-i}(X)_\Q = 0$ for $i > 0$ for quotient singularities, and more generally for all rational homology manifolds.

Note also that quotient singularities are pre-$m$-rational for all $m$, by \cite[Corollary 4.3]{SVV}, hence using Theorem \ref{thm rosie}, in the affine case we have $K_{-i} (X) =  KH_{-i} (X) = 0$ for all $i > 0$.

For quasi-projective varieties with isolated quotient singularities, the negative $K$-groups are studied in \cite{projective quotients}. By Corollary 
2.30 in  \textit{loc. cit.}, we have 
$$K_{-1}(X)_{\Q} = 0 \,\,\,\,\,\,{\rm and} \,\,\,\,\,\, K_{-i} (X) = 0 \,\,\,\,{\rm for}\,\,\,\,i \ge 2.$$
Note that there exist simple examples where $K_{-1} (X)$ need not vanish integrally; see Example 2.33 in \textit{loc. cit.}, where it is shown that if $X$ is the quotient of $\mathbb{P}^1 \times \mathbb{P}^1$ by the diagonal action of $\ZZ/ 2 \ZZ$, then $K_{-1} (X) \simeq \ZZ/ 2\ZZ$.

We now show that our conjectures hold in full generality for all complex quasi-projective varieties with quotient singularities. The key point is a 
stronger statement about the behavior of cdh cohomology under finite surjective morphisms.

\begin{proposition}\label{prop finite map}
Let $\pi\colon Y \to X$ be a finite surjective map between normal varieties over a field $k$ of characteristic zero. Then 
$H^i_{\cdh}(X, \mathbb{Q}(j))$ is a direct summand of $H^i_{\cdh}(Y, \mathbb{Q}(j))$.

In particular, if $H^i_{\cdh}(Y, \mathbb{Q}(j))=0$, then $H^i_{\cdh}(X, \mathbb{Q}(j))=0$.
\end{proposition}

\begin{proof}
In this proof we appeal to the original approach to cdh-motivic cohomology, developed by Voevodsky. To prove the statement, we construct a trace map
    $$ \mathrm{Tr} \colon \pi_* (\mathbb{Q}(j)|_Y) \to \mathbb{Q}(j)|_X. $$
Recall from \cite[Definition 3.1]{MVW} that $\mathbb{Q}(j) = C_* \mathbb{Q}_{\mathrm{tr}}(\mathbb{G}_m^{\wedge j})[-j]$ is the complex of presheaves with transfers obtained as the Suslin complex of $\mathbb{Q}_{\mathrm{tr}}(\mathbb{G}_m^{\wedge j})$, where  by definition 
$\mathbb{Q}_{\mathrm{tr}}(\mathbb{G}_m^{\wedge j})(U) = \mathrm{Cor}_k(U, \mathbb{G}_m^{\wedge j})_{\mathbb{Q}}$ for $U\in \mathrm{Cor}_k$. As defined in \cite[Appendix 1A]{MVW}, 
    \begin{itemize}
        \item Objects in $\mathrm{Cor}_k$ are finite-type $k$-schemes;
        \item Morphisms $\mathrm{Cor}_k(V,W)$ are given by cycles in $V\times_k W$ that are universally integral over $V$, each of whose components are finite and surjective over $V$.
    \end{itemize}
 We denote by $\mathbb{Q}(j)|_Y$ and $\mathbb{Q}(j)|_X$ the restrictions of its cdh-sheafification to the small cdh sites of $Y$ and $X$, respectively.


For any open $V \subset X$, the restriction $\pi|_V\colon \pi^{-1}(V) \to V$ is finite and surjective. 
    Since $Y$ and $X$ are normal, by \cite[Theorem 1A.6]{MVW} the graph $\Gamma_{\pi|_V}\subset \pi^{-1}(V) \times V$ and its transpose $\Gamma_{\pi|_V}^t\subset V \times \pi^{-1}(V)$ are relative cycles over $\pi^{-1}(V)$ and $V$, respectively. Moreover, by \cite[Theorem 1A.3]{MVW}, these relative cycles are universally integral, hence belong to $\mathrm{Cor}_k(\pi^{-1}(V), V)_{\mathbb{Q}}$ and $\mathrm{Cor}_k(V, \pi^{-1}(V))_{\mathbb{Q}}$, respectively.

    Precomposing with $\Gamma_{\pi|_V}^t\in \mathrm{Cor}_k(V, \pi^{-1}(V))_{\mathbb{Q}}$ yields a map on sections
    $$ (\Gamma_{\pi|_V}^t)^*\colon \mathbb{Q}(j)(\pi^{-1}(V)) \to \mathbb{Q}(j)(V), $$
    given by the proper pushforward of cycles along the finite projection $\pi^{-1}(V) \times \Delta^n \times \mathbb{G}_m^{\wedge j} \to V \times \Delta^n \times \mathbb{G}_m^{\wedge j}$. These maps are compatible with restrictions, so upon cdh-sheafification they glue to a global map of sheaves
    $$ \mathrm{Tr} \colon \pi_* (\mathbb{Q}(j)|_Y) \to \mathbb{Q}(j)|_X. $$
Since $\pi$ is finite, we have canonical isomorphisms
$$H^i_{\cdh}(Y, \mathbb{Q}(j)) \simeq H^i_{\cdh}(X, \pi_* (\mathbb{Q}(j)|_Y)).$$ 
Applying $H^i_{\cdh}(X, -)$ to $\mathrm{Tr}$ then yields a global trace map on rational cdh-motivic cohomology:
    $$ \mathrm{Tr}_\pi\colon  H^i_{\cdh}(Y, \mathbb{Q}(j)) \xrightarrow{\simeq} H^i_{\cdh}(X, \pi_* (\mathbb{Q}(j)|_Y)) \xrightarrow{H^i(\mathrm{Tr})} H^i_{\cdh}(X, \mathbb{Q}(j)). $$
At the same time, the finite surjective map $\pi$ also induces a pullback map
    $$\pi^*\colon H^i_{\cdh}(X, \mathbb{Q}(j)) \to H^i_{\cdh}(Y, \mathbb{Q}(j)).$$
At the level of finite correspondences, $\Gamma_\pi \circ \Gamma_\pi^t$ is given by multiplication by the degree $d = \deg(\pi)$ (cf. \cite[\S1]{MVW}), so
    $$ \mathrm{Tr}_\pi \circ \pi^* = \cdot d,$$
and this immediately implies the statement.

\end{proof}

\begin{remark}
    The proof shows that the trace map $\mathrm{Tr}_{\pi}$ exists over arbitrary fields. However, in characteristic $p|d$, it does not provide a splitting of $\pi^*$.
\end{remark}

\begin{remark}[{\bf Trace map on $K$-groups.}]
    For a finite map $\pi\colon Y\to X$ between normal varieties, it follows from the proof of Proposition \ref{prop finite map} that there are trace maps
    \[KH_j (Y)\to KH_j (X)\]
    for all $j$. Moreover, when $Y$ has rational and pre-$m$-Du Bois singularities for all $m$, it follows that $X$ has pre-$m$-Du Bois singularities for all $m$, by \cite[Proposition 4.2(1)]{SVV}. In this case the isomorphism between $KH$ and $K$-groups implied by Theorem \ref{thm rosie} gives trace maps
    \[K_{-i}(Y)\to K_{-i}(X) \,\,\,\,\,\, {\rm for~ all}  \,\,\,\,i\ge s = \dim X_{\rm sing}.\]
\end{remark}

\begin{corollary}\label{cor smooth cover}
Let $X$ be a normal variety of dimension $n$. If there exists a finite surjective map $\pi\colon Y \to X$ from a smooth variety $Y$, then
$$H^i_{\cdh}(X, \mathbb{Q}(j))=0 \,\,\,\,\,\,{\rm for} \,\,\,\,i>2j \,\,\,\,{\rm or}\,\,\,\,i > j + n.$$
In particular, this applies to any quasi-projective variety $X$ with quotient singularities.
\end{corollary}

\begin{proof}
The first part follows from Proposition \ref{prop finite map}, together with (\ref{eqn:van-smooth-1}) and (\ref{eqn:van-smooth-2}).
 For the second part, as pointed out in \cite{GS} before Remark 1.9, 
 it is known that if $X$ is quasi-projective and has quotient singularities, 
 there is a finite surjective map $\pi\colon Y \to X$ from a smooth variety $Y$. 
\end{proof}

We deduce the main Conjectures \ref{main-KH-conjecture} and \ref{main-K-conjecture} for quotient singularities.

\begin{corollary}
Let $X$ be a quasi-projective variety with quotient singularities. Then 
$$KH_{-i} (X)_{\mathbb{Q}} = 0\,\,\,\,\,\,{\rm for~ all}\,\,\,\, i>0$$
and if $s = \dim X_{\rm sing}$, then
$$K_{-i} (X)_{\mathbb{Q}} = 0  \,\,\,\,\,\,{\rm for~ all}\,\,\,\, i\ge \max ~\{s, 1\}.$$
\begin{proof}
    The vanishing of $KH_{-i} (X)$ follows from Corollary \ref{cor smooth cover} and the Atiyah-Hirzebruch type spectral sequence 
    in Theorem \ref{thm motivic cohom}.
The $KH$-groups and $K$-groups are equal for $i \ge s$ by Theorem \ref{thm rosie}(ii), since $Y$ is pre-$m$-Du Bois (in fact pre-$m$-rational) for all $m$.
\end{proof}
\end{corollary}

\begin{remark}
    If $X = Y/G$ is a global quotient, then taking $d=|G|$ in the proof of Proposition \ref{prop finite map} shows that $KH_{-i}(X)=K_{-i}(X)$ has $|G|$-torsion when $i>0$ is at least the dimension of the singular locus of $X$. Note that this is consistent with the example from \cite{projective quotients} cited above.
    \end{remark}

\subsection{Affine and projective cones}\label{scn:cones}
The negative $K$-groups of an affine cone $X$ over a projective variety are understood. First, by \S\ref{scn:graded}, we have  
$$KH_{-i}(X) = 0\,\,\,\,\,\,{\rm  for~all} \,\,\,\, i > 0.$$ 
This verifies the affine version of Conjecture \ref{main-KH-conjecture}, over $\ZZ$: affine cones are contractible, hence $H^i (X, \ZZ) = 0$ for all $i > 0$. Next, Theorem \ref{thm rosie}(ii)  says that
the affine version of Conjecture \ref{main-K-conjecture} is also satisfied integrally in this case.

Better still, there is a general description. The $K$-theory of a cone over a smooth projective variety $Z$ was computed in \cite[Theorem 1.2]{cones}. That argument plus an extra calculation yields an extension \cite[Corollary D]{singular cones} to cones over arbitrary projective varieties, as follows.

\begin{theorem}\label{cone K}
Let $Z$ be an $(n-1)$-dimensional complex projective variety endowed with an ample line bundle $L$, and let $X = C(Z, L)$ be the affine cone over $X$ associated to $L$. Then
\begin{align*}
    K_{-i}(X) = \bigoplus_{j=0}^{n-1-i} \bigoplus_{k \geq 1} \mathbb{H}^{i+j}(Z, \underline{\Omega}^j_{Z / \mathbb{Q}} \otimes L^k) \,\,\,\,\,\,{\rm for ~all} \,\,\,\,i > 0.
\end{align*}
\end{theorem}

Here $ \underline{\Omega}^j_{Z /\mathbb{Q}}$ are the Du Bois complexes over $\Q$ of the singular variety $Z$. When $Z$ is smooth, they coincide with the K\"ahler differentials over $\Q$.

Using this, the vanishing of $K$-groups in the required range can also be checked directly using Du Bois theory; it amounts to showing 
\[\mathbb{H}^{i+j}(Z, \underline{\Omega}^j_{Z / \mathbb{Q}} \otimes L^k)=0\]
for all $i>0$, $i\ge n-2m-2 + s$, $0\le j\le n-1-i$ and $k\ge 1$, and  it is not hard to verify this using computations in  \cite{singular cones} and \cite{rosie}.

\smallskip

\noindent 
\emph{The projective case.}
This time we take  $X$ to be a projective cone over a projective variety $Z$. We claim that 
$$KH_{-i}(X) \simeq KH_{-i} (Z) \,\,\,\,\,\,{\rm  for~all}\,\,\,\, i>0,$$
so in particular $KH_{-i}(X) = 0$ for all $i> 0$ if $Z$ is smooth.

Indeed, denoting by $x$ the cone point of $X$, we consider the abstract blowup square (where $\widetilde{X}$ need not be smooth): 
\[\begin{tikzcd}
    Z \ar[r]\ar[d] &\widetilde{X}\ar[d] \\
    x\ar[r] &X
\end{tikzcd}\]
Being a cdh square, this induces a long exact sequence
\begin{align*}
\cdots \to KH_k(X) &\to KH_k(\widetilde{X}) \oplus KH_k(x) \to KH_k(Z) \to \\
KH_{k-1}(X) &\to KH_{k-1}(\widetilde{X}) \oplus KH_{k-1}(x) \to KH_{k-1}(Z) \to \cdots
\end{align*}
Since $\widetilde{X}$ is a $\mathbb{P}^1$-bundle over $Z$, the maps
\[KH_k(\widetilde{X})\to KH_k(Z)\]
are surjective by the natural splitting given by the zero section. Moreover, by the projective bundle formula \cite[Theorem 7.3]{derived categories}\footnote{Theorem 7.3 in \textit{loc. cit.} gives the projective bundle formula for $K^B$, the nonconnective spectrum that gives Bass' negative $K$-theory. Taking cdh sheafification gives the result for $KH$.} we have
\[KH(\widetilde{X})\simeq KH(Z)\oplus KH(Z),\]
and plugging this into the exact sequence above gives the claim.

This shows that Conjecture \ref{main-KH-conjecture} holds for $X$ if and only if it holds for $Z$, which is for instance the case (integrally) when $Z$ is smooth. This is due to a standard calculation showing that  
$$\gr^W_j H^k(X,  \Q)\simeq \gr^W_{j-2} H^{k-2}(Z, \Q).$$
Indeed, when $Z$ is smooth, the projective bundle formula (see \cite[Lemma 7.32]{voisin i})  implies that for $k \ge 2$ we have 
$H^k(X, \Q)\simeq H^{k-2}(Z, \Q)(-1)$. When $Z$ is singular, one can use a hyperresolution to reduce to the smooth case.

Consequently, Conjecture \ref{main-K-conjecture} for $X$ is also satisfied (integrally) when $Z$ is smooth.

\subsection{Toric varieties}\label{scn:toric}
There is a formula for the rational $KH$-theory of a projective toric variety, in terms of its weight filtration; see \cite[Corollary~2.5]{toric} for an exposition of unpublished work of B. Totaro.

\begin{theorem}[Totaro]
If $X$ is a projective toric variety, then for all $i$
\begin{align*}
    KH_i(X)_\mathbb{Q} \simeq \bigoplus_{p,q \geq 0} \bigg( \gr_{2p+2i-2q}^W H^{2p+i-q}(X, \mathbb{Q}) \otimes K_q(\mathbb{C}) \bigg).
\end{align*}
\end{theorem}
It  is easily seen that this implies our main Conjecture \ref{main-KH-conjecture} for projective toric varieties.

Since toric varieties are seminormal and have pre-$m$-Du Bois singularities for all $m \geq 0$ \cite[Proposition E]{SVV}, Theorem
 \ref{thm rosie}(ii) implies $K_i(X) \simeq KH_i(X)$ for all $i \le  0$. We deduce:

\begin{corollary}\label{cor:K-toric}
If $X$ is a projective toric variety, then for all $i \leq 0$
\begin{align*}
    K_i(X)_\mathbb{Q} \simeq \bigoplus_{p,q \geq 0} \bigg( \gr_{2p+2i-2q}^W H^{2p+i-q}(X, \mathbb{Q}) \otimes K_q(\mathbb{C}) \bigg).
\end{align*}
In particular, if $X$ is simplicial then $K_{i}(X)_\mathbb{Q} = 0$ for $i  < 0$.
\end{corollary}

\begin{proof}
We are only left with justifying the last statement. To this end, recall that simplicial toric varieties have quotient singularities, so their Hodge structures are pure.
\end{proof}

The  statements  in Corollary \ref{cor:K-toric} seem to not have been noted before. Moreover, if $X$ is any affine toric variety, $K_i(X) = KH_i(X) = 0$ for all $i < 0$ by \cite[\S 1.2(3)]{toric}. This matches our conjectures; indeed, at least if $X$ is defined by a full-dimensional affine cone $\sigma$,  a lattice point $v \in \mathrm{Int}(\sigma)$ defines a one-parameter subgroup with a unique fixed point onto which $X$ deformation retracts. Since $X$ is contractible, its Hodge structures are certainly pure.

\subsection{Varieties with short hyperresolution}\label{scn:short-hyp}
If the projective variety $X$ has a hyperresolution of length $\le \ell$, then $H^k(X,\Q)$ has weights $\ge k- \ell+1$ by the 
weight spectral sequence (\ref{eqn:weight-ss-rational}), so Conjecture \ref{main-KH-conjecture} predicts $KH_{-i} (X)_{\Q}=0$ for $i> \ell - 1$. 
This in fact holds integrally by considering the spectral sequence 
    \[E_1^{p,q}=KH_q(X_p)\implies KH_{q-p}(X).\]
computing $KH$ in terms of the hyperresolution.

\begin{example}[{\bf Secant varieties}]
Here is an interesting example that has been studied extensively recently. Let $X = \Sigma(Y,L)$ be the secant variety associated to a smooth projective variety $Y$ and a $3$-very ample line bundle $L$ on $Y$. Then there is a blowup square
\[\begin{tikzcd}
W \arrow[r, hook] \arrow[d] & Z \arrow[d] \\
Y  \arrow[r, hook]              & X 
\end{tikzcd}\]
in which $Y$, $Z$ and $W$ are all smooth; see \cite[\S2.1]{secant}.

The existence of such a diagram implies that $H^k(X, \Q)$ has weights $\in \{k-1, k\}$. By the discussion above, we have 
$$KH_{-i} (X) =0 \,\,\,\,\,\,{\rm for ~all} \,\,\,\, i > 1.$$ 
The pre-$m$-Du Bois property of such secant varieties is studied in detail in \emph{loc. cit.} in terms of the geometric properties of $L$; in any case,
$X$ also satisfies Conjecture \ref{main-K-conjecture} integrally.
\end{example}

\section{Appendix}

\subsection{Applying the (generalized) Bloch-Beilinson conjecture}\label{scn:BB-conj}
Here we show that Conjecture \ref{main-cdh-conjecture} follows from a version of the Bloch-Beilinson conjectural filtration on (higher) Chow groups.

Let us first recall Bloch's conjecture on $0$-cycles on surfaces, used in the proof of Theorem \ref{thm:Chow-threefolds}. It was formulated as a converse to Mumford's theorem saying that for 
surfaces with $h^{2,0}(S) \neq 0$, the Chow group of $0$-cycles is infinite dimensional.

\begin{conjecture}\label{conj:Bloch}
If $S$ is a smooth projective surface for which $h^{2,0}(S) = 0$, then the Albanese map $CH_0(S)_{\mathrm{hom}} \to \Alb(S)$ is an isomorphism.
\end{conjecture}

At least up to torsion, this is a special case of more general conjecture of Bloch and Beilinson, who predicted the existence of a filtration on the Chow groups of a smooth projective variety which witnesses the fact that Hodge theory plays an important role in governing Chow groups.

\begin{conjecture}\cite[Section $11.2.2$]{voisin ii}\label{conj bb}
Let $X$ be a smooth complex projective variety. Then there exists a decreasing filtration $F^\bullet CH^j(X)_\mathbb{Q}$ on the Chow groups of $X$ with rational coefficients, such that:
\begin{enumerate}
    \item $F^0 CH^j(X)_\mathbb{Q} = CH^j(X)_\mathbb{Q}$, $F^1 CH^j(X)_\mathbb{Q} = CH^j(X)_{\mathrm{hom}, \mathbb{Q}}$, and $F^2 CH^j(X)_\mathbb{Q} = \ker AJ_X$,
    \item $F^{j+1} CH^j(X)_\mathbb{Q} = 0$,
    \item The filtration $F^\bullet CH^j(X)_\mathbb{Q}$ is preserved by correspondences: the map $\Gamma_*$ induced by a cycle $\Gamma \in CH^\ell(X \times Y)$ satisfies
    \begin{align*}
        \Gamma_*(F^\nu CH^j(X)_\mathbb{Q}) \subset F^\nu CH^{j+\ell-\dim X}(Y)_\mathbb{Q},
    \end{align*}
    \item The induced map 
    \begin{align*}
        \gr^\nu_F \Gamma_*: \gr^\nu_F CH^j(X)_\mathbb{Q} \rightarrow \gr^\nu_F CH^{j+\ell-\dim X}(Y)_\mathbb{Q}
    \end{align*}
    vanishes if the map of map of cohomology groups
    \begin{align*}
        [\Gamma]_*: H^{2j-\nu}(X, \mathbb{Q}) \rightarrow H^{2j-\nu + 2\ell -2 \dim X}(Y, \mathbb{Q})
    \end{align*}
    is zero.
\end{enumerate}
\end{conjecture}

\begin{remark}\label{remark conj hodge}
It is in fact expected that $\gr_F^\nu$ vanishes if $[\Gamma]_*$ vanishes only on the Hodge pieces $H^{r,s}$ of $H^{2j-\nu}(X, \mathbb{Q})$ for $s \leq j - \nu$. See \cite[\S11.2.2]{voisin ii}. We will refer to this stronger version of condition $(4)$ as condition $(4')$ throughout. We choose to distinguish between $(4)$ and $(4')$ because it seems that $(4')$ does not follow from standard expectations regarding mixed motives, unlike $(4)$.

Assuming condition $(4')$ for a surface $S$, if $h^{2,0}(S) = 0$, then $\gr_F^2 CH^2(S)_\mathbb{Q} = 0$. This implies Bloch's conjecture, up to torsion.
\end{remark}

\begin{remark}[{\bf Higher Chow groups}]\label{remark higher chow}
The situation is expected to be similar for higher Chow groups; cf. for instance \cite[\S4]{Asakura}.  We take the Bloch-Beilinson conjecture \ref{conj bb} to also include the existence of a decreasing filtration $F^\bullet CH^j(X, i)_\mathbb{Q}$ on the higher Chow groups of $X$ with rational coefficients, such that $F^{j-i+1} CH^j(X, i)_\mathbb{Q}$ = 0, $F^\bullet$ respects correspondences, and $\gr_F^\nu \Gamma_*$ vanishes if $[\Gamma]_*$ is zero on $H^{2j-i-\nu}(X, \mathbb{Q})$

Analogously to Remark \ref{remark conj hodge}, we take condition $(4')$ to say $\gr_F^\nu \Gamma$ vanishes if $[\Gamma]_*$ is zero on the Hodge pieces $H^{r,s}$ of $H^{2j-i-\nu}(X, \mathbb{Q})$ for $s \leq j-\nu$.
\end{remark}

We will need a stronger version of Conjecture \ref{conj bb} and  Remark \ref{remark higher chow}.  Consider a semisimplicial variety 
\[
\begin{tikzcd}[row sep=1.4em, column sep=3.5em]
X_\bullet = \bigg( \cdots X_2 \arrow[r] \arrow[r, shift left=0.9ex] \arrow[r, shift right=0.9ex] & 
X_1 \arrow[r] \arrow[r, shift left=0.6ex] &
X_0 \bigg),
\end{tikzcd}
\]
where each $X_a$ is smooth and projective.
By taking the alternating sum of face maps, there are induced complexes of (higher) Chow groups
\begin{align*}
  0 \to   \gr_F^\nu CH^j(X_0, i)_\mathbb{Q} \rightarrow \gr_F^\nu CH^j(X_1, i)_\mathbb{Q} \rightarrow \gr_F^\nu CH^j(X_2, i)_\mathbb{Q} \rightarrow \cdots
\end{align*}
Conditions $(4)$ or  $(4')$ in Conjecture \ref{conj bb} give a criterion under which a map of Chow groups vanishes, but we would also like to understand 
whether a complex as above is exact. We propose the following version of the Bloch-Beilinson conjecture, designed for this purpose.

\begin{conjecture}\label{conj bb exact}
The Chow complex
\begin{align*}
    \cdots \rightarrow \gr_F^\nu CH^j(X_{a-1},i)_\mathbb{Q} \rightarrow \gr_F^\nu CH^j(X_a,i)_\mathbb{Q} \rightarrow \gr_F^\nu CH^j(X_{a+1},i)_\mathbb{Q} \rightarrow \cdots 
\end{align*}
is exact in the middle if the complex of pure Hodge structures
\begin{align*}
    \cdots \rightarrow H^{2j-i-\nu}(X_{a-1}, \mathbb{Q}) \rightarrow H^{2j-i-\nu}(X_{a}, \mathbb{Q}) \rightarrow H^{2j-i-\nu}(X_{a+1}, \mathbb{Q}) \rightarrow \cdots 
\end{align*}
is exact. In fact, we extend condition $(4')$ to say it is sufficient to assume exactness for the Hodge pieces $H^{r,s} H^{2j-i-\nu}(X_a, \mathbb{Q})$ of degree $s \leq j - \nu$ only.

In particular, if this holds for all $\nu \ge 0$, the Chow complex
\begin{align*}
    \cdots \rightarrow \CH^j(X_{a-1},i)_\mathbb{Q} \rightarrow CH^j(X_a,i)_\mathbb{Q} \rightarrow CH^j(X_{a+1},i)_\mathbb{Q} \rightarrow \cdots 
\end{align*}
is exact in the middle.
\end{conjecture}

Note that this last prediction makes no reference to the Bloch-Beilinson filtration.

\begin{remark}
When the semisimplicial variety arises from a simple normal crossing space, pieces of such complexes have made an appearance in the literature before. 
The paper \cite{LS} (especially \S4) makes very nice use of such a construction for rationality questions, with roots going back to the earlier \cite{PSch}.
\end{remark}

\begin{remark}
While the conjectural statement above might be new, it is a consequence of the same circle of ideas leading to the original Bloch-Beilinson conjecture. 
Namely, it can be derived from the conjectural picture regarding the properties of the derived category of mixed motives, as for instance in 
 \cite{jannsen}, \cite{motives} and \cite{ayoub}.
\end{remark}




\smallskip

We are primarily interested in the case where $\epsilon_\bullet \colon X_\bullet \rightarrow X$ is a cubical hyperresolution of a projective variety $X$ (see \cite[Example $5.4$]{peters steenbrink}). In this case, 
the complex of pure Hodge structures  in the statement of Conjecture \ref{conj bb exact} computes the weight piece $\gr_{2j-i-\nu}^W H^{2j-i-\nu +a}(X, \mathbb{Q})$, according to (\ref{eqn:weight-hyper}). We obtain the following:

\begin{lemma}\label{lem sing chow complex}
Assume Conjecture \ref{conj bb exact}, and let $\epsilon_\bullet \colon X_\bullet \rightarrow X$ be a cubical hyperresolution of a projective variety $X$. The Chow complex 
\begin{align*}
    \cdots \rightarrow \CH^j(X_{a-1},i)_\mathbb{Q} \rightarrow CH^j(X_a,i)_\mathbb{Q} \rightarrow CH^j(X_{a+1},i)_\mathbb{Q} \rightarrow \cdots 
\end{align*}
is exact in the middle if the Hodge structures $\gr_{2j-i-\nu}^W H^{2j-i-\nu +a}(X, \mathbb{Q})$ vanish for all $\nu \geq 0$. In fact, assuming condition $(4')$, it suffices to assume
$$H^{r,s}\gr_{2j-i-\nu}^W H^{2j-i-\nu +a}(X, \mathbb{Q}) = 0 \,\,\,\,{\rm  for~ all}  \,\,\,\,s \leq j - \nu {\rm ~and ~} \nu \geq 0.$$
\end{lemma}

\smallskip

We apply this to study the cdh-motivic cohomology groups of $X$. Recall there is a descent spectral sequence 
$$E_1^{p,q} = H^q (X_p, \mathbb{Q}(j)) \Rightarrow H^{p+q}_{\cdh}(X, \mathbb{Q}(j)),$$ 
where the $E_1$-terms are motivic cohomology groups of the smooth variety $X_p$, so we have an identification 
$$H^q (X_p, \mathbb{Q}(j)) \simeq CH^j(X_p, 2j-q)_\mathbb{Q}.$$

\smallskip

\begin{proposition}\label{thm sing mot cohom}
Conjecture  \ref{conj bb exact} implies Conjecture \ref{main-cdh-conjecture}. In other words, if 
 $X$ is a projective variety, and we fix integers $j$ and $k$,  we have
$H^k_{\cdh}(X, \mathbb{Q}(j)) = 0$  if
$$W_{b-k + 2j} H^b(X, \mathbb{Q}) = 0 \,\,\,\,\,\,{\rm for ~all } \,\,\,\,  b \leq k.$$
In fact, assuming condition $(4')$, it suffices to assume
\begin{align*}
    H^{r,s} \gr_{b-c}^W H^{b}(X, \mathbb{Q}) = 0 \,\,\,\, {\rm for ~all } \,\,\,\, b 
    \leq k, c \geq k-2j, \text{ and } r \leq b+j-k.
\end{align*}
Moreover, if this holds except when $b= k$, $c = k-2j$ and $r = j$, then 
$$H^k_{\rm cdh} (X, \Q(j) ) \simeq \gr^W_{2j} H^k(X, \Q).$$
\end{proposition}
\begin{proof}
The descent spectral sequence implies that $H^k_{\cdh}(X, \mathbb{Q}(j))$ vanishes if $E_2^{p,q} = 0$ for $p + q = k$. This is the same as saying the Chow complex
\begin{align*}
    \cdots \rightarrow \CH^j(X_{p-1},2j-q)_\mathbb{Q} \rightarrow CH^j(X_p,2j-q)_\mathbb{Q} \rightarrow CH^j(X_{p+1},2j-q)_\mathbb{Q} \rightarrow \cdots 
\end{align*}
is exact. Note we have vanishing for trivial reasons, unless $q \leq 2j$. Lemma \ref{lem sing chow complex} implies, assuming condition $(4')$, that the above complex is exact if
\begin{align*}
    H^{r,s} \gr_{q-\nu}^W H^{k-\nu}(X, \mathbb{Q}) = 0 \text{ for all } s \leq j-\nu \text{ and } q \leq 2j, \nu \geq 0.
\end{align*}
Let us assign new variables $b = k - \nu$ and $c = k-q$. The previous condition becomes
\begin{align*}
    H^{r,s} \gr_{b-c}^W H^{b}(X, \mathbb{Q}) = 0 \text{ for all } b 
    \leq k, c \geq k-2j, \text{ and } s \leq b+j-k.
\end{align*}
Without condition $(4')$, we may simply discard the Hodge filtration variables $r$ and $s$, and the condition becomes simply the vanishing of certain weights.
\end{proof}

\begin{remark}[{\bf Beilinson-Soul\'e conjecture}]\label{rmk:BS-conj}
If $X$ is smooth and projective, the conclusion of Theorem \ref{thm sing mot cohom} includes the statement that 
$$H^k (X, \mathbb{Q}(j)) = CH^j(X, 2j-k)_\Q = 0 \,\,\,\,{\rm for} \,\,\,\,k < 0.$$ 
This was conjectured by Beilinson and Soul\'e, and is open (see \cite[Lemma $4.13$]{jannsen}).
\end{remark}

Conjecture \ref{main-KH-conjecture}, our main prediction for  the negative $KH$-groups of $X$, is a consequence.

\begin{corollary}\label{cor sing KH}
Assume Conjecture \ref{conj bb exact}. Let $X$ be a projective variety and fix an integer $i$. Then $KH_{-i}(X)_\mathbb{Q} = 0$ if
\begin{align*}
    W_{j -i} H^j(X, \mathbb{Q}) = 0 \,\,\,\,\,\,{\rm  for ~all}\,\,\,\, j \le 2n - i.
\end{align*}
\end{corollary}

\begin{proof}
The Atiyah-Hirzebruch type spectral sequence in Theorem \ref{thm motivic cohom} provides a decomposition
\begin{align*}
    KH_{-i}(X)_\mathbb{Q} = \bigoplus_{j \geq 0} H^{2j+i}_{\cdh}(X, \mathbb{Q}(j)).
\end{align*}
We then apply Theorem \ref{thm sing mot cohom} for each of the summands. 
\end{proof}

\begin{remark}
Assuming condition $(4')$, each of the summands in the decomposition above vanishes if, for $j \geq 0$,
\begin{align}\label{eqn vanishing KH proof}
    H^{r,s} \gr_{b-c}^W H^{b}(X, \mathbb{Q}) = 0 \text{ for all } b 
    \leq 2j+i, c \geq i, \text{ and } s \leq b-j-i.
\end{align}
It appears that this proves more than the statement of Corollary \ref{cor sing KH}. It turns out it is actually equivalent to that statement, as we presently explain. Suppose that $\gr_{b-c}^W H^b(X, \mathbb{Q}) \neq 0$ for some $c \geq i$ and some $b$. There exists a non-vanishing Hodge piece $H^{r,s}$ where $s \leq (b-c)/2$. It suffices to choose $j \geq 0$ such that all the above inequalities are satisfied. We set $j = b-i-s \geq b-c-s \geq 0$. The inequality $2s+i \leq (b-c)+i \leq b$ says precisely that $b \leq 2j + i$.
\end{remark}

\subsection{Non-negative $K$-theory for varieties with pure Hodge structures}
Our main conjecture regarding negative $KH$-theory is most powerful for varieties all of whose Hodge structures are pure, when it predicts that all of the negative $KH$-groups are torsion. In this setting it can be extended to the non-negative range,  in the following way.

\begin{conjecture}\label{conj nonnegative KH}
Suppose $X$ is a projective variety, all of whose Hodge structures are pure. If $\epsilon_\bullet \colon X_\bullet \rightarrow X$ is a cubical hyperresolution, then for all $i \geq 0$
\begin{align}\label{eqn nonnegative KH}
    KH_i(X)_\mathbb{Q} \simeq \bigoplus_{j \geq 0} \ker \bigg( CH^j(X_0, i) \rightarrow CH^j(X_1, i) \bigg)_\mathbb{Q}.
\end{align}
In fact, the $j$th summand is the cdh cohomology group $H^{2j-i}_{\cdh}(X, \mathbb{Q}(j))$.

Moreover, if $X$ is seminormal with pre-$m$-Du Bois singularities for all $m$ (for example, when $X$ has quotient singularities), then $K_0(X) \simeq KH_0(X)$, described over $\Q$ by the formula above.
\end{conjecture}

Note that, thanks to (\ref{eqn:van-smooth-1}),  the conjecture predicts that $CH^n (\widetilde{X}, i)_\Q$ injects in $KH_i(X)_\mathbb{Q}$ as a direct summand if $j \ge i + \dim X$, for any resolution of singularities $\widetilde{X}$ of $X$. In particular, it predicts that $CH^n (\widetilde{X})_\Q$ is a 
direct summand of $KH_0 (X)_\mathbb{Q}$.

One can compare  (\ref{eqn nonnegative KH}) with the well-known decomposition for $i \geq 0$, when $X$ is smooth:
\begin{align}\label{eqn K_i decomp}
    K_i(X)_\mathbb{Q} \simeq \bigoplus_{j \geq 0} CH^j(X,i)_\mathbb{Q}.
\end{align}
More generally, by \cite[vii]{higher cycles}, (\ref{eqn K_i decomp}) holds for any variety $X$, if we replace $K_i(X)_\mathbb{Q}$ with the $G$-theory group $G_i(X)_\mathbb{Q}$. (Recall that the Poincar\' e duality map $K_i(X) \rightarrow G_i(X)$ is an isomorphism whenever $X$ is smooth.)

We show that the prediction above, just like the conjectures on negative $K$-theory, is also implied by the generalized Bloch-Beilinson conjecture in the previous section.

\begin{proposition}\label{prop BB imples nonnegative KH}
Conjecture \ref{conj bb exact} implies Conjecture \ref{conj nonnegative KH}.
\end{proposition}
\begin{proof}
Consider the descent spectral sequence 
$$    E_1^{p,q} = H^q(X_p, \mathbb{Q}(j)) \simeq CH^j(X_p, 2j-q)_\mathbb{Q} \Rightarrow H_{\cdh}^{p+q}(X, \mathbb{Q}(j)).$$
Then clearly
\begin{align}\label{eqn E_2 for nonnegative KH}
    E_2^{0, 2j-i} \simeq  \ker \bigg( CH^j(X_0, i)_\mathbb{Q} \rightarrow CH^j(X_1, i)_\mathbb{Q} \bigg).
\end{align}
If we assume Conjecture \ref{conj bb exact}, then Lemma \ref{lem sing chow complex} and the purity of the Hodge structures on $X$ imply that $E_2^{p,q} = 0$ for all $p \neq 0$. Thus the spectral sequence degenerates at the $E_2$ page, and
\begin{align*}
    H^{2j-i}_{\cdh}(X, \mathbb{Q}(j)) \simeq E_\infty^{0,2j-i} \simeq E_2^{0,2j-i}.
\end{align*}
We now use the Atiyah-Hirzebruch spectral sequence in Theorem \ref{thm motivic cohom}. Note that we may pull the tensor product in (\ref{eqn E_2 for nonnegative KH}) outside of the kernel, because $\mathbb{Q}$ is a flat $\mathbb{Z}$-module. The final statement is a consequence of Theorem \ref{thm rosie}.
\end{proof}

Unconditionally the conjecture seems quite difficult, even in the case $i = 0$. We prove it for $KH_0$ of a surface, but already $KH_0$ of a threefold or $KH_1$ of a surface seem out of reach.

\begin{theorem}\label{KH_0 surface}
Suppose $S$ is a projective surface, all of whose Hodge structures are pure. If $\epsilon_\bullet \colon S_\bullet \rightarrow S$ is a cubical hyperresolution such that $\dim S_i \leq 2 -i$ for each $i$, then
\begin{align*}
    KH_0(S)_\mathbb{Q} \simeq \bigoplus_{j = 0}^2 \ker \bigg( CH^j(S_0) \rightarrow CH^j(S_1) \bigg)_\mathbb{Q}.
\end{align*}
In particular $KH_0(S)_\mathbb{Q}$ contains $CH^2 (\widetilde S)_\Q$ as a direct summand, for any resolution $\widetilde{S}$ of $S$.
\end{theorem}

\begin{proof}
Using the notation in the proof of Proposition \ref{prop BB imples nonnegative KH}, it suffices to prove for each weight $j= 0,1,2$ the vanishing $E_2^{p,q} = 0$, for $p \neq 0$ and $p+q \geq 2j$. We treat the three weights separately.

In weight $j = 0$, we have $E_1^{p,q} = H^q(S_p, \mathbb{Q}(0))= 0$ for $q \neq 0$ by Theorem \ref{thm motivic cohom}$(3)$. Thus, $E_2^{p,0} \simeq H^p_{\cdh}(S, \mathbb{Q}(0))$. The latter space is identified with $W_0 H^p(S, \mathbb{Q})$ by Corollary \ref{cor:cdh-weight-zero}, and this vanishes for $p > 0$ by the purity of Hodge structures. This proves the required vanishing in weight zero.

In weight $j = 1$, we have $E_1^{p,q} =H^q(S_p, \mathbb{Q}(1)) = 0$ for $q \neq 1$ or $2$ by Theorem \ref{thm motivic cohom}$(3)$. Moreover, this shows $E_2^{p,1} \simeq H^{p}(S_\bullet , \mathscr{O}_{S_\bullet}^*)_\mathbb{Q}$ and $E_2^{p,2} \simeq H^{p}(\Pic(S_\bullet))_\mathbb{Q}$ in the notation of \S \ref{scn:weight1}. Both of these spaces vanish for $p > 0$ by Corollary \ref{cor weight one rational} and the purity of Hodge structures.

Finally, in weight $j = 2$, we have the vanishing $E_1^{p,q} = H^q(S_p, \mathbb{Q}(2)) =0$ for $p + q > 4$ by Lemma \ref{lem motivic soule-weibel}. It remains to show that $E_2^{p,q} = 0$ for $p + q = 4$ and $p > 0$, that is to say $E_2^{1,3} = E^{2,2}_2 = 0$. In other words, we must show surjectivity of the morphisms
\begin{align*}
    CH^2(S_0, 1)_\mathbb{Q} \rightarrow CH^2(S_1, 1)_\mathbb{Q}
\end{align*}
and
\begin{align*}
    CH^2(S_1, 2)_\mathbb{Q} \rightarrow CH^2(S_2,2)_\mathbb{Q}.
\end{align*}
Their surjectivity follows from the arguments proving Lemmas \ref{lem beta} and \ref{lem gamma}, respectively, again using the purity of Hodge structures.
\end{proof}

\begin{remark}[{\bf Surfaces with rational singularities}]
If $X$ is a normal projective surface, the purity of all Hodge structures is equivalent to $X$ being a rational homology manifold. This always holds when $X$ has rational singularities. The converse is also true when $X$ has Du Bois (so for instance log canonical) singularities. For all of these statements, see \cite[Theorem 11.1]{PP1} and the discussion around it.

Thus, under mild assumptions, Theorem \ref{KH_0 surface} is about surfaces with rational singularities. Moreover, such a surface is also pre-$m$-Du Bois for all $m$, by \cite[Proposition C]{PSV}. Therefore the decomposition in the theorem applies to $K_0 (S)_\Q$.
\end{remark}


\subsection{Miscellanous examples}\label{scn:misc}
We consider a few examples that lead to torsion but nontrivial $K$-groups, showing that our results cannot be upgraded to the integral setting without further assumptions.

\begin{example}[{\bf Integral weights for Kummer varieties}]\label{ex:Kummer}
We generalize Example \ref{example kummer surface} in the following proposition, which provides a useful source of counterexamples. 
The statements refer to the integral weight filtration of Gillet-Soul\'e, discussed in \S\ref{scn:weights-integral}.

\begin{proposition}\label{proposition kummer example}
Let $A$ be an abelian variety of dimension $g$ and let $X = A/\iota$ be the Kummer variety, where $\iota\colon x \mapsto -x$ is the negation involution. Then
\[
H^{i}(X,\mathbb{Z}) \;\cong\;
\begin{cases}
\mathbb{Z}^{\binom{2g}{\,i\,}}, 
& \text{if } 0 \le i \le 2g \text{ and } i \text{ is even},\\[6pt]
(\ZZ/ 2 \ZZ)^{\big(\sum_{j=0}^{2g-i} \binom{2g}{j}\big)}, 
& \text{if } 3 \le i \le 2g-1 \text{ and } i \text{ is odd},\\[6pt]
0, 
& \text{otherwise.}
\end{cases}
\]
When $i$ is even, $H^i(X, \mathbb{Z})$ is pure of weight $i$, and when $i$ is odd, $H^i(X, \mathbb{Z})$ is pure of weight $i-1$.
\end{proposition}

Let us denote the quotient map by $f\colon A \rightarrow X$ and the smooth locus of $X$ by $U$, so that the restriction $\left.f\right|_V \colon V = f^{-1}(U) \rightarrow U$ is the quotient of a free $(\mathbb{Z}/2\mathbb{Z})$-action.

The singular locus $X_\mathrm{sing}$ consists $2^{2g}$ points, each locally analytically of the form $(\mathbb{C}^g/\pm, 0)$ Taking their blow-up yields a log resolution $Y \to X$, and in fact a hyperresolution
\begin{equation}\label{eqn:equation kummer resn}
\begin{tikzcd}
E \arrow[r] \arrow[d] & Y \arrow[d] \\
X_\mathrm{sing} \arrow[r] & X
\end{tikzcd}
\end{equation}
where $E$ is a disjoint union of $2^{2g}$ copies $E_i$ of projective space $\mathbb{P}^{g-1}$, each with normal bundle 
$N_{E_i/Y} \simeq \mathscr{O}(-2)$.

We proceed via a series of lemmas.

\begin{lemma}\label{lemma U cohom}
The cohomology groups of $U$ are
\[
H^{i}(U,\mathbb{Z}) \;\simeq\;
\begin{cases}
\mathbb{Z}^{\binom{2g}{i}} \oplus (\ZZ/ 2 \ZZ)^{\big(\sum_{j=0}^{i-1} \binom{2g}{j} \big)}, 
& \text{if } i \leq 2g-2 \text{ is even},\\[6pt]
\mathbb{Z}^{2^{2g}-1}, 
& \text{if } i = 2g-1,\\[6pt]
0, 
& \text{otherwise.}
\end{cases}
\]
The pullback $\left.f\right|_V^*$ restricts to an isomorphism $H^i(U, \mathbb{Z})_\mathrm{tf} \rightarrow H^i(V, \mathbb{Z}) = H^i(A, \mathbb{Z})$ on the torsion-free part $H^i(U, \mathbb{Z})_{\mathrm{tf}}$ of $H^i(U, \mathbb{Z})$, when $i \leq 2g-2$ is even.
\end{lemma}

\begin{proof}
This was proved in \cite{spanier} as an intermediate step towards computing the cohomology of $Y$. Indeed, the cohomology of $U$ was explicitly recorded in the text. Let us consider the pullback $H^i(U, \mathbb{Z}) \rightarrow H^i(V, \mathbb{Z}) = H^i(A, \mathbb{Z}) = \wedge^i \mathbb{Z}^{2g}$. Since $\mathbb{Z}/2\mathbb{Z}$ acts freely on $V$, and it acts trivially on even degree cohomology, we have an inclusion $2 \cdot H^i(V, \mathbb{Z}) \subset \im \left.f\right|_V^*$. This implies that the induced map $H^i(U, \mathbb{Z})_{\mathrm{tf}} \rightarrow H^i(V, \mathbb{Z})$ is either an isomorphism, or else
\begin{align*}
    \ker \bigg( H^i(U, \mathbb{Z})_{\mathrm{tf}}  \otimes \ZZ/ 2 \ZZ \rightarrow H^i(V, \mathbb{Z}) \otimes \ZZ/ 2 \ZZ \bigg) \neq 0.
\end{align*}
The paper \cite{spanier} also proves that the reduction $H^i(U, \ZZ/ 2 \ZZ) \rightarrow H^i(V, \ZZ/ 2 \ZZ)$ modulo $2$ is surjective. By the universal coefficient theorem, the latter possibility cannot happen. The result follows.
\end{proof}

Next, we have a basic excision result.

\begin{lemma}\label{lemma excision}
The cohomology groups of the pair $(X, U)$ are
\[
H^{i}(X, U; \mathbb{Z}) \;\simeq\;
\begin{cases}
(\ZZ / 2 \ZZ)^{\big( 2^{2g} \big)}, 
& \text{if } 3 \leq i \leq 2g-1 \text{ is odd},\\[6pt]
\mathbb{Z}^{2^{2g}}, 
& \text{if } i = 2g,\\[6pt]
0, 
& \text{otherwise.}
\end{cases}
\]
\end{lemma}

\begin{proof}
By excision, the relative cohomology group $H^i(X,U; \mathbb{Z})$ is the direct sum over the $2^{2g}$ singular points $x \in X$ of the reduced cohomology groups $\Tilde{H}^{i-1}(L_x, \mathbb{Z})$, where $L_x$ is the link. Each link is the quotient of the unit sphere $S^{2g-1}$ by the antipodal involution $x \mapsto -x$, so is homeomorphic to real projective space $\mathbb{RP}^{2g-1}$. The cohomology of $\mathbb{RP}^{2g-1}$ is well known.
\end{proof}

Rationally, the cohomology of $X$ agrees with that of the abelian variety $A$. Integrally however they they differ because of two-torsion.

\begin{lemma}\label{lemma two divisibility}
If $i$ is an even integer, the image of the pullback map $f^*\colon H^i(X, \mathbb{Z}) \rightarrow H^i(A, \mathbb{Z})$ is divisible by two.
\end{lemma}
\begin{proof}
Let $\alpha \in H^i(X, \mathbb{Z})$. To show $f^* \alpha$ is divisible by two, it suffices to show the intersection pairing $(f^* \alpha, \beta)$ is even for every homology class $\beta \in H_i(A, \mathbb{Z})$. The projection formula says $(f^* \alpha, \beta) = (\alpha, f_* \beta)$, so it suffices to show that each $f_*\beta$ is two-divisible.

We may choose a basis $e_i$ of $H_i(A, \mathbb{Z})$ represented by subtori $A_i$ of $A$. These are invariant under the involution $\iota$, so the quotient restricts to a degree two map $A_i \rightarrow f(A_i)$. It follows that $f_* e_i = 2[f(A_i)]$ is divisible by two for each basis element.
\end{proof}

We are now ready to prove the proposition.

\begin{proof}[Proof of Proposition \ref{proposition kummer example}]
Fix an even integer $2 \leq i \leq 2g-4$, and consider the inclusion $j\colon U \rightarrow X$. The long exact sequence of the pair $(X, U)$ breaks up into pieces of the form
\begin{align}\label{eqn kummer sequence}
0 \rightarrow H^{i}(X, \mathbb{Z}) \xrightarrow{j^*} H^i(U, \mathbb{Z}) \rightarrow H^{i+1}(X, U; \mathbb{Z}) \rightarrow H^{i+1}(X, \mathbb{Z}) \rightarrow 0,
\end{align}
where $H^i(U, \mathbb{Z}) \simeq \mathbb{Z}^{\binom{2g}{i}} \oplus (\ZZ/ 2 \ZZ)^{\sum_{j=0}^{i-1} \binom{2g}{j}}$ and $H^{i+1}(X, U; \mathbb{Z}) \simeq (\ZZ/ 2 \ZZ)^{2^{2g}}$ by Lemmas \ref{lemma U cohom} and \ref{lemma excision}. We deduce from this that $H^i(X, \mathbb{Z}) \simeq \mathbb{Z}^{\binom{2g}{i}}$. To compute $H^{i+1}(X, \mathbb{Z})$, we must compute the $\ZZ/ 2\ZZ$-vector space $\coker j^*$.

Consider the commutative diagram
\[
\begin{tikzcd}\label{equation kummer resn}
H^i(A, \mathbb{Z}) \arrow[r, "\simeq"]  & H^i(V, \mathbb{Z})  \\
H^i(X, \mathbb{Z}) \arrow[r, "j_{\mathrm{tf}}^*"] \arrow[u,"f^*"]  & H^i(U, \mathbb{Z})_\mathrm{tf} \arrow[u, "\simeq"]
\end{tikzcd}
\]
The image of $f^*$ is two-divisible by Lemma \ref{lemma two divisibility}, and $H^i(U, \mathbb{Z})_\mathrm{tf} \rightarrow H^i(V, \mathbb{Z}) \simeq H^i(A, \mathbb{Z})$ is an isomorphism by Lemma \ref{lemma U cohom}. It follows that the image of $j_{\mathrm{tf}}^*$ in $H^i(U, \mathbb{Z})_\mathrm{tf}$ is two-divisible, hence $\coker j^*_\mathrm{tf} \simeq (\ZZ/2 \ZZ)^{\binom{2g}{i}}$.

The pullback $j^*$ is of the form
\begin{align*}
    j^* = (j^*_\mathrm{tf}, j^*_\mathrm{tor}): H^i(X, \mathbb{Z}) \rightarrow H^i(U, \mathbb{Z}) \simeq H^i(U, \mathbb{Z})_\mathrm{tf} \oplus H^i(U, \mathbb{Z})_\mathrm{tor}.
\end{align*}
Since $j_{\mathrm{tf}}^*$ is injective, there is a short exact  sequence of $\mathbb{F}_2$-vector spaces
\begin{align*}
    0 \rightarrow (\ZZ/ 2\ZZ)^{\big(\sum_{j=0}^{i-1} \binom{2g}{j} \big)} = H^i(U, \mathbb{Z})_{\mathrm{tor}} \rightarrow \coker j^* \rightarrow \coker j^*_{\mathrm{tf}} \simeq (\ZZ/ 2\ZZ)^{\binom{2g}{i} } \rightarrow 0.
\end{align*}
It splits, so $\coker j^*$ is of rank $\sum_{j=0}^i \binom{2g}{j}$.

Referring back to the sequence (\ref{eqn kummer sequence}), we conclude that $H^{i+1}(X, \mathbb{Z})$ is an $\mathbb{F}_2$-vector space of rank
\begin{align*}
    2^{2g} - \sum_{j=0}^i \binom{2g}{j} = \sum_{i+1}^{2g} \binom{2g}{j} = \sum_{j=0}^{2g-(i+1)} \binom{2g}{j}
\end{align*}
as claimed. This first equality is a well-known combinatorial identity.

We consider the remaining cases. It is clear that $H^0(X, \mathbb{Z}) \simeq \mathbb{Z}$, and $H^1(X, \mathbb{Z}) = 0$ by Lemmas \ref{lemma U cohom} and \ref{lemma excision}. At the other end of the long exact sequence, we have
\begin{align*}
    \cdots \rightarrow H^{2g-1}(X, \mathbb{Z}) \rightarrow H^{2g-1}(U, \mathbb{Z}) \rightarrow H^{2g}(X, U; \mathbb{Z}) \rightarrow H^{2g}(X, \mathbb{Z}) \rightarrow 0.
\end{align*}
We can identify the map $H^{2g-1}(U, \mathbb{Z}) \simeq \mathbb{Z}^{2^{2g}-1} \rightarrow H^{2g}(X, U; \mathbb{Z}) \simeq \mathbb{Z}^{2^{2g}}$ with the boundary $H^{2g-1}(V, \mathbb{Z}) \rightarrow H^{2g}(Y, V; \mathbb{Z})$. It is therefore injective with cokernel isomorphic to $\mathbb{Z}$. This shows $H^{2g}(X, \mathbb{Z}) \simeq \mathbb{Z}$, and our same argument as before now applies to compute $H^{2g-2}(X, \mathbb{Z})$ and $H^{2g-1}(X, \mathbb{Z})$.

It remains to specify weights. Consider the Mayer-Vietoris sequece associated to the resolution square (\ref{eqn:equation kummer resn}):
\begin{align*}
    \cdots \rightarrow H^{i-1}(Y, \mathbb{Z}) \rightarrow H^{i-1}(E, \mathbb{Z}) \rightarrow H^i(X, \mathbb{Z}) \rightarrow H^i(Y, \mathbb{Z}) \rightarrow \cdots
\end{align*}
When $i$ is even, $H^{i-1}(E, \mathbb{Z}) = 0$, so $H^i(X, \mathbb{Z})$ injects into $H^i(Y, \mathbb{Z})$, hence it is pure of weight $i$. When $i$ is odd, $H^i(X, \mathbb{Z})$ is torsion. Since $H^i(Y, \mathbb{Z})$ is torsion-free by \cite{spanier}, we conclude that $H^i(X, \mathbb{Z})$ is a quotient of $H^{i-1}(E, \mathbb{Z})$, hence is pure of weight $i-1$.
\end{proof}
\end{example}

\begin{example}[{\bf $K$-groups of Kummer varieties}]\label{example-K-Kummer}
We continue the previous example, by looking at the $K$-groups of $X = A/\iota$, with $A$ an abelian variety of dimension $g \geq 2$.  Since $X$ has isolated quotient singularities, we know from \S\ref{scn:graded} that it is $K_{-1}$-regular, and moreover
$$K_{-1}(X)_{\Q} = 0 \,\,\,\,\,\,{\rm and} \,\,\,\,\,\, K_{-i} (X) = 0 \,\,\,\,{\rm for}\,\,\,\,i \ge 2.$$
Note that due to the short hyperresolution, in this example the second statement can also be derived from \S\ref{scn:short-hyp}.

We next show that $K_{-1} (X) \neq 0$. The cdh-motivic cohomology groups contributing to $K_{-1}(X) \simeq KH_{-1}(X)$ are of the form
\begin{align*}
    H^{2j+1}_{\cdh}(X, \mathbb{Z}(j)) \simeq & \coker \bigg( H^{2j}_{\cdh}(Y, \mathbb{Z}(j)) \oplus H^{2j}_{\cdh}(\mathrm{Sing}(X), \mathbb{Z}(j)) \rightarrow  H^{2j}_{\cdh}(E, \mathbb{Z}(j)) \bigg) \\
    \simeq & \coker \bigg( CH^j(Y) \oplus CH^j(X_\mathrm{sing}) \xrightarrow{\phi}  CH^j(E) \bigg).
\end{align*}
It is not difficult to show geometrically that these groups are zero after tensoring with $\mathbb{Q}$, and more precisely that they are finite and killed by multiplication by $2$.
Consider for  now for each $1 \leq j \leq g-1$ the commutative diagram
\[\begin{tikzcd}
CH^j(Y) \arrow[r, "\varphi"] \arrow[d] & CH^j(E) \arrow[d, "\simeq"] \\
H^{2j}(Y, \mathbb{Z}) \arrow[r, "\psi"]              & H^{2j}(E, \mathbb{Z}).
\end{tikzcd}\]
The computation of the integral weight filtration on $H^{2j+1}(X, \mathbb{Z})$ in Proposition \ref{proposition kummer example} shows $\psi$ is not surjective. The diagram shows $\varphi$ is not surjective either, from which we deduce that $H^{2j+1}_{\cdh}(X, \mathbb{Z}(j)) \neq 0$.
It is unclear how many of these terms $E_2^{j+1,-j} = H^{2j+1}_{\cdh}(X, \mathbb{Z}(j))$ survive until the $E_\infty$-page of the Atiyah-Hirzebruch spectral sequence. However, we claim that $E_2^{2,-1} = H^{3}_{\cdh}(X, \mathbb{Z}(1))$ does survive. Indeed, the only possibly nontrivial differential is $d_2: E_2^{0,0} \rightarrow E_2^{2,-1}$. But this is zero, because $E_2^{0,0} \simeq E_\infty^{0,0}$ by Remark \ref{rmk:Z-copy}.
\end{example}

\begin{example}[{\bf Cayley surface}]\label{exmp cayley}
We describe a nodal surface with $K_{-1} (S) \neq 0$. This is not a new phenomenon, but will be useful in producing one for threefolds, in the next example.

Consider the Cayley surface $S = ( wxy + xyz + yzw + zwx = 0) \subset \mathbb{P}^3$. It can be constructed as follows. Take four general lines $c_i$ in $\mathbb{P}^2$ and denote their intersections by $p_{ij}$. Consider the blow-up $Y \rightarrow \mathbb{P}^2$ of the $p_{ij}$, with exceptional divisors $e_{ij}$. The proper transforms $\ell_i$ of the lines $c_i$ are $(-2)$-curves, which we contract to obtain the Cayley surface $\pi\colon Y \rightarrow S$. Consider the exact sequence (\ref{eqn:Weibel-sequence}):
\begin{align*}
    \mathrm{Cl}(S) \xrightarrow{a} \bigoplus_{i=1}^4\mathrm{Cl}( \widehat{\mathcal{O}}_{S, s_i}) \rightarrow K_{-1}(S) \rightarrow 0.
\end{align*}
The four singular points $s_i = \pi(\ell_i)$ are of $A_1$ type, so the sum of local analytic class groups is $(\mathbb{Z}/2\mathbb{Z})^4$. Denote the natural basis by $v_i$. It is clear that $a([e_{ij}]) = v_i + v_j$. We claim that this generates the image of $a$. 
Indeed, the map $a$ may be thought be of sending a class $[d]$ to the vector 
$$a(v) = ( \widetilde{d} \cdot \ell_1,\widetilde{d} \cdot \ell_2, \widetilde{d} \cdot \ell_3, \widetilde{d} \cdot \ell_4) \in (\mathbb{Z}/2\mathbb{Z})^4,$$ 
where $\widetilde{d}$ is the proper transform of $d$ in $Y$. But we have a linear equivalence on $Y$
\begin{align*}
    \widetilde{\ell}_1 + \widetilde{\ell}_2 + \widetilde{\ell}_3 + \widetilde{\ell}_4 \sim 2 \big( 2h - \sum_{i \neq j} e_{ij} \big).
\end{align*}
Hence, the components of $a(v)$ sum to zero. We conclude that $K_{-1}(S) \simeq \mathbb{Z}/2\mathbb{Z}$.
\end{example}

\begin{example}[{\bf A klt threefold with $K_{-2} \neq 0$}]\label{exmp 3-fold with cayley}
Consider the Cayley surface $S = ( f(w,x,y,z) = 0)$ of Example \ref{exmp cayley} and take a general quartic $g(w,x,y,z)$. 
Let 
$$X : = (t \cdot f(w,x,y,z) + g(w,x,y,z) = 0 ) \subset \mathbb{P}^4.$$

In the affine chart $t = 1$, $X$ is defined by the equation
\begin{align*}
    f(x,y,z,w) + g(x,y,z,w) = 0.
\end{align*}
This is singular at the origin $p = (0:0:0:0:1)$, around which the lowest degree term is $f$. Thus, the exceptional divisor $E$ of the blow-up $f\colon Y \rightarrow X$ is isomorphic to the projectivized tangent cone $S = ( f =0 )$. It is straightforward to see that $X$ is smooth away from $p$, because the intersection $(f = 0) \cap (g=0) \subset \mathbb{P}^3$ is smooth by the generality of $g$. Moreover, $Y$ is smooth. Since the multiplicity of $X$ at $p$ is three, we have $K_Y \sim f^* K_X$, so $X$ is klt (but not terminal). The blow-up square
\[\begin{tikzcd}
E \arrow[r, hook] \arrow[d] & Y \arrow[d, "f"] \\
\{p\} \arrow[r, hook]              & X.
\end{tikzcd}\]
yields an exact sequence
$$\cdots \to KH_j (Y)\oplus KH_j(p) \to KH_j (E) \to KH_{j-1} (X) \to KH_{j-1} (Y)\oplus KH_{j-1} (p) \to \cdots$$
in homotopy $K$-theory, and using the smoothness of $Y$ we have
$$K_{-2}(X) \simeq KH_{-2}(X) \simeq KH_{-1}(E)\simeq \mathbb{Z}/2\mathbb{Z},$$
where the last isomorphism follows from Example \ref{exmp cayley}.
\end{example}

\begin{example}[{\bf $K_{-2} = 0$ for $cA_n$ threefolds}]\label{example-K2-cAn}
For a class of examples that are better behaved than our general results, we have the following:

\begin{proposition}\label{prop K minus 2 cAn}
    A threefold $X$ with isolated $cA_n$ singularities has $K_{-2}(X)=0$.
\end{proposition}

\begin{proof}
We already know by Example \ref{ex:cA1} that this is true for $cA_1$ singularities, and we prove the proposition inductively.  Using the well-known fact that for isolated singularities $K_{-i} (X)$ is an analytic invariant, completely determined by $K_{-i} (\widehat{\mathscr{O}}_{X,x})$ over all $x \in X_{\rm sing}$, we may assume $X$ has exactly one singular point $p$, around which the local analytic equation is $xy+g(z,u)=0$, where the lowest degree among the terms of $g$ is $\tau$.

We take a weighted blowup $X' \to X$ of weights $(1, \tau-1, 1,1)$, with exceptional divisor $E$. This is the first step of a feasible resolution of $X$, and $X'$ is known to have ``milder" $cA_n$ singularities; see \cite{chen}. The exceptional divisor $E$ is a surface defined by the equation
\[E = \big( xy + \prod_{i=1}^n (z - a_i u)^{\ell_i} = 0 \big) \subset \mathbb{P}(1, \tau-1, 1, 1)\]
where $\tau = \sum_{i=1}^n \ell_i$. 
We use the Mayer-Vietoris sequence in homotopy $K$-theory for the blow-up square 
        $$\cdots \to KH_{-1}(E) \to KH_{-2}(X) \to KH_{-2}(X') \oplus KH_{-2}(p) \to \cdots$$
Note that $E,X$ and $X'$ all have rational (hence Du Bois) singularities, so we can replace $KH$ by $K$ by Theorem \ref{thm rosie}. Note that $K_{-2}(p)=0$ since $p$ is a point.

By induction, we also know that 
$K_{-2}(X')=0$, since $X'$ has milder singularities in the sense of \cite{chen}, meaning of type $cA_m$ with $m < n$.
The result then follows if we show $K_{-1} (E) = 0$.

   To this end, we compute $K_{-1} (E)$ using Weibel's formula (\ref{eqn:Weibel-sequence}).      
       The singular points of $E$ consist of a point $P_y$ of type $\frac{1}{\tau-1}(1,1)$ in the $y\neq 0$ chart, and $n$ points $P_i$ of type $A_{l_i-1}$ in the $u\neq 0$ chart. Their local class groups are $\mathbb{Z}/(\tau-1)\mathbb{Z}$ and $\mathbb{Z}/l_i\mathbb{Z}$, respectively. 
        Thus 
        $$\bigoplus_{x \in E_{\text{sing}}} \operatorname{Cl}(\widehat{\mathcal{O}}_{E,x}) \cong \mathbb{Z}/(\tau-1)\mathbb{Z} \oplus \bigoplus_{i=1}^n \mathbb{Z}/l_i\mathbb{Z}.$$
        The open subset $U = E \setminus \{x=0\}$ is isomorphic to $\mathbb{A}^2$, so $\operatorname{Cl}(E)$ is generated by the $n$ irreducible components $D_i = \{x=0, z=a_i u\}$ of $E \setminus U$.\footnote{In fact, one can show that there are no relations among the $D_i$'s, and so $\operatorname{Cl}(E)\simeq \mathbb{Z}^n$. We will not need this result here.}
The restriction map 
        $$\operatorname{Cl}(E) \to \bigoplus_{x \in E_{\text{sing}}} \operatorname{Cl}(\widehat{\mathcal{O}}_{E,x})$$ 
        sends $D_i$ to $e_i + e_y$, where $e_i$ and $e_y$ are the canonical generators of $\mathbb{Z}/l_i\mathbb{Z}$ and $\mathbb{Z}/(\tau-1)\mathbb{Z}$. Since $\sum_{i=1}^n l_i=\tau$, simple algebra shows that the map is surjective, hence $K_{-1} (E) = 0$.
\end{proof}

Thus overall $cA_n$ threefolds have $K_{-3} (X) = K_{-2} (X) =  K_{-1} (X)_\Q = 0$, though $K_{-1} (X)$ itself is sometimes nonzero.

One can check that this method does not extend to $cA/r$ singularities, though we have not yet been able to construct such singularities 
for which $K_{-2} (X) \neq 0$. Finally, given the description of $K_{-2}(X)$ in \S\ref{scn:small-dim}, the proposition has an amusing elementary consequence 
for which we do not have an obvious direct proof:

\begin{corollary}
Let $f\colon Y \to X$ be a log resolution of a threefold $cA_n$ singularity, assumed to be an isomorphism away from the singularity.  If  $S_j$ and $C_k$ denote the surface and curve strata of the exceptional locus, then the natural restriction map
$$\bigoplus_j \Pic S_j \rightarrow \bigoplus_k \Pic C_k$$
is surjective.
\end{corollary}

\end{example}

\end{document}